\documentclass[12pt,a4paper]{article}
\usepackage{amsmath}
\usepackage{amsthm}
\usepackage{amsfonts}
\usepackage{amssymb}
\usepackage{amscd}
\numberwithin{equation}{section}
\newtheorem{theorem}{Theorem}[section]
\newtheorem{cor}[theorem]{Corollary}
\newtheorem{lemma}[theorem]{Lemma}
\newtheorem{prop}[theorem]{Proposition}
\theoremstyle{definition}
\newtheorem{defi}{Definition}[section]
\newtheorem{rem}[defi]{Remark}

\newtheorem*{thma}{Theorem V}
\newtheorem*{thmb}{Theorem HP}

\newcommand\Aa{{\mathcal A}}

\renewcommand\j{{\bf j}}

\newcommand\half{\frac{1}{2}}

\newcommand\ov{\overline}

\renewcommand\({\left(}
\renewcommand\){\right)}

\newcommand\rhat{\widehat\rho}

\newcommand\be{\beta}

\newcommand\w{\wedge}
\newcommand\g{\mathfrak g}
\newcommand\ga{\widehat{\mathfrak g}}
\newcommand\h{\mathfrak h}
\newcommand\ha{\widehat{\mathfrak h}}

\newcommand\n{\mathfrak n}
\newcommand\bb{\mathfrak b}
\newcommand\D{\Delta}
\renewcommand\l{\lambda}
\newcommand\Dp{\Delta^+}

\newcommand\Da{\widehat\Delta}
\newcommand\Pia{\widehat\Pi}
\newcommand\Dap{\widehat\Delta^+}
\newcommand\Wa{\widehat{W}}
\renewcommand\d{\delta}
\renewcommand\r{\mathfrak r}

\renewcommand\a{\alpha}
\renewcommand\aa{\mathfrak a}
\renewcommand\b{\bb}

\renewcommand\k{\mathfrak k}

\renewcommand\th{\theta}

\renewcommand\i{{\bf i}}

\newcommand\nat{\mathbb N}
\newcommand\ganz{\mathbb Z}
\renewcommand\j{\mathfrak j}
\newcommand\s{\sigma}
\renewcommand\L{\Lambda}

\renewcommand\aa{\mathfrak a}
\newcommand\la{\lambda}
\renewcommand\u{{\bf u}}

\newcommand\C{\mathbb C}
\newcommand\R{\mathbb R}
\newcommand\si{\sigma}

\newcommand\va{|0\rangle}

\renewcommand\ha{\widehat{\mathfrak h}}

\newcommand\bv{{\bf v}}

\newcommand\p{\mathfrak p}

\newcommand{\rank}{{\rm rank}}

\newcommand{\sa}{\mathfrak{a}}
\newcommand{\asdim}{\text{\rm asdim}}

\begin{document}
\title{Multiplets of representations, twisted  Dirac operators and 
Vogan's conjecture in affine setting}
\author{Victor~G. Kac\\ Pierluigi M\"oseneder Frajria\\ Paolo  Papi}

\maketitle
\setcounter{tocdepth}{3}
\begin{abstract} We extend classical results of Kostant et al. on 
multiplets of representations of finite dimensional
Lie algebras and on the cubic Dirac operator to the setting of affine 
Lie algebras and twisted affine cubic Dirac operator. We
prove in this setting an analogue of Vogan's conjecture on 
infinitesimal characters of Harish--Chandra modules in terms of Dirac
cohomology. For our calculations we use the machinery of Lie 
conformal and vertex algebras.
\end{abstract}
\tableofcontents
\section{Introduction} 
Let $\g$ be a finite dimensional reductive Lie algebra over $\C$ with 
a non degenerate symmetric
invariant bilinear form $(\  ,\ )$, and let $\aa$ be a reductive 
subalgebra of $\g$, such that
the restriction of the form  $(\  ,\ )$ to $\aa$ is non-degenerate. 
For future reference, we call $(\g,\aa)$ a reductive 
pair. Then
the adjoint representation of $\aa$ in $\g$ induces a homomorphism 
$\aa\to so(\p)$, where $\p$ is the orthogonal
complement of $\aa$ in $\g$. Hence, by restriction, we obtain a 
representation of the Lie algebra
$\aa$ in the spinor representation $F$ of $so(\p)$. Assume now that 
$\rank(\aa)=\rank(\g)$. Then 
$\dim(\p)$ is even, hence $F$ decomposes in a direct sum of two 
irreducible representations $F^+$
and $F^-$ of $so(\p)$. The following ``homogeneous Weyl character 
formula" was established in 
\cite{GKRS}, as an identity in the character ring of $\aa$: 
\begin{equation}\label{in1}
V(\l)\otimes F^+-V(\l)\otimes F^-
=\sum_{w\in W'}(-1)^{\ell(w)}U(w(\l+\rho_\g)-\rho_{\sa}),
\end{equation}
where $\rho_\g,\rho_\aa$ are half the sum of the positive roots of 
$\g$ and $\aa$ respectively, $V(\l)$ is an irreducible representation 
of $\g$ with highest weight  $\l$,
 $U(\mu)$ is an irreducible representation of $\aa$ with highest 
weight  $\mu$, and
$W'$ is the set of minimal length  representatives in the  Weyl 
group  of $\g$
from each right coset of the Weyl group  of $\aa$.  In Remark~\ref{nominimal} we adjust definitions as to make \eqref{in1} meaningful even when the representatives are not minimal.

The collection of 
$\aa$-modules
in the right-hand side of \eqref{in1} is called a multiplet. It is 
clear from
 \eqref{in1} that the Casimir operator of $\aa$ has equal eigenvalues 
on all modules from this
multiplet, and also that the signed sum of dimensions of these 
modules is $0$. In the special case
of $\aa=so_9$, important for $M$-theory, these kind of multiplets 
have been earlier discovered by
physicists, and formula 
\eqref{in1} for the embedding $so_9\subset F_4$ reproduced their 
observations.\par
In fact we observe (see Proposition \ref{qdimens}) that  any multiplet satisfies a stronger property, namely the 
signed sum of
{\it quantum dimensions} of modules of a multiplet is zero. Recall 
that the quantum dimension of an $\aa$-module $L(\L)$, where $\aa$ is
a semisimple subalgebra\footnote{the definition of quantum dimension
for an arbitrary reductive subalgebra $\aa$ is given in  \eqref{qdimred}.}, 
equals
\begin{equation}\label{qdimensione}
\dim_q L(\L)=\prod\limits_{\a\in 
(\D_\aa^+)^\vee}\frac{[(\L+\rho_\aa,\a)]_q}{[(\rho_\aa,\a)]_q},\end{equation} 
where $(\D_\aa^+)^\vee$ is the set of positive coroots of $\aa$ and 
$[n]_q=\frac{q^{n/2}-q^{-n/2}}{q^{1/2}-q^{-1/2}}$ (note that 
$\lim_{q\to 1}\dim_q L(\L)=\dim L(\L)$).\par In a subsequent paper
\cite{Kold}, Kostant showed that for a ``cubic" Dirac operator
${\not\!\partial}_{\g/\aa}$, with a cubic term  associated to the 
fundamental 3-form of $\g$, the kernel 
on $V(\l)\otimes F$ decomposes precisely in a direct sum of the 
$\aa$-modules of the multiplet:
\begin{equation}\label{in2}
Ker_{V(\l)\otimes F}{\not\!\partial}_{\g/\aa}=
\bigoplus_{w\in  W'}U(w(\l+\rho_\g)-\rho_{\sa}).
\end{equation}
\vskip5pt
Both decompositions \eqref{in1} and \eqref{in2} were extended to the 
setting of affine Lie algebras
by Landweber \cite{land}. The affine analogue of the ``cubic" Dirac 
operator, used in \cite{land},
was introduced much earlier by Kac and Todorov in their work 
\cite{KacT} on unitary representations
of Neveu-Schwarz and Ramond superalgebras.\par
Another interesting application of the  Dirac operator 
${\not\!\partial}_{\g/\aa}$ was
a conjecture of Vogan  expressing the infinitesimal character of an 
irreducible
$(\g,K)$-module $X$ of a real reductive group $G$ in terms of an 
$\aa$-type of $X$ appearing in 
the cohomology defined by ${\not\!\partial}_{\g/\aa}$ (the ``Dirac 
cohomology"), where $\aa$ is the
complexification of the Lie algebra of $K$ (cf. Theorem V in the 
Introduction). This conjecture has been proved in
\cite{HP} as a consequence of an algebraic  statement (cf. Theorem HP 
below) relative to a ``quadratic" Dirac
operator associated to a symmetric pair $(\g,\aa)$ (i.e., $\aa$ is 
the fixed point set of an involution of $\g$).
Subsequently Kostant
\cite{K3} observed that using his ``cubic" Dirac operator yields this 
algebraic statement for any reductive pair.\par
In the present paper we define  a general $\sigma$-twisted ``cubic" 
Dirac operator in the affine
setting, where $\sigma$ is a semisimple automorphism (not necessarily 
of finite order) of the
reductive Lie algebra $\g$, leaving invariant the subalgebra $\aa$, 
and extend all the results
mentioned above to this setting. As in \cite{KacT} we work with the 
``superization"
$\g[\xi]$ of $\g$, where $\xi$ is an odd indeterminate with 
$\xi^2=0$, in order to treat
 the Kac-Moody and Clifford affinizations simultaneously. Also, we 
use the very convenient for
calculations machinery of Lie conformal and vertex algebras, and 
their twisted representations.
\par
In more detail, let 
$$\widehat\g^{super}=\g[\xi][t,t^{-1}]\oplus\C\mathcal 
K\oplus\C\ov{\mathcal K}$$
be the ``superization" of the affine Lie algebra 
$\widehat\g=\g[t,t^{-1}]\oplus\C\mathcal K$, where
$\mathcal K,\ov{\mathcal K}$ are even central elements, and the 
remaining commutation relations are 
($a,b\in\g$):
\begin{align*}
[at^m,bt^n]&=[a,b]t^{m+n}+m\d_{m,-n}(a,b)\mathcal K,\\
[at^m,b\xi t^n]&=[a,b]\xi t^{m+n},\\
[a\xi t^m,b\xi t^n]&=\d_{m,-n-1}(a,b)\ov{\mathcal K}.
\end{align*}
Denote by $V^{k+g,1}$ the $\widehat\g^{super}$-module, induced from 
the 1-dimensional representation
of the subalgebra $\g[\xi][t]\oplus\C\mathcal K\oplus\C\ov{\mathcal 
K}$
given by
$$\g[\xi][t]\mapsto 0,\qquad\mathcal K\mapsto k+g,\qquad\ov{\mathcal 
K}\mapsto 1.$$
We assume here that the Casimir operator $Cas$ of $\g$, acting on 
$\g$, has a unique eigenvalue
 $2g$. (If $Cas$ has several eigenvalues, we take the tensor product 
of the corresponding modules,
and below we take the sum of the corresponding ``cubic" Dirac 
operators.) The module 
$V^{k+g,1}$ has a canonical structure of a vertex algebra with vacuum 
vector
$|0\rangle=1$, translation operator $T$, induced by $-\frac{d}{dt}$, 
and generating fields
$$\{a(z)=\sum_{j\in\ganz}(at^j)z^{-j-1},\,\ov 
a(z)=\sum_{j\in\ganz}(a\xi
t^j)z^{-j-1}\}_{a\in\g}.$$
Denote by $a$ and $\ov a$ the corresponding elements under the 
field-state correspondence:
$a=(at^{-1})|0\rangle,\,\ov a=(a\xi t^{-1})|0\rangle$. Let $\{x_i\}$ 
and $\{x^i\}$ be dual bases
of $\g$, i.e. $(x_i,x^j)=\d_{i,j}$. The following element of the 
vertex algebra
$V^{k+g,1}$ is the father of all affine $\sigma$-twisted ``cubic" 
Dirac operators (cf. \cite{KacT},
\cite{KacD}):
\begin{equation*}G_\g=\sum_i:x_i\overline  
x^i:+\tfrac{1}{3}\sum_{i,j}:\overline{[x_i,
x_j]}\overline  x^i\overline x^j:.
\end{equation*} 
This element satisfies the following key $\l$-bracket relation (cf. 
\cite{KacT},
\cite{KacD}):
\begin{equation}\label{in3}[G_\g{}_\la  
G_\g]=L_\g+\frac{\la^2}{2}(k+\frac{g}{3})\dim\g,
\end{equation}
where
\begin{equation}\label{primosugawara}L_\g=\sum_i:x_ix^i:+ (k+g)\sum_i :T(\ov x_i)\ov x^i:+\sum_{i,j}:\ov 
x^i[x_i,x_j]\ov x^j:.\end{equation}
\vskip5pt
The basic object of our considerations is the element $G_\g-G_\aa$, 
which is the father of all
 affine relative
$\sigma$-twisted ``cubic" Dirac operators, in the following sense. 
First, notice that the vertex
algebra
$V^{k+g,1}$ is isomorphic to the vertex algebra $V^k(\g)\otimes 
F(\g)$, the isomorphism being 
induced by the map $x\mapsto x-\frac{1}{2}\sum_i:\ov{[x,x_i]}\ov 
x^i:,\,\ov x\mapsto \ov x\
(x\in \g)$, where $V^k(\g)$ is the affine vertex algebra of level 
$k$, isomorphic to the subalgebra
of $V^{k,1}$ generated by the fields $\{a(z)\}_{a\in\g}$, and $F(\g)$ 
is the fermionic vertex
algebra, isomorphic to the subalgebra
of $V^{k,1}$ generated by the fields $\{\ov a(z)\}_{a\in\g}$.
Since the vertex algebra $V^k(\g)$ (resp. $F(\g)$) is the affine 
analogue of the universal enveloping
algebra $U(\g)$ (resp. of the Clifford algebra generated by $\ov 
\g=\g\xi$), any (twisted)
representation of the vertex algebra $V^{k+g,1}$ gives rise to a 
representation of the affine Lie
algebra $\widehat \g$ of the form $M\otimes F(\ov\g)$, where $M$ is a 
(twisted) $\widehat \g$-module of
level $k$ and $F(\ov\g)$ is the restriction of the (twisted) spinor 
representation of
 $\widehat{so(\g)}$ to $\widehat \g$. (Here by twisted $\widehat 
\g$-module we mean a
representation of the twisted affine Lie algebra).\par
Denote by $\tau$ the automorphism of the Lie superalgebra $\g[\xi]$, 
defined by being 
$\sigma$ on $\g$ and $-\sigma$ on $\ov\g=\g\xi$. This automorphism 
induces an automorphism of the
vertex algebra $V^{k+g,1}$, also denoted by $\tau$, such that 
$\tau(G_\g)=-G_\g$. It follows that 
in any twisted representation $M\otimes F(\ov\g)$ of $V^{k+g,1}$, the 
odd field corresponding to the
element $G_\g$ is of the form
$$Y^{M\otimes 
F(\ov\g)}(G_\g,z)=\sum_{n\in\ganz}G^{(\sigma)}_{\g,n}z^{-n-\frac{3}{2}}.$$

The operator $G^{(\sigma)}_{\g,0}$ (resp. 
$G^{(\sigma)}_{\g,0}-G^{(\sigma)}_{\aa,0}$) is called the
$\sigma$-twisted (resp. relative $\sigma$-twisted) affine Dirac 
operator.\par One shows that the operators
$G^{(\sigma)}_{\g,n},\,(n\in\ganz)$ and their brackets define a 
representation of the Ramond
superalgebra in $M\otimes F(\ov\g)$ (an easy way to do it is to use 
the $\l$-bracket relation 
\eqref{in3} and similar  $\l$-bracket relations for 
$[L_\g\phantom{.}_\l G_\g]$ and 
$[L_\g\phantom{.}_\l L_\g]$). In particular, one gets the formula for
$(G^{(\sigma)}_{\g,0})^2=\frac{1}{2}[G^{(\sigma)}_{\g,0}, 
G^{(\sigma)}_{\g,0}]$, similar to Kostant's
formula in the finite dimensional setup \cite{Kold} (going to the 
affine setup one should replace $Cas$ in
Kostant's formula by the Virasoro operator
 $L^{(\sigma)}_{\g,0}$).\par Moreover, the relative operators
$G^{(\sigma)}_{\g,\aa,n}=G^{(\sigma)}_{\g,n}-G^{(\sigma)}_{\aa,n}\ 
(n\in\ganz)$ and their brackets again define
a representation of the Ramond superalgebra in $M\otimes F(\ov\g)$, 
which intertwines  the
representation of $\widehat \aa$. This construction was used in 
\cite{KacT} to describe all unitary
discrete series representations of the Ramond (and Neveu-Schwarz) 
superalgebra.
Here we use this construction to obtain a twisted representation of 
$\widehat \aa$, and to derive a formula, similar to \eqref{in2} in 
the affine setting.
This generalizes the main result of \cite{land} from $\sigma=1$ to 
arbitrary $\sigma$. 
\par 
The proof of the decomposition \eqref{in1} in \cite{GKRS} works also 
in the affine setting 
(cf. \cite{land} for $\sigma=1$). As a corollary, we obtain that the 
signed sum of asymptotic
dimensions of $\widehat \aa$-modules from a  multiplet equals zero. 
Of course, all the $\widehat\aa$-modules 
are infinite dimensional, but one can use a substitute for dimension, 
called {\it asymptotic dimension} \cite{Kac} (cf.
\eqref{dimensioneasintotica}), which is a positive real number that 
has all the basic properties of dimension. Moreover, if $\aa$ is reductive,  the sum runs over an infinite set, but a suitable use of Theta functions makes the sum finite.\vskip5pt
Next we introduce a notion of Dirac cohomology  $H((G_{\g,\aa})_0,M)$ 
of a $\sigma$-twisted $\ga$-module $M$. It
turns out that it is not difficult to prove a non-vanishing result 
for this cohomology which is an affine analogue of 
Kostant's result
\cite{K3}. In light of  this, it is natural to speculate 
 in our affine setting on results in the spirit of the following 
conjecture of Vogan \cite{V}, proved by 
Huang and Pand{\v{z}}i{\'c} \cite{HP}.\par
 Let $G$ be a connected real reductive group with a maximal
compact subgroup
$K$ and let $\tilde { K}$ be the two-fold spin cover of $K$. Let 
$\g=\k\oplus\p$ be the Cartan decomposition of the
complexified Lie algebra $\g$ of $G$. 
\begin{thma}{\it Let $X$ be an irreducible $(\g,K)$-module. Let 
$\gamma$ be the highest 
weight of an irreducible   $\tilde
{K}$-module appearing  in the
 Dirac cohomology of $X$. Then the infinitesimal character of
$X$ is given by $\gamma + \rho_\k$, where $\rho_\k$ is the half sum 
of the positive roots of
$\k$.}\end{thma}
\vskip5pt\noindent For more details on this statement (in particular 
for an explanation of how $\gamma
+ \rho_\k$ defines an infinitesimal character) see the discussion in 
\cite[2.3]{HP}. \par
This theorem has been proved 
in  \cite{HP} as a consequence of a purely  algebraic statement, 
which is as follows.
Consider the embeddings $\k\to \g\to U(\g),\,\k\to so(\p)\to Cl(\p)$, 
and let 
 $\k_\D$ be the associated diagonal copy of $\k$ in $U(\g)\otimes 
Cl(\p)$. Let $Z(\g),\,Z(\k_\D)$ be the centers 
of $U(\g),\,U(\k_\D),$ let $\h,\, \mathfrak t$ be Cartan subalgebras 
of $\g,\k$ and $W,\,W_{\k}$ their respective Weyl groups.
Recall the isomorphisms $S(\h)^W\cong P(\h^*)^W,
\,S(\mathfrak t)^{W_{\k}}\cong P(\mathfrak t^*)^{W_{\k}}$ (here
$P(V)$ denotes the algebra of polynomial functions on the vector 
space $V$ and $S(V)$ the symmetric algebra on $V$).

\begin{thmb}{\it Let $\{z_i\}$ be an orthonormal basis
of $\p$ with respect to  the Killing form of $\g$  and let $D=\sum_i 
z_i\otimes z_i$ be the Dirac operator.
\begin{enumerate}
\item For any $z\in Z(\g)$ there is a unique $\zeta(z)\in Z(\k_\D)$ 
and an element $a\in U(\g)\otimes Cl(\p)$
such  that $z\otimes 1=\zeta(z)+ aD+Da$.
\item The map  $\zeta:Z(\g)\to Z(\k_\D)\cong Z(\k)$
 is an algebra homomorphism which makes the following diagram 
commutative:
\[
\begin{CD} 
Z(\g) & @>\zeta>> & Z(\k)\\ @VVV & & @VVV\\ 
S(\h)^W & @>Res>> &S(\mathfrak t)^{W_{ K}}
\end{CD}
\]
\end{enumerate} 
Here the vertical arrows are Harish-Chandra isomorphisms, whereas 
$Res$ is the 
restriction of polynomials on $\h^*$ to $\mathfrak t^*$.}
\end{thmb}

 Kostant observed in \cite{K3} that Theorem HP holds for any  
reductive pair $(\g,\aa)$ provided
one uses the cubic Dirac operator. Since this operator is in a 
precise sense a specialization of our
$G^{(\sigma)}_{\g,\aa,0}$ (see the discussion  in \cite[\S 
3.4]{KacD}), it is  natural to ask whether it is possible to
formulate a kind of affine analogue of Theorems V or HP. 
We obtain this analogue in the following setting.\par We replace 
the $(\g,{ K})$-module  $X$ by  a highest weight twisted 
$\widehat\g$-module
$M$. 
 Let $\g^{\ov 0},\,\aa^{\ov 0}$ denote the fixed
point  subalgebras of $\sigma$ in $\g,\,\aa$ respectively. 
Let $\h_0$ denote a Cartan subalgebra of $\g^{\ov 0}$. Fix a Cartan 
subalgebra $\h_\aa$ of
$\aa^{\ov 0}$ and let $\ha_\aa$ be the corresponding Cartan 
subalgebra of $\widehat\aa$.
Let $C_\g$ denote the  Tits cone  of $\widehat\g$ (cf. \eqref{tits}) 
and let $\rhat_\sigma,\,
\rhat_{\aa\sigma}$ be as in \eqref{rhoaff}. Let $\Wa$ be the Weyl 
group of $\ga$. The following
result is our affine analog of Vogan's conjecture.

 \begin{theorem}\label{primovogan} Assume that the centralizer
 $C(\h_\aa)$ of $\h_\aa$ in $\g^{\bar0}$ equals $\h_0$.
Fix
$\L\in
\widehat
\h^*_0$ such that
$\L+\rhat_\sigma\in C_\g$ and let
$M$ be a highest weight module for $\widehat\g$ with highest weight 
$\L$. Let $f$ be a
holomorphic
$\Wa$-invariant function on $C_\g$. Suppose that a twisted highest 
weight $\widehat \aa$-module  of
highest weight
$\mu$
 occurs in the Dirac cohomology $H((G_{\g,\sa})_0, M)$. Then 
\begin{equation*}f_{|{\ha}_\sa^*}(\mu+\rhat_{\aa\si})= 
f(\L+\widehat\rho_{\si}).\end{equation*}
\end{theorem}

\par
\vskip10pt
The content of the paper is as follows. In Sections 2 and 3 we 
introduce the basic  material
on vertex and Lie conformal algebras, and on their twisted 
representations, respectively.
Here we discuss in some detail the examples of affine, fermionic and 
super affine vertex algebras,
and their twisted representations.\par
In Section 4 we introduce (for any reductive pair) the relative 
affine Dirac operators, in the framework of twisted
representations of super affine vertex algebras. We compute their 
squares and the values of
the squares on highest weight vectors (formula \eqref{gzeronuovo} and 
Propositions \ref{azioneDeG} and
\ref{ker}). Also, we obtain  formula \eqref{masterrho}, as a 
corollary of these computations. It is
shown  in Section 6 that in the special case when $\sigma$ is  a 
finite order automorphism, formula
\eqref{masterrho} turns into the ``very strange formula" 
\cite[(13.15.4)]{Kac}.\par
In Section 5  we decompose (under the assumption $\rank\,\aa^{\ov 
0}=\rank\,\g^{\ov 0}$) the kernel of a relative
affine Dirac operator in the multiplets  (Theorem \ref{multiplets}), 
compute the (common) eigenvalue of the
affine Casimir operator on representations of each multiplet 
(Corollary \ref{casimiraref}) and show that 
the signed sum of 
 asymptotic dimensions of representations of a  multiplet is zero 
(Proposition \ref{asdim}).\par
In Section 7, in the general setting of reductive pairs, we obtain a 
non-vanishing result for affine Dirac
cohomology, similar to Kostant's
\cite{K3} in the finite dimensional setting.\par
In by far the longest Section 8 we prove  our affine analogue of 
Vogan's conjecture: the main result
is Theorem \ref{Vogan1}, which is a technically more precise 
formulation of Theorem \ref{primovogan} above. Though the flavor of
Huang-Pand{\v{z}}i{\'c}'s proof remains (notably in exploiting the  
exactness of suitable Koszul complexes), we have
to overcome several  difficulties which are  due to the completion, 
which has to be introduced in order to have a
large holomorphic center, constructed in \cite{Kacpnas}, and the 
corresponding  Harish-Chandra type homomorphism.\par
In Section 9 we give proofs, omitted in previous sections, of various 
technical results.

\section{Basic definitions and examples}\label{basic}
For background on vertex algebras, conformal Lie algebras and twisted 
vertex 
algebras see
\cite{KacV}, \cite{KacD}, \cite{KW}.\par A vector superspace  is a 
$\ganz/2\ganz=\{\ov 0,\ov 1\}$-graded vector space $V=V_{\ov 0}\oplus 
V_{\ov  1}$. If 
$\a\in\ganz/2\ganz$, then set $p(v)=\a$ if $v\in V_\a$ and call 
$p(v)$ the  parity of $v$. We also set
$p(v,w)=(-1)^{p(v)p(w)}$. Recall that an 
$End(V)$-valued quantum field  is a  series 
$$ a(z)=\sum_{n\in\ganz}a_{(n)}z^{-n-1}
$$ where $a_{(n)}\in End(V)$ all have the same parity (called the 
parity of $a(z)$)  and, for all $v\in V$,
$a_{(n)}v=0$ for
$n>>0$.
\begin{defi}
 A \emph{vertex algebra} is triple $(V,|0\rangle, Y)$, where $V$ is a 
vector superspace, $\va$ is an
even vector in $V$, and $Y: a\mapsto  
Y(a,z)=\sum\limits_{n\in\ganz}a_{(n)}z^{-n-1}$ is a parity 
preserving linear map from $V$ to the space of $End(V)$-valued quantum 
fields. These data satisfy  the
following axioms, where $Ta=a_{(-2)}\va$: 

\item{i)} $T\va=0$, $Y(a,z)\va_{|z=0}=a$,
\item{ii)} $[T,Y(a,z)]=\partial_zY(a,z)$,
\item{iii)} $(z-w)^N[Y(a,z),Y(b,w)]=0$ for some $N\in \ganz_+$.
\end{defi}

As a consequence of the axioms one deduces that, if $a,b\in V$, then, 
in $End(V)$,
\begin{equation}\label{bracket} [a_{(n)},b_{(m)}]=
\sum_{j\in\ganz_+}\binom{n}{j}(a_{(j)}b)_{(n+m-j)},\quad n,m\in\ganz.
\end{equation}

In a vertex algebra one defines a bilinear product $:\cdot:$, called 
the normal  order 
product, by 
$:ab:=a_{(-1)}b$. Letting  
$Y^+(a,z)=\sum_{n<0}a_{(n)}z^{-n-1}$ and 
$Y^-(a,z)=\sum_{n\ge0}a_{(n)}z^{-n-1}$, one defines
the normal order product of quantum fields by 
$$ :Y(a,z)Y(b,z):=Y^+(a,z)Y(b,z)+p(a,b)Y(b,z)Y^-(a,z).
$$
Then
$$ :Y(a,z)Y(b,z):=Y(:ab:,z).
$$ The normal order product is, in general, neither commutative nor
 associative, but the following ``quasi-commutativity" and
 ``quasi-associativity" relations hold:
\begin{equation}\label{qc} :ab:=p(a,b):ba:+\int_{-T}^0[a_\l b]d\l.
\end{equation}
\begin{equation}\label{qa}::ab:c:=:a:bc::+:(\int_{0}^T  d\l\,a)[b_\l 
c]:+p(a,b):(\int_{0}^T d\l\,b)[a_\l 
c]:.\end{equation}
where
\begin{equation}\label{lambdaprodotto}
[a_\l b]=\sum_{n\geq 0}\frac{\l^n}{n!}a_{(n)}b.
\end{equation}\vskip10pt\begin{defi}A {\it Lie conformal 
superalgebra}  is a $\ganz\!/2\ganz$-graded
$\C[T]$-module $R=R_{\overline{0}} \oplus  R_{\overline{1}}$, endowed 
with a parity preserving
$\C$-bilinear map $R \otimes  R \to \C [\lambda]\otimes R$, denoted 
by $[a_{\lambda}b]$, such that the 
following axioms hold:
\vspace*{-1ex}\begin{align*}%
\text{(sesquilinearity)\quad } & [T a_{\lambda}b] = -\lambda 
[a_{\lambda} b], \ %
T[a_{\lambda}b] = 
[Ta_{\lambda}b] + [a_{\lambda}Tb],\\
\text{(skewsymmetry)\quad  } &  [b_{\lambda}a] =-p(a,b)          
[a_{-\lambda -T}b],\\
\text{(Jacobi identity)\quad } & [a_{\lambda}[b_{\mu}c]] -
    p(a,b) [b_{\mu} [a_{\l}c]] =
    [[a_{\lambda}b]_{\lambda + \mu}c].
  \end{align*}

\end{defi}

A vertex algebra can be endowed with the structure of a Lie 
conformal  superalgebra by introducing the $\l$-product via 
\eqref{lambdaprodotto}. Moreover 
$[\cdot_\l\cdot]$  and $:\cdot :$  are related
by the {\it non-commutative Wick formula} :
\begin{equation}\label{Wick} [a_\l:bc:]=\, :[a_\l b]c:+p(a,b):b[a_\l 
c]:+\int_0^\l[[a_\l b]_\mu c]d\mu.
\end{equation} Combining Wick formula with skewsymmetry we get the 
``right non-commu\-tative  Wick
formula":
\begin{equation}\label{RWick} [:ab:_\l c]= :(e^{T\partial_\l}a)[b_\l 
c]:\!+p(a,b):(e^{T\partial_\l}b)[a_\l  c]:\!
+p(a,b)\int_0^\l[b_\mu [a_{\l-\mu}c]]d\mu.
\end{equation}

Given a Lie conformal superalgebra $R$  one can construct its 
universal  enveloping vertex algebra 
$V(R)$. This vertex algebra is  characterized by the following 
properties:
\begin{enumerate}
\item There is an embedding $R\to V(R)$ of Lie conformal 
superalgebras.
\item Given an ordered  basis $\{a_i\}$ of $R$, the monomials 
$:a_{i_1}a_{i_2}\cdots a_{i_n}:$ with $i_j\le i_{j+1}$ and 
$i_j<i_{j+1}$ if 
$p(a_{i_j})=\ov 1$ form a basis of $V(R)$.
\end{enumerate} Here and further, the normal ordered product of more 
than two fields is defined from
right to left, as usual.\par\vskip5pt We will be using the following 
three examples of Lie
conformal superalgebras and their  universal enveloping  vertex 
algebras.

\subsection{The affine vertex algebra}

Given a reductive finite dimensional complex Lie algebra $\g$ endowed 
with a  nondegenerate invariant
bilinear form $(\cdot,\cdot)$, one defines the  current Lie  
conformal algebra $Cur(\g)$ as 
$$ Cur(\g)=(\C[T]\otimes \g)+\C K
$$ with $T(K)=0$ and the $\l$-bracket defined for $a,b\in1\otimes 
\g$  by
$$ [a_\l b]=[a,b]+\l(a,b)K,\quad [a_\l K]=[K_\l K]=0.
$$ Let $V(\g)$ be its universal enveloping vertex algebra. Given 
$a,b\in \g$ then  it follows from
\eqref{bracket} that, in $End(V(\g))$, 
$$ [a_{(n)},b_{(m)}]=[a,b]+\d_{n,m}n(a,b)K,\quad m,n\in\ganz.
$$

The vertex algebra 
$$ V^k(\g)=V(\g)/:(K-k\va)V(\g):
$$ is called the {\it level $k$ universal affine vertex algebra}.

\subsection{The fermionic vertex algebra} Given a  vector superspace $A$  
with a nondegenerate bilinear
form $(\cdot,\cdot)$  such that 
$(a,b)=(-1)^{p(a)}(b,a)$, one can construct the  Clifford  Lie 
conformal algebra as
$$R^{Cl}(A)=(\C[T]\otimes A)\oplus \C K'$$ with $T(K')=0$ and the 
$\l$-bracket defined by 
$$ [a_\l b]=(a,b)K',\quad [a_\l K']=[K'_\l K']=0.
$$ Let $V(A)$ be its universal enveloping vertex algebra. Given 
$a,b\in A$, it  follows from
\eqref{bracket} that, in $End(V(A))$, 
$$ [a_{(n)},b_{(m)}]=\d_{n+m,-1}(a,b)K',\quad m,n\in\ganz.
$$

The vertex algebra 
$$ F(A)=V(A)/:(K'-\va)V(A):
$$ is called the {\it fermionic vertex algebra}.

\subsection{The super affine vertex algebra} Let $\g$ be a reductive 
complex  
finite dimensional Lie 
algebra endowed with a  non-degenerate symmetric bilinear invariant 
form $(\cdot,\cdot)$.
Regard $\g$ as an even superspace and let $\overline\g$ be $\g$ 
viewed as an odd  superspace.
Consider the Lie conformal superalgebra
$R^{super}=(\C[T]\otimes\g)\oplus(\C[T]\otimes\overline \g)\oplus\C 
\mathcal K\oplus\C 
\overline{\mathcal K}$  with  $T(\mathcal K)=T(\overline{\mathcal  
K})=0$, 
 $\mathcal K,\,\overline{\mathcal  K}$ being even central elements, 
the 
$\lambda$-brackets for $a,b\in 1\otimes \g$ being
\begin{equation}\label{prod} [a_\lambda b]=[a,b]+\lambda 
(a,b)\mathcal K,\quad [a_\lambda \overline
b]=[\overline a_\lambda b]=\overline{[a,b]},\quad [\overline 
a_\lambda \overline
b]=(a,b)\overline{\mathcal K}.
\end{equation}
 Denote by  $V(R^{super})$  the corresponding  universal enveloping 
vertex algebra, by 
 $V^{1}(R^{super})$  its quotient modulo   the ideal generated by  
$\overline{\mathcal K}-\va$ and
by   
$V^{k,1}(R^{super})$  the  quotient modulo the ideal generated by 
$\mathcal K-k\va$ and
$\overline{\mathcal K}-\va$.

\vskip10pt We conclude this section by showing how our three examples 
are related. First of all we
recall that a tensor product of vertex algebras $V_1$ and $V_2$  is  
the vertex algebra $V_1\otimes
V_2$ with vacuum vector  $\va_1\otimes \va_2$ and  state-field 
correspondence defined by
$Y(a\otimes b,z)=Y(a,z)\otimes Y(b,z)$.

Assume that $\g$ is semisimple or abelian.  Let $Cas$ be  the Casimir 
operator of $\g$ with respect to
$(\cdot,\cdot)$. We can and do  assume that the form is chosen so 
that $Cas=2\,g\,I_{\g}$ when acting
on $\g$, where $g$ is a positive real number.
\begin{prop}\label{isomorfi}  Let   $\{x_i\}$ be an orthonormal basis 
of $\g$.  For $x\in\g$ set  
$$\widetilde x= x-
\frac{1}{2}\sum\limits_i :\overline{[x,x_i]}\overline x_i:,\ 
\widetilde K=\mathcal K-g |0\rangle.$$ The map $x\mapsto \widetilde  x,\, 
K\mapsto \widetilde K$, $\ov
y\mapsto \ov y$ defines a Lie conformal superalgebra homomorphism 
$ Cur(\g)\oplus  R^{Cl}(\overline
\g)\to V^1(R^{super})$, which induces an isomorphism of vertex 
algebras 
$V^k(\g)\otimes F(\overline \g)\cong V^{k+g,1}(R^{super})$.
\end{prop}
\begin{proof} We first show that for  $a, b\in \g$ we have (in 
$V^1(R^{super})$)
\begin{equation}\label{due} [a_\l\widetilde 
b]=\widetilde{[a,b]}+\l(\mathcal K-g |0\rangle)(a,b)=
[\widetilde a_\l b].
\end{equation} 
The first equality in \eqref{due} follows from the Wick formula 
\eqref{Wick} and the Jacobi identity:
\begin{align*} [a_\l\widetilde b]&=[a_\l  
b]-\tfrac{1}{2}\sum_i[a_\l:\overline{[b,x_i]}\overline x_i:]\\&=[a_\l
b]-
\tfrac{1}{2}\sum_i(:[a_\l\overline{[b,x_i]}]\overline  
x_i:+:\overline{[b,x_i]}[a_\l\overline x_i]:+\int_0^\l[[a_\l 
\overline{[b,x_i]}]_\mu \overline x_i]d \mu)\\&=[a_\l b]-
\tfrac{1}{2}\sum_i(:\overline{[a,[b,x_i]]}\overline 
x_i:+:\overline{[b,x_i]} \ \overline{[a,x_i]}:+\l([a,[b,x_i]],x_i))\\&=[a_\l 
b]-
\tfrac{1}{2}\sum_i(:\overline{[a,[b,x_i]]}\overline  
x_i:+:\overline{[b,[x_i,a]]}\overline x_i:)-\la
g(a,b) |0\rangle\\&=[a_\l  
b]-\tfrac{1}{2}\sum_i:\overline{[[a,b],x_i]}\overline x_i:-\la  
g(a,b) |0\rangle.
\end{align*} Recalling that $[a_\l b]=[a,b]+\l(a,b)\mathcal K$ we 
have 
\eqref{due}.  By skewsymmetry of the $\la$-bracket, we readily obtain 
the second equality in \eqref{due}.

Next we prove that 
\begin{equation}\label{tre} [\overline a_\l 
\widetilde b]=0.
\end{equation} Just compute, using the Wick formula 
\eqref{Wick}:\begin{align*}[\overline a_\l \widetilde b]&=[\overline 
a_\l b]-
\tfrac{1}{2}\sum_i[\overline a_\l:\overline{[b,x_i]}\overline  
x_i:]\\&=\overline{[a,b]}-\tfrac{1}{2}(\sum_i
:[\overline a_\l 
\overline{[b,x_i]}]\overline x_i:-\sum_i:\overline{[b,x_i]}[\overline 
a_\l 
\overline 
x_i]:)\\&=\overline{[a,b]}-\tfrac{1}{2}(\sum_i(a,[b,x_i])\overline 
x_i-\sum_i(a,x_i)\overline{[b,x_i]})\\&=\overline{[a,b]}-
\tfrac{1}{2}\overline{[a,b]}+\tfrac{1}{2}\overline{[b,a]}=0.\end{align*} 

Finally, using \eqref{tre} we get
$[\widetilde  a_\l\widetilde b]=[a_\l \tilde b]$, and,  using 
\eqref{due}, we  find
\begin{equation}\label{atbt}[\widetilde a_\l\widetilde  
b]=\widetilde{[a,b]}+\l (\mathcal
K-g)(a,b)=\widetilde{[a,b]}+\l \widetilde K(a,b) =\widetilde{[a_\l 
b]}.\end{equation} 
This proves the first
part of the statement. 

Clearly the ideal generated by $K-k\va$ gets mapped to the ideal of 
$V^1(R^{super})$ generated by
$\mathcal K-(k+g)\va$, so our map factors to a map from 
$V^k(\g)\otimes F(\ov \g)$ to
$V^{k+g,1}(R^{super})$. We now show that this map is an isomorphism.  
For this it suffices to show that
it maps a basis of $V^k(\g)\otimes F(\ov \g)$ to a basis of
$V^{k+g,1}(R^{super})$. A basis of $V^k(\g)\otimes F(\ov \g)$ is 
given by vectors
$$ :T^{i_1}(x_{j_1})\cdots T^{i_h}(x_{j_h}):\otimes :T^{r_1}(\ov 
x_{s_1})\cdots T^{r_t}(\ov x_{s_t}):,$$
with $j_1\leq\ldots\leq j_k,\,s_1<\ldots< s_l$. These map to 
$$ :T^{i_1}(\tilde x_{j_1})\cdots T^{i_h}(\tilde x_{j_h})T^{r_1}(\ov 
x_{s_1})\cdots T^{r_t}(\ov x_{s_t}):. $$ 
Define a filtration of $V^{k+g,1}(R^{super})$ (as a vector space) 
by setting 
$$ V^{k+g,1}(R^{super})_m=span\{:T^{i_1}(x_{j_1})\cdots 
T^{i_h}(x_{j_h})T^{r_1}(\ov x_{s_1})\cdots
T^{r_t}(\ov x_{s_t}):\mid h\le m\}.
$$ Then, since $[\tilde x_\l \ov y]=0$, we see that
\begin{align*}
        :T^{i_1}(\tilde x_{j_1})&\cdots T^{i_h}(\tilde x_{j_h})T^{r_1}(\ov 
x_{s_1})\cdots T^{r_t}(\ov
x_{s_t}):\\&=:T^{i_1}(x_{j_1})\cdots T^{i_h}(x_{j_h})T^{r_1}(\ov 
x_{s_1})\cdots T^{r_t}(\ov x_{s_t}):+ a,
\end{align*} with $a\in V^{k+g,1}(R^{super})_{h-1}$. Since the 
vectors 
$$ :T^{i_1}(x_{j_1})\cdots T^{i_h}(x_{j_h})T^{r_1}(\ov x_{s_1})\cdots 
T^{r_t}(\ov x_{s_t}):
$$ form a basis of $V^{k+g,1}(R^{super})$  we are done.

\end{proof}

\section{Representations of vertex algebras} A {\it field module over 
a Lie conformal superalgebra
$R$}  is a vector  superspace $M$  endowed with a linear map $Y^M$ 
from
$R$ to the superspace of $End(M)$-valued quantum fields,
$$ a\mapsto Y^M(a,z)=\sum_{n\in\ganz}a^M_{(n)}z^{-n-1}
$$ such that for all $a,b \in R$, $m,n \in \ganz$ one has:
\begin{align}
&[a^M_{(n)},b^M_{(m)}]=\sum_{j\in\ganz_+}\binom{n}{j}(a_{(j)}b)^M_{(n+m- 
j)},\label{rep1}\\
&Y^M(T(a),z)=\partial_z Y^M(a,z)\label{rep2}.
\end{align} Also recall  that a representation of a vertex algebra 
$V$ in a vector superspace 
$M$ is a linear map 
$Y^M:V\to End(M)[[z,z^{-1}]]$ as above, such that  
$Y^M(|0\rangle,z)=I_M$, 
$ :Y^M(a,z)Y^M(b,z): = Y^M(:ab:,z)
$ and \eqref{rep1},  \eqref{rep2} hold.  Notice that the vertex 
algebra $V$ itself is a representation of
$V$.
\vskip5pt

More generally, one has the notion of a twisted field module over a 
Lie  conformal superalgebra and of
a twisted representation of a vertex algebra.  Let $R$ be a Lie 
conformal algebra and $\si$ a
semisimple  automorphism of $R$. Assume (for simplicity) that the 
eigenvalues of $\si$ are  of modulus
one. Then $R$ decomposes in a direct sum of $\C[T]$-submodules
$R^{\ov\mu}=\{a\in R\mid \si(a)=e^{2\pi i\ov\mu}a\}$, 
$\ov\mu\in\R/\ganz$.  A $\si$-twisted field module
over $R$ is a vector  superspace 
$M$  endowed with a linear map $Y^M:a\mapsto Y^M(a,z)$, where  
$Y^M(a,z)$ with $a\in R^{\ov\mu}$ is an 
$End(M)$-valued $\si$-twisted quantum field  

$$ Y^M(a,z)=\sum_{n\in \ov\mu}a^M_{(n)}z^{-n-1},\,\, a_{(n)}^Mv=0 \, 
\text { for } n>>0,
$$ such that \eqref{rep1} and \eqref{rep2} hold. Here and further 
$\ov\mu$  stands for a coset of
$\R/\ganz$.

 Let $V$ be a vertex algebra and $\si$ a semisimple automorphism of 
$V$ with modulus one
eigenvalues.   If $M$ is a $\si$-twisted field module of $V$, viewed 
as a Lie conformal  superalgebra,
and $a\in V^{\ov\mu}$, choose $\mu\in\R$ in the coset 
$\ov\mu$ and set
$$Y^M_+(a,z)=\sum_{n<\mu}a^M_{(n)}z^{-n-1},\, Y^M_- 
(a,z)=\sum_{n\ge\mu}a^M_{(n)}z^{-n-1}.$$  Define
the normal ordered product  of $\si$-twisted fields as
$$:Y^M(a,z)Y^M(b,z):=Y^M_+(a,z)Y^M(b,z)+p(a,b)Y^M(b,z)Y_-^M(a,z),$$
(it depends on the choice of $\mu$ in $\ov \mu$).\par A 
$\sigma$-twisted
representation of a vertex algebra $V$ is a $\si$-twisted  field 
module $M$ of $V$ (viewed as a Lie
conformal superalgebra), such that
\begin{equation}\label{vac}Y^M(|0\rangle,z)=I_M,
\end{equation}
\begin{equation}\label{tensore} :Y^M(a,z)Y^M(b,z):\ =
\sum_{j\in\ganz_+}\binom{\mu}{j}Y^M(a_{(j-1)}b,z)z^{-j},\quad a\in 
V^{\ov\mu}.
\end{equation}

If $\si$ is an automorphism of a Lie conformal algebra $R$ then we 
can extend $\si$ to an
automorphism, still denoted by $\si$, of $V(R)$  by  
$$
\si(:a_{i_1}\cdots a_{i_k}:)=:\si(a_{i_1})\cdots\si(a_{i_k}):.
$$ The following lemma is a restatement of Proposition 1.1 of 
\cite{KW}, which  provides a handy way to
construct $\si$-twisted representations of 
$V(R)$ from $\si$-twisted field modules over $R$.

\begin{lemma}\label{prop1.1}Any $\si$-twisted field module over a 
Lie  conformal algebra $R$ extends
uniquely to a $\si$-twisted representation over 
$V(R)$, using \eqref{vac} and \eqref{tensore}.
\end{lemma}

We now apply Lemma \ref{prop1.1} to our examples of vertex  algebras. 

\subsection{Twisted representations of the affine vertex 
algebra}\label{twisted affine}

Let $\sigma$ be a semisimple automorphism of $\g$ with modulus 1 
eigenvalues that keeps the bilinear  form invariant. Then $\si$ can be
viewed as  an automorphism of $Cur(\g)$ by setting $\si(K)=K$ and 
letting $\si$ and
$T$  commute. It follows that  we can extend $\si$ to an automorphism 
of $V(\g)$  that clearly stabilizes
$:(K-k\va)V(\g):$. We obtain therefore an automorphism  of 
$V^k(\g)$.\par 
Let
$L'(\g,\si)=\sum_{j\in\R}(t^j\otimes\g^{\ov j})\oplus \C K$, where
$\g^{\ov j}=\{x\in\g\mid
\si(x)=e^{2\pi i\ov j}x \}$, and,  as before,
$\ov j\in\R/\ganz$ denotes the coset containing $j$. This is a Lie 
algebra  with bracket defined by
$$ [t^m\otimes  a,t^n\otimes b ]=t^{m+n}\otimes[a,b] + 
\d_{m,n}m(a,b)K,\quad m,n\in\R,
$$ $K$ being a central element. We say  that a  $L'(\g,\si)$-module 
$M$ is {\it restricted} if, for any $v\in
M$, 
$(t^j\otimes a)(v)=0$ for $j>>0$. We say that $M$ is a representation 
of level 
$k$ if $Kv=kv$ for all $v\in M$.\par If $(\pi,M)$ is a restricted 
$L'(\g,\si)$-module of level $k$ and $a\in \g^{\ov j}$, then define 
$Y^M(a,z)=\sum_{n\in \ov j}\pi(t^n\otimes a)z^{-n-1}$ and 
$Y^M(K,z)=kI_M$.  Clearly these  fields
satisfy \eqref{rep1} and \eqref{rep2}.  Applying  Lemma \ref{prop1.1} 
we obtain a $\si$-twisted
representation of 
$V^k(\g)$ on $M$.
\par A particular example of a restricted module is given by a 
highest weight module.  Let $\h_0$ be a
Cartan subalgebra of $\g^{\ov 0}$ and $\h'=\h_0+\C K$. If 
$\mu\in (\h')^*$, we set $\ov\mu=\mu_{|\h_0}$. Denote by  $\D_0$  the 
set of roots  for the
pair
$(\g^{\ov 0},\h_0)$ and fix a subset of positive roots $\D_0^+$ in $\D_0$. 
For $\a\in\D_0$, let $(\g^{\ov 0})_\a$ denote the corresponding root space and set 
$\mathfrak n=\sum_{\a\in\D_0^+}(\g^{\ov 0})_\a,\,\mathfrak 
n'=\mathfrak  n+\sum_{j>0}t^j\otimes\g^{\ov
j}$. Fix $\L\in (\h')^*$ and set $k=\L(K)$. A 
$L'(\g,\si)$-module $M$ is called a highest weight module with 
highest weight 
$\L$ if there is a nonzero vector $v_\L\in M$ such that 
\begin{equation}\label{highest}\mathfrak n'(v_\L)=0,\,\ 
hv_\L=\L(h)v_\L \text{  for $h\in\h'$, }\
U(L'(\g,\si))v_\L=M.\end{equation}If $\mu\in\h^*_0$, we let 
$h_\mu$ be the unique element of $\h_0$ such that $(h,h_\mu)=\mu(h)$. 
Let 
$\D_{\ov j}$ be the set of $\h_0$-weights of $\g^{\ov j}$. Set 
\begin{equation}\label{rhos}
\rho_0=\half\sum_{\a\in\Dp_0}\a,\quad\rho_{\ov j}=
\half\sum_{\a\in\D_{\ov j}}(\dim\g^{\ov j}_\a)\a\quad \text{if $\ov 
j\ne\ov 0$},\quad\rho_\si=
\sum_{0\le j\le\frac{1}{2}}(1- 2j)\rho_{\ov j}.
\end{equation} Choose an orthonormal basis $\{x_i\}$  of $\g$ and set
 
\begin{align}&\label{sugaw}L^\g=\half\sum_i:x_ix_i:\in
V^k(\g),\\\label{zg}&z(\g,\si)=\half\sum_{0\le  j< 
1}\frac{j(1-j)}{2}\dim\g^{\ov j}.\end{align}
\begin{lemma}\label{xixi} If $M$ is a highest weight module over 
$L'(\g,\si)$  with highest weight $\L$,
then
\begin{align}\label{azioneasinistra}
(L^\g)^M_{(1)}(v_\L)&=\half(\ov\L+2\rho_\si,\ov\L)v_\L+kz(\g,\si)v_\L.
\end{align} 
\end{lemma}
\begin{proof} If 
$\{y_i\}$ is any basis of $\g$ and $\{y^i\}$ is its dual basis, then, 
clearly, 
$2L^\g=\sum_i:y_iy^i:.$ We can and do choose $\{y_i\}$ so that 
$y_i\in\g^{\ov s_i}$, for some $ \ov 
s_i\in\R/\ganz$. By
\eqref{tensore} we have
\begin{align}\notag
\left(\sum_i  (:y_{i}y^{i}:)\right)^M_{(1)}&=\sum_i\left(\sum_{n<s_i} 
(y_{i})^M_{(n)} 
(y^{i})^M_{(-n)}+\sum_{n\geq s_i} (y^{i})^M_{(-n)} 
(y_{i})^M_{(n)}\right)\\&- 
\sum_{r\in\ganz_+}\binom{s_i}{r+1}((y_{i})_{(r)}(y^{i}))^M_{(- 
r)}.\label{Sugawaraexpansion}
\end{align}
We choose $s_i\in [0,1)$, 
thus
$$
\left(\sum_i (:y_{i}y^{i}:)\right)^M_{(1)}(v_\L)=\sum_i((y^{i})^M_{(- 
s_i)} (y_{i})^M_{(s_i)}-
s_i[y_{i},y^{i}]^M_{(0)}-k\binom{s_i}{2})(v_\L)$$as  in  
\cite[(1.15)]{KW}.
Write\begin{align*}\sum_i((y^{i})^M_{(-s_i)}  (y_{i})^M_{(s_i)}- 
s_i&[y_{i},y^{i}]^M_{(0)}-
k\binom{s_i}{2})(v_\L)=\sum_{i:s_i=0}(y^{i})^M_{(0)}(y_{i})^M_{(0)}(v_\L)\\&-
\sum_{i:s_i>0}(s_i[y_{i},y^{i}]^M_{(0)}+k\binom{s_i}{2})(v_\L).\end{align*}\par 
Choosing an orthonormal
basis $\{h_i\}$ of $\h_0$ and writing 
$\sum\limits_{i: s_i=0}y^iy_i=2h_{\rho_0}+\sum_i 
h_i^2+2\sum_{\a\in\Dp_0}X_{-\a}X_\a$, we  find that
\begin{align*}\left(\sum_i 
(:y_{i}y^{i}:)^M\right)_{(1)}(v_\L)&=(\ov\L+2\rho_0,\ov\L)v_\L+k(\sum_{0<j<1}
\frac{j(1-j)}{2}\dim\g^j)v_\L\\ &-
\sum_{i:s_i>0}s_i[y_{i},y^{i}]^M_{(0)}(v_{\L}).\end{align*}\par In 
order to  evaluate
$\sum_{i:s_i>0}s_i[y_{i},y^{i}]^M_{(0)}(v_{\L})$, we observe  that
$$\sum_{i:s_i=s}[y_i,y^i]=\sum_{i:s_i=1-s}[y^i,y_i].$$ This relation 
is easily  derived by exchanging the roles
of $y_i$ and $y^i$.  
Hence\begin{align*}\sum_{i:s_i>0}s_i[y_{i},y^{i}]^M_{(0)}(v_{\L})&=
\sum_{i: \frac{1}{2}>s_i>0}s_i[y_{i},y^{i}]^M_{(0)}(v_{\L})+\sum_{i: 
1>s_i>\frac{1}{2}}s_i[y_{i},y
^{i}]^M_{(0)}(v_{\L})\\&=\sum_{i: 
\frac{1}{2}>s_i>0}s_i[y_{i},y^{i}]^M_{(0)}(v_{\L}
)+\sum_{i: 0<s_i<\frac{1}{2}}(s_i-1)[y_{i},y^{i}]^M_{(0)}(v_{\L})\\&=-
\sum_{i: \frac{1}{2}>s_i>0}(1- 
2s_i)[y_{i},y^{i}]^M_{(0)}(v_{\L}).\end{align*}\par We can choose
$y_i\in 
\g^{\ov s_i}_\alpha$ so  that $[y_i,y^i]=h_{\alpha}$, hence $$ 
\sum_{i:\frac{1}{2}>s_i>0}(1- 
2s_i)[y_{i},y^{i}]^M_{(0)}(v_{\L})=\sum_{i:0<s_i<\frac{1}{2}}(1-
2s_i)(2\rho_{s_i},\ov\L)v_\L. $$This completes the proof of 
\eqref{azioneasinistra}.\end{proof}
\par We  extend the Lie algebra $L'(\g,\si)$ by setting $\widehat  
L(\g,\si)=L'(\g,\si)\oplus \C d$, where
$d$ is the derivation of $L'(\g,\si)$  such that $d(K)=0$ and $d$ 
acts as $t\frac{d}{dt}$ on
$L(\g,\si)$.  Set $\ha_0=\h_0\oplus\C K\oplus 
\C d$. 
If $\L\in \ha_0^*$, a $\widehat L(\g,\si)$-module $M$ is called a   
highest weight module with highest
weight $\L$ if $M$  is a  highest weight module  for  $L'(\g,\si)$ 
with highest weight $\L_{|\h'}$ and
$d\cdot  v_{\L_{|\h'}}=\L(d)v_{\L_{|\h'}}$. We let $d^M$ be the operator on $M$ given by the action of $d$.

\begin{lemma}\label{xixibis}
 If $M$ is a  highest weight module over 
$\widehat L(\g,\si)$ with highest weight $\L$ and level $k$, then, as 
an operator  on $M$, 
\begin{equation}\label{casa} (L^\g)_{(1)}^{M}+(k+g)  
d^M=\left(\half(\ov \L+2\rho_\si,\ov
\L)+kz(\g,\si)+(k+g)\L(d)\right)I_M.
\end{equation}
\end{lemma}

\begin{proof}It is well known (and easy to show) that if $x\in\g$,  
then $[x_\l L^\g]=(k+g)\l  x,
$ hence, by \eqref{rep1}
\begin{equation}\label{Sugawaraction} [x^M_{(n)},  
(L^\g)^M_{(1)}]=(k+g)nx^M_{(n)}.
\end{equation}
 It follows that, as  operators on $M$,
\begin{equation}\label{casimir} [t^n\otimes x, 
(L^\g)^M_{(1)}+(k+g)d^M]=0.
\end{equation} By Lemma~\ref{xixi}, 
$$((L^\g)^M_{(1)}+(k+g)d^M)\cdot  v_\L=(\half(\ov \L+2\rho_\si,\ov \L)+ 
kz(\g,\si)+(k+g)\L(d))v_\L.
$$ Since $M=U(L'(\g,\si))\cdot v_\L$, \eqref{casimir} implies the 
result.
\end{proof}

\subsection{Twisted representations of the fermionic vertex \mbox{algebra}}
\label{twisted fermionic}
Analogously to the affine vertex algebra case, if $A$ is an odd 
vector superspace with a 
non-degenerate bilinear
symmetric form $(\cdot,\cdot)$  and $\si$ is a  semisimple 
automorphism of $A$ with modulus one
eigenvalues that keeps the bilinear  form invariant, then we can 
extend $\si$ to $R^{Cl}(A)$ by letting
$T$ and $\si$ commute and setting $\si(K')=K'$. As in the affine 
case, we can extend $\si$ to $F(A)$.

We set 
$L(A,\si)=\oplus_{\mu\in\R}(t^\mu\otimes A^{\ov \mu} )$ and define 
the bilinear form 
$<\cdot,\cdot>$ on $L(A,\si)$ by setting $<t^\mu\otimes 
a,t^\nu\otimes  b>=\d_{\mu+\nu,- 1}(a,b)$. Let
$Cl(L(A,\si))$ be the corresponding Clifford algebra. We choose   a 
maximal  isotropic subspace
$L^+(A,\si)$ of $L(A,\si)$ as follows: 
 fix a $\si$-invariant maximal isotropic subspace $A^+$  of 
$A^{-\ov{\frac{1}{2}}}$, and let 
$$L^+(A,\si)= (\oplus_{\mu>-\frac{1}{2}}(t^\mu\otimes A^{\ov \mu} 
))\oplus(t^{-
\frac{1}{2}}\otimes A^+ ).$$ \par We obtain a Clifford module 
$F^\si(A)=Cl(L(A,\si))/Cl(L(A,\si))L^+(A,\si)$.   Then we can define 
fields 
$Y^\si(a,z)=\sum_{n\in\ov\mu} (t^n\otimes a)z^{-n-1}, a\in\ 
A^{\ov\mu},$ where we let 
$t^n\otimes a$ act on $F^\si(A)$ by left multiplication. Set 
$Y^\si(K',z)=I_{F^{\sigma}(A)}$.  Lemma \ref{prop1.1} now
gives a $\si$-twisted representation of $F(A)$ on 
$F^\si(A)$. If $a\in F(A)$ then we write $a^\si_{(n)}$ instead of 
$a^{F^\si(A)}_{(n)}$.
 Fix a basis $\{b_i\}$ of $A$ and let $\{b^i\}$ be its dual basis. Set
\begin{equation}\label{da}
 L^A=-\half\sum_i:T(b_i)b^i:\,\in F(A).
\end{equation} It is well known (and easy to compute) that 
\begin{equation}\label{isprimary} [L^A\phantom{.}\!\!_\l a]=-(T+\half \l)a
\end{equation}
 for $a\in A$, hence, by \eqref{rep1}, 
\begin{equation}
 [L^A_{(1)},a_{(n)}]=(n+\half)a_{(n)}
 \end{equation}
 As in \cite[(1.16)]{KW}, we have  from \eqref{tensore}:
$$ 
Y^{\si}(L^A,z)=-\half(\sum_i:Y^{\si}(T(b_i),z)Y^{\si}(b^i,z): 
+\binom{s_i}{2}z^{-2}).
$$
  
\subsection{Twisted representations of the super affine vertex 
algebra}\label{repsup}

Fix, once and for all, a semisimple automorphism $\si$  of $\g$ with 
modulus 1 eigenvalues that keeps the bilinear  form invariant. 
This automorphism extends to two automorphisms of the Lie conformal 
superalgebra $R^{super}$,
denoted by $\sigma$ and $\tau$, as follows. Both fix $\mathcal K$ and 
$\overline{\mathcal K}$ and
commute with $T$, both act on $1\otimes \g$ as $1\otimes \sigma$, and 
$\sigma$ (resp. $\tau$) acts
on $1\otimes \ov\g$ as $1\otimes \sigma$ (resp. $-1\otimes \sigma$). 
We therefore obtain two
automorphisms of $V(R^{super})$, also denoted  by $\sigma$ and
$\tau$. Clearly  $\sigma$ and $\tau$ stabilize $$:(\mathcal 
K-k\va)V(R^{super}):+:(\overline{\mathcal 
K}-\va)V(R^{super}):,$$ so we obtain  automorphisms of 
$V^{k,1}(R^{super})$, also denoted by  $\sigma$ and $\tau$.
\par Denote by  $\ov\tau$ the automorphism $\tau$   restricted  to  
$R^{Cl}(\ov \g)=\C[T]\otimes
\ov\g+\C \overline{\mathcal K}$. As above, we can  extend this 
automorphism to $F(\ov\g)$, also
denoted by $\ov\tau$.
\par Observe that $\widetilde{ \si(x)}=\tau(\tilde  x)$ for $x\in\g$.
  Indeed, 
\begin{equation}\label{fortau}
\si(x)-\half\sum_i:\overline{[\si(x),x_i]}\ov x_i:=\tau(x)-
\half\si(\sum_i:\overline{[x,\si^{-1}(x_i)]}\,\overline{\si^{-1}(  
x_i)}):=\tau(\tilde x).
\end{equation}
\par
\begin{rem}\label{actionMF} It follows from \eqref{fortau} that the 
isomorphism 
$V^{k+g,1}(R^{super})\to
V^k(\g)\otimes F(\overline\g)$ intertwines $\tau$ and 
$\si\otimes\ov\tau$. Thus, if $M$ is a level 
$k$ highest weight $L'(\g,\si)$-module, then $M\otimes F^{\ov 
\tau}(\ov \g)$ is a 
$\si\otimes \ov\tau$-twisted representation of $V^k(\g)\otimes F(\ov 
\g)$  hence a $\tau$-twisted
representation of 
$V^{k+g,1}(R^{super})$. In particular, if $b\in\g$, then $\widetilde 
b$ acts only on the first factor
of $M\otimes F^{\ov \tau}(\ov \g)$ whereas $\ov b$ acts only on the 
second factor.
\end{rem}

Next let $\sa$ be  a reductive $\sigma$-invariant subalgebra of $\g$, 
such that $(\cdot,\cdot)$ remains
nondegenerate  when restricted to $\sa$. Let $\p$ be the  orthogonal 
complement of $\sa$ in $\g$. If
$x\in\g$ we write $x=x_\sa+x_\p$ for the orthogonal decomposition of 
$x$.  We fix once and for all an 
orthonormal basis   $\{a_i\}$  of $\sa$ and an orthonormal basis 
$\{b_i\}$ of $\p$. 

 Let $Cas_\sa$ be the Casimir of $\sa$ with
 respect to 
$(\cdot,\cdot)_{|\sa}$. Write $\sa=\sum_S\sa_S$ for the eigenspace 
decomposition  of $\sa$ under the
action of $Cas_\sa$, and let $2g_S$ be the eigenvalue relative to 
$\sa_S$. In particular we let
$\sa_0$ be the  center of $\sa$, while $\sa_S$ is semisimple for 
$S>0$. 

 We construct the Lie  conformal algebra $Cur(\sa_S)$ and, given 
$k\in\C$, the corresponding vertex 
algebra $V^k(\sa_S)$ using the form $(\cdot,\cdot)$ restricted to 
$\sa_S$,  so that, if $x,y\in\sa_S$,
then $[x_\l y]=[x,y]+\l(x,y)K_S$. We abuse slightly  of  notation by 
letting
$Cur(\sa)=(\C[T]\otimes\sa)\oplus(\oplus_S \C   K_S)$.

The super affine conformal algebra $R^{super}(\sa)$ corresponding to 
$\sa$ embeds naturally in
$R^{super}$, thus we have an embedding of $V^{k+g,1}(R^{super}(\sa))$ 
in $V^{k+g,1}(R^{super})$. In
particular $M\otimes F^{\ov\tau}(\ov \g)$  turns into a 
representation of $V^{k+g,1}(R^{super}(\sa))$
by restriction.  Set, for $x\in\sa$, 
$$(\tilde x)_\sa=x-\half\sum_i:\overline{[x,a_i]}\ov  a_i:.
$$

 Since $ad(Cas_\sa)_{|\sa_S}=2g_SI_{\sa_S}$, applying  
Proposition~\ref{isomorfi} to
$R^{super}(\sa)$, we have that $x\mapsto (\tilde x)_\sa$, $\ov 
x\mapsto \ov x$ induces an
isomorphism   $(\otimes_S V^{k+g-g_S}(\sa_S))\otimes 
F^{\ov\tau}(\ov\sa)\to
V^{k+g,1}(R^{super}(\sa))$ that intertwines $\si\otimes \ov\tau$ with 
$\tau$. 

It follows that we can look
upon $M\otimes F^{\ov\tau}(\ov\g)$ as a $\si\otimes \ov\tau$-twisted 
representation of
$(\otimes_SV^{k+g-g_S}(\sa_S))\otimes F^{\ov\tau}(\ov\sa)$.

In order to understand this representation we write $N=M\otimes 
F^{\ov\tau}(\ov\g)$ as $(M\otimes
F^{\ov\tau}(\ov\p))\otimes F^{\ov\tau}(\ov\sa)$. For $x\in\sa$ set 
\begin{equation}\label{thetaref}\theta(x)=(\tilde x)_\sa-\tilde  
x.\end{equation}
Then $x$ acts on $N$ via $Y^N((\tilde x)_\sa,z)=Y^N(\tilde
x,z)+Y^N(\theta(x),z)=Y^M(x,z)\otimes 
I_{F^{\ov\tau}(\ov\g)}+I_M\otimes
Y^{F^{\ov\tau}(\ov\g)}(\theta(x),z)$. Since 
        \begin{equation}\label{theta(x)}
        \theta(x)=\half\sum_i:\overline{[x,b_i]}\ov  b^i:,
\end{equation}
we have that $\theta(x)\in F(\ov\p)$, and  in turn
$Y^{F^{\ov\tau}(\ov\g)}(\theta(x),z)=Y^{F^{\ov\tau(\ov\p)}}(\theta(x),z)\otimes
I_{F^{\ov\tau}(\ov\sa)}$. Thus, as a representation of 
$\otimes_SV^{k+g-g_S}(\sa_S)\otimes
F^{\ov\tau}(\ov\sa)$, $M\otimes F^{\ov \tau}(\ov\g)$ is $(M\otimes 
F^{\ov \tau}(\ov\p))\otimes
F^{\ov \tau}(\ov\sa)$, where the representation of  
$V^{k+g-g_S}(\sa_S)$ on $M\otimes F^{\ov
\tau}(\ov\p)$ is the one induced by the field module of $Cur(\sa)$ 
defined by 
\begin{equation}\label{actiononmtf} Y^M(x,z)\otimes 
I_{F^{\ov\tau}(\ov\p)}+I_M\otimes
Y^{F^{\ov\tau}(\ov\p)}(\theta(x),z).
\end{equation}

Recall that $\Dp_0$ is a  subset of positive roots for the set of 
$\h_0$-roots of $\g^{\ov 0}$. Let
$\b=\h_0\oplus\n$   be the corresponding Borel subalgebra. Fix a 
Cartan subalgebra $\h_\sa$ of 
$\sa^{\ov 0}=\sa\cap\g^{\ov 0}$. We can assume that $\h_\sa\subset\h_0$, so 
that  $\h_0=\h_\sa\oplus\h_\p$ is 
the orthogonal decomposition of $\h_0$. Furthermore, as  shown in 
\S~1.1 of \cite{K3}, we can assume
that 
$\n=\n\cap\sa^{\ov 0}\oplus\n\cap(\p\cap\g^{\ov 0})$ and that, if 
$\n_\sa=\n\cap\sa^{\ov 0}$,  then $\h_\sa\oplus\n_\sa$
is a Borel subalgebra of $\sa^{\ov 0}$.  Let $\Dp_\sa$  be  the 
corresponding subset of positive roots in the
set $\D_\sa$ of 
$\h_\sa$-roots of $\sa^{\ov 0}$.  Set 
$\n_-=\sum\limits_{\a\in-\Dp_0}\g^{\ov 0}_\a$.  The same  argument 
used  for
$\n$ shows that $\n_-=\n_-\cap\sa^{\ov 0}\oplus\n_-\cap(\p\cap 
\g^{\ov 0})$. We define $L'(\sa,\si)=\sum_{j\in\R}t^j\otimes\sa^j\oplus 
(\sum_S\C  K_S)$. This is a Lie algebra
with bracket defined by
$$[t^i\otimes a,t^j\otimes b ]=t^{i+j}\otimes [a,b] + 
\d_{i,j}i(a,b)K_S$$  for 
$a\in\sa_S$, the elements $K_S$ being  central. 

Set $\n'_\sa=\n'\cap L'(\sa,\si)$. If $\mu\in(\h_\sa\oplus(\sum_S\C  
K_S))^*,$ we say that a
$L'(\sa,\si)$-module $M$ is a highest weight module of  weight $\mu$  
if there is a nonzero vector
$v_\mu\in M$ such that 
\begin{equation}\mathfrak n'_\sa(v_\mu)=0,\,\ hv_\mu=\mu(h)v_\mu 
\text{ for 
$h\in\h_\sa\oplus(\sum_S\C K_S)$, }\ 
U(L'(\sa,\si))v_\mu=M.\end{equation}\par If $M$ is a highest
weight module for 
$L'(\g,\si)$ with highest weight $\L$ and $k=\L(K)$, then $M\otimes 
F^{\ov\tau}(\ov \p)$ is a
representation of
$\otimes_S  V^{k+g-g_S}(\sa_S)$, thus we  can regard  $M\otimes 
F^{\ov\tau}(\ov \p)$ as a
$L'(\sa,\si)$-module. In particular, by letting $M$  be the trivial 
representation of $L'(\g,\si)$, we have 
an action  of $L'(\sa,\si)$ on $F^{\ov\tau}(\ov\p)$.
\par Let $\L_0^S$ be the element of $(\h_\aa+\sum_S\C K_S)^*$ defined 
by 
$$\L_0^S(\h_\aa)=0,\quad\L_0^S(K_T)=\d_{ST}.$$ Define moreover 
\begin{equation}\label{rhoa}
\rho_{\aa j}=\half\sum\limits_{\a\in(\D_{\ov 
j})_{|\h_\aa}}\dim(\sa^{\ov j})_\a\a,\quad
\rho_{\aa\,\si}=\sum\limits_{0\leq j<\half} (1-2j)\rho_{\aa  j}.
\end{equation}

\begin{lemma}\label{lemmamu} Let $M$ be a level $k$ highest
weight module for 
$L'(\g,\si)$ with highest weight $\L$. Set 
$M'=U(L'(\sa,\si))(v_\L\otimes 1)$. Then $M'$ is a highest weight 
$L'(\sa,\si)$-module with highest  weight
\begin{equation}\label{mu}\mu=(\L+\rho_\sigma)_{|\h_\aa}-
\rho_{\aa\,\sigma}+\sum_S(k+g- g_S)\L_0^S.\end{equation}
\end{lemma}
\begin{proof}
 If $x\in \sa^j$ (with $-\half\le j<\half$) then, by 
\eqref{actiononmtf}, $t^n\otimes x$ acts via 
$x^M_{(n)}\otimes I_{F^{\ov\tau}(\ov\p)} +I_M\otimes\theta(x)^{\ov \tau}_{(n)}$.  If $n>0$ 
then $x^M_{(n)}(v_\L)=0$. Moreover 
$\theta(x)^{\ov\tau}_{(n)}=\half\sum_i:\overline{[x,b_i]}\ov  
b^i:^{\ov\tau}_{(n)}$, hence, using 
\eqref{tensore},
\begin{align*}\theta(x)^{\ov 
\tau}_{(n)}(1)&=\half\sum_{i: 
j+s_i>0}\overline{[x,b_i]}^{\ov\tau}_{(j+s_i-
\frac{1}{2})}(\ov b^i)^{\ov\tau}_{(-j-s_i-\frac{1}{2}+n)}(1)\\&-
\half\sum_{i: j=s_i=-\frac{1}{2}}(\ov b^i)^{\ov\tau}_{(-
\frac{1}{2}+n)}\overline{[x,b_i]}^{\ov\tau}_{(-
\frac{1}{2})}(1)\\&=0.
\end{align*} 

If $n=0$ and $x\in \n_\sa$ then, since 
$\n_\sa\subset\n$, $x^M_{(0)}(v_\L)=0$. Moreover 
\begin{align*}\theta(x)^{\ov\tau}_{(0)}(1)&=\half\sum_{i: s_i\ge 0} 
\overline{[x,b_i]}^{\ov\tau}_{(s_i-\frac{1}{2})}(\ov 
b^i)^{\ov\tau}_{(-s_i-
\frac{1}{2})}(1)\\&=\half\sum_{i: s_i= 0} 
\overline{[x,b_i]}^{\ov\tau}_{(-
\frac{1}{2})}(\ov b^i)^{\ov\tau}_{(-\frac{1}{2})}(1)\\&+\half\sum_{i: 
s_i> 0}  ([x,b_i],
b^i).\end{align*}Since
$x\in\n$ and, if $s_i>0$,  
$[b_i,b^i]\in\h_0$, we have that $\sum_{i,s_i> 0} ([x,b_i], 
b^i)=\sum_{i,s_i>  0} (x,[b_i, b^i])=0$.  We
choose a maximal isotropic space $\h_\p^+$ of $\h_\p$, so that
$(\n\cap\p)\oplus\h_\p^+$ is a maximal isotropic space in 
$\g^{\ov 0}\cap\p$. 

We can  choose the basis $\{b_i\}$ as the union of a basis $\{x_i\}$ 
of $\n\cap\p$ with  an orthonormal
basis $\{h_i\}$ of $\h_\p$ and a basis $\{y_i\}$ of $\n_-
\cap\p$. Set $\{x^i\}$ (resp. $\{y^i\}$) be the basis of $\n_-\cap\p$ 
(resp. 
$\n\cap\p$) dual to $\{x_i\}$ (resp. $\{y_i\}$). Then 
\begin{align*}\sum_{i: s_i= 0} \overline{[x,b_i]}^{\ov\tau}_{(-
\frac{1}{2})}(\ov b^i)^{\ov\tau}_{(-\frac{1}{2})}(1)&=\sum_{i} 
\overline{[x,x_i]}^{\ov\tau}_{(-\frac{1}{2})}(\ov x^i)^{\ov\tau}_{(-
\frac{1}{2})}(1)\\&+\sum_{i} \overline{[x,h_i]}^{\ov\tau}_{(-
\frac{1}{2})}(\ov h_i)^{\ov\tau}_{(-\frac{1}{2})}(1)\\&+\sum_{i} 
\overline{[x,y_i]}^{\ov\tau}_{(-\frac{1}{2})}(\ov y^i)^{\ov\tau}_{(-
\frac{1}{2})}(1).\end{align*}Since $x\in\n$, then $[x,h]\in\n\cap\p$ 
and 
$[x,x_i]\in\n\cap\p$, thus\begin{align*}\sum_{i: s_i= 0} 
\overline{[x,b_i]}^{\ov\tau}_{(-\frac{1}{2})}(\ov b^i)^{\ov\tau}_{(-
\frac{1}{2})}(1)&=\sum_{i}(x,[x_i,x^i])+\sum_{i} (x,[h_i,  
h_i])=0.\end{align*}It remains to compute the
highest weight $\mu$. If 
$h\in\h_\sa$, then$$(h)^N_{(0)}\cdot (v_\L\otimes 1)= 
hv_\Lambda\otimes 1+  v_\L\otimes
\theta(h)^{\ov \tau}_{(0)}\cdot 1=\L(h)v_\Lambda\otimes 1+  
v_\L\otimes \theta(h)^{\ov \tau}_{(0)}\cdot
1.$$Now $$\theta(h)^{\ov 
\tau}_{(0)}\cdot 1=\sum_i:\ov{ [h,b_i]}\ov b^i:^{\ov 
\tau}_{(0)}\cdot  1.$$ Applying \eqref{tensore} we find
that 
$$ Y^{\ov\tau}(\theta(h),z)=\half\sum_i:Y^{\ov 
\tau}(\ov{[h,b_i]},z)Y^{\ov \tau}(\ov b^i,z):-
\sum_{\substack{ j>0\\ \a\in(\D_{\ov j})_{|\h_\aa}}} j \a(h) 
\dim(\g^j\cap\p)_\a z^{- 1}.
$$

Writing out  explicitly the normal order in the r.h.s of the 
previous  equation,  we 
get:\begin{align*}&\theta(h)^{\ov \tau}_{(0)}\cdot  
1=\\&\half\sum_i\ov{[h,b_i]}^{\ov \tau}_{s_i-\half}(\ov
b^i)^{\ov\tau}_{-s_i-
\half}\cdot 1-\sum_{j>0}2j(\rho_j-\rho_{\aa 
j})(h)=\\&\half\sum_{i:s_i=0} 
\ov{[h,b_i]}^{\ov \tau}_{-\half}(\ov b^i)^{\ov \tau}_{-\half}\cdot 
1+\half 
\sum_{i:s_i>0}([h,b_i],b^i)-\sum_{j>0}2j(\rho_j-\rho_{\aa  
j})(h)\\&\half\sum_{i:s_i=0}
\ov{[h,b_i]}^{\ov
\tau}_{-\half}(\ov b^i)^{\ov 
\tau}_{-\half}\cdot 1+ \sum_{j>0}(\rho_j-\rho_{\aa j})(h)-
\sum_{j>0}2j(\rho_j-\rho_{\aa j})(h).\end{align*}Choosing  bases 
$\{x_i\}$ in 
$\n\cap\p$, $\{y_i\}$ in $\n^-\cap\p$ and $\{h_i\}$ in $\h_\p$ as 
above we have 
$$\half\sum_{i:s_i=0} \ov{[h,b_i]}^{\ov \tau}_{-\half}(\ov b^i)^{\ov 
\tau}_{-
\half}\cdot 1=\half\sum_{i:s_i=0} \ov{[h,x_i]}^{\ov 
\tau}_{-\half}(\ov x^i)^{\ov 
\tau}_{-\half}\cdot 1=(\rho_0-\rho_{\aa\,0})(h).$$ The final outcome 
is that 
$$
\theta(h)^{\ov \tau}_{(0)}\cdot 1=(\rho_\sigma-
\rho_{\aa\,\sigma})(h).
$$ Since we are looking at $M\otimes F^{\ov\tau}(\ov\p)$ as a 
representation of 
$\otimes_S V^{k+g-g_S}(\aa_S)$, then $K_S$ acts as $(k+g- g_S)I$. 
\end{proof}

Similarly to what we have done with $L'(\g,\si)$,  we define 
$\widehat L(\sa,\si)$  by extending $L'(\sa,\si)$ with  $d_\sa$  and  
letting $d_\sa$ and $K_S$ commute
for all $S$. We wish to  extend the action of $L'(\sa,\si)$ on 
$F^{\ov \tau}(\ov\p)$ to $\widehat
L(\sa,\si)$. In order to do this we need the following computation. 
Recall from \eqref{da} the
definition of the element $L^A$ for the vector superspace $A$.

\begin{lemma}\label{dzero}In $V^{k,1}(R^{super})$,
$$ [\theta(x)_\l L^{\ov\p} ]=-
\l\theta(x).
$$

Consequently, due to \eqref{rep1},  for any $a\in\C$, we can extend 
the action of $L'(\sa,\si)$ on
$F^{\ov\tau}(\ov\p)$ to 
$\widehat L(\sa,\si)$ by letting $d_\sa$ act as 
$(L^{\ov\p})^{\ov\tau}_{(1)}+aI$.
\end{lemma}
\begin{proof} By  \eqref{tre} and Wick formula \eqref{Wick}, we have 
\begin{align*}
\half\sum_i[\theta(x)_\l :T(\ov b_i)\ov  b_i:]&=\half\sum_i[(\tilde 
x)_\sa{}_\l :T(\ov b_i)\ov b_i:]-
\half\sum_i[\tilde x_\l :T(\ov b_i)\ov b_i:]\\&=\half\sum_i[(\tilde  
x)_\sa{}_\l :T(\ov b_i)\ov
b_i:]\\&=\half\sum_i(:[(\tilde x)_\sa{}_\l   T(\ov b_i)]\ov b_i:+: 
T(\ov b_i)[(\tilde x)_\sa{}_\l 
b_i]:)\\&+\frac{1}{2}\sum_{i}\int_0^\l[(\tilde x)_\sa{}_\la T(\ov  
b_i)]_\mu\ov b_i]d\mu.
\end{align*} 
As in the proof of Proposition 
\ref{isomorfi} it can be computed easily that, if $y\in\p$ and 
$x\in\sa$, then 
$[(\tilde x)_\sa{}_\l \ov y ]=\overline{[x,y]}$. By sesquilinearity of the $\l$-bracket  we have  then $[(\tilde x)_\sa{}_\l   T(\ov 
b_i)]=T([(\tilde x)_\sa{}_\l\ov  b_i])+\l[(\tilde
x)_\sa{}_\l
\ov b_i]=T(\overline{[x, b_i]})+\l\overline{[x,b_i]}$,  hence we can 
write
\begin{align*}
\half\sum_i[\theta(x)_\l :T(\ov b_i)\ov  
b_i:]&=\half\sum_i(:T(\overline{[x,b_i]})\ov
b_i:+\l:\overline{[x,b_i]}\ov b_i:+ : T(\ov b_i)\overline{[x, 
b_i]}:)\\&+\frac{1}{2}\sum_{i}\int_0^\l[(T(\overline{[x, 
b_i]})+\l\overline{[x,b_i]})_\mu\ov b_i]d\mu\\
&=\l\theta(x)+\frac{1}{2}\sum_{i}\int_0^\l(-\l([x, 
b_i],b_i)+\l([x,b_i],  b_i)d\mu\\ &=\l\theta(x).
\end{align*}
\end{proof}

\section{Dirac operators}\label{diracoperators}The  affine Dirac 
operator was introduced by Kac and
Todorov in \cite{KacT}. It is the  following odd element  of 
$V^{k+g,1}(R^{super})$:
\begin{equation}\label{G}G_\g=\sum_i:x_i\overline  
x^i:+\tfrac{1}{3}\sum_{i,j}:\overline{[x_i,
x_j]}\overline  x^i\overline x^j:.
\end{equation} 
Here $\{x_i\},\,\{x^i\}$ is a pair of dual bases of $\g$ w.r.t. the 
invariant  form
$(\ ,\ )$. 
Then, choosing $x_i$ as  the eigenvectors of $\sigma$, say 
$\sigma(x_i)=a_ix_i$, we see
that
$\tau(\ov x_i)=-a_i\ov x_i$, $\tau(\ov x^i)=-a_i^{-1}\ov x^i$, hence
\begin{equation}\label{aztau}
\tau(G_\g)=-G_\g.\end{equation}
The element 
$G_\g$ has the following  properties:
\begin{align} &[a_\la G_\g]=\la (k+g)\overline a,\label{aG}\\  
&[\overline a_\la G_\g]=a\label{abarG}.
\end{align}
Note  that
\begin{equation}\label{diractilde} G_\g=\sum_i:\tilde x_i\overline 
x^i:-
\tfrac{1}{6}\sum_{i,j}:\overline{[x_i, x_j]}\overline  x^i\overline 
x^j:.
\end{equation}
In Section \ref{Gquadro} we will show  (cf. \cite{KacD}) that
\begin{equation}\label{[GG]}[G_\g{}_\la  G_\g]=\sum_i :\tilde 
x_i\tilde  x^i:+(k+g)\sum_i:T(\overline
x_i)\overline  x^i:+\frac{\la^2}{2}(k+\frac{g}{3})\dim\g.
\end{equation}   Identifying 
$V^{k+g,1}(R^{super})$ with $V^k(\g)\otimes F(\ov \g)$ we  have that 
\eqref{diractilde} and \eqref{Gquadro}
 can be rewritten as
\begin{equation}\label{tensordirac}G_\g=\sum_ix_i\otimes\overline x^i-
\tfrac{1}{6}\sum_{i,j}\va\otimes:\overline{[x_i, x_j]}\overline  
x^i\overline x^j:
\end{equation} and
\begin{equation}\label{gquad}[G_\g{}_\la  
G_\g]=2(L^\g\otimes\va)-2(k+g)(\va\otimes
L^{\ov\g})+\frac{\la^2}{2}(k+\frac{g}{3})\dim\g.
\end{equation} 
where $L^\g$ is defined in \eqref{sugaw} and $L^{\ov\g}$ is defined 
in \eqref{da}. Note that if $L_\g$ is as in \eqref{primosugawara}, then we have
$$L_\g=2(L^\g\otimes\va)-2(k+g)(\va\otimes
L^{\ov\g}).$$\par
We observe that $G_{\g}\in V^k(\g)\otimes F(\ov \g)$, so,  fixing a 
restricted module $M$ for $L'(\g,\si)$
of level $k$ and setting 
$N=M\otimes F^{\ov \tau}(\ov\g)$, we can consider the  twisted 
quantum field
$$ Y^N(G_{\g},z)=\sum_{n\in\tfrac{1}{2}+\ganz}G^N_{(n)} z^{-n- 1}=
\sum_{n\in\ganz}G^N_{(\tfrac{1}{2}+n)}
z^{-n-
\tfrac{3}{2}}.$$
(Recall from \eqref{aztau} that $\ov\tau(G_\g)=-G_\g$.) Let  
$G^N_n=G^N_{(\tfrac{1}{2}+n)}$. We want to calculate $(G^N_0)^2$. Using 
\eqref{gquad}  we have
\begin{align}\label{gzeroquad} 
(G^N_0)^2&=\frac{1}{2}[G^N_0,G^N_0]=\frac{1}{2}[G^N_{(\tfrac{1}{2
})},G^N_{(\tfrac{1}{2})}]=\tfrac{1}{2}(G_{\g}{}_{(0)}G_{\g})^N_{ (1)}-
\tfrac{1}{16}(k+\frac{g}{3})(\dim\g)\, I_N.
\end{align}

Combining \eqref{gzeroquad} and  \eqref{gquad} with  $\l=0$, we 
obtain 
\begin{equation}(G^N_0)^2=(L^\g)^M_{(1)}\otimes I_{F^{\ov\tau}(\ov\g)}- 
(k+g)I_M\otimes
(L^{\ov\g})^{\ov \tau}_{(1)}-
\tfrac{1}{16}(k+\frac{g}{3})(\dim\g)\, I_N\label{gzeronuovo}.
\end{equation}

\vskip 5pt We are interested in calculating 
$G^N_0(v_\Lambda\otimes 1)$, $v_\Lambda$ being a highest weight vector 
of a 
$L'(\g,\sigma)$-module $M$ with highest weight $\L$. From 
\eqref{tensordirac} we know that $G^N_0$ splits as the sum of a  
quadratic and a cubic term. We shall
calculate the action of these two sums  separately.  We assume that 
$x_i\in 
\g^{\ov s_i}$ where $-\tfrac{1}{2}\leq s_i < \tfrac{1}{2}$, so that 
$x^i\in 
\g^{-\ov s_i}$. Writing  $Y^{\overline \tau}$ for $Y^{F^{\overline 
\tau}(\overline\g)}$ and, if $a\in F(\bar\g)$, $a^{\bar\tau}_{(r)}$ 
for $a^{F^{\overline 
\tau}(\overline\g)}_{(r)}$, we have 
\begin{align}\notag&Y^{\overline 
\tau}(\sum_{i,j}:\overline{[x_i,x_j]}\overline x^i\overline 
x^j:,z)=\\  &\sum_{i,j}:Y^{\overline
\tau}(\overline{[x_i,x_j]},z)Y^{\overline 
\tau}(\overline x^i,z)Y^{\overline\tau}(\overline 
x^j,z):+3\sum_{i}(s_i+\half)Y^{\overline\tau}(\overline{[x_i,x^i]},z)z^{- 
1}.\label{Kanomaly}
\end{align} This equality follows by repeated applications  of 
\eqref{tensore}. We now simplify the
second summand of the right hand side.  We already observed that
$$
\sum_{s_i=t}[x_i,x^i]=-\sum_{s_i=- t}[x_i,x^i]
$$ hence in particular $\sum_{s_i=0}[x_i,x^i]=0$.  Thus
$$
\sum_{i}(s_i+\half)[x_i,x^i]=\sum_{0<s_i<\frac{1}{2}}2s_i[x_i,x^i].
$$

If we single out the coefficient of $\sum_{i,j}:Y^{\overline 
\tau}(\overline{[x_i,x_j]},z)Y^{\overline \tau}(\overline  
x^i,z)Y^{\overline\tau}(\overline x^j,z):$
corresponding to $z^{-\frac{1}{2}- 1}$ and apply it to 
$v_\Lambda\otimes 1$ (indeed to $1$), we  have
\begin{align} &\sum_{i,j}:Y^{\overline 
\tau}(\overline{[x_i,x_j]}_\p,z)Y^{\overline \tau}(\overline  
x^i,z)Y^{\overline\tau}(\overline 
x^j,z):_{(\frac{1}{2})}(1)=\notag\\&\sum_{i,j}(\sum_{s_i=s_j=0}(\overline{[x_i 
,x_j]})^{\ov
\tau}_{(-\half)}(\overline x^i)^{\ov \tau}_{(-\half)}(\overline 
x^j)^{\ov \tau}_{(-\half)})(1)-
3\sum_{s_j>0}(\overline{[x_j,x^j]})^{\ov \tau}_{(-
\half)}(1).\label{Kcubico}
\end{align} The first summand in \eqref{Kcubico} is  the cubic term 
in Kostant's Dirac operator for
$\g^{\ov 0}$. In  Kostant \cite{K3} it is proven that
$$
\sum_{i,j}(\sum_{s_i=s_j=0}(\overline{[x_i,x_j]})^{\ov \tau}_{(-
\half)}(\overline x^i)^{\ov \tau}_{(-\half)}(\overline x^j)^{\ov 
\tau}_{(-\half)})(1)=-6(\ov  h_{\rho_0})^{\ov
\tau}_{(-\frac{1}{2})}(1).
$$ With easy calculations  one proves that 
\begin{equation}\label{Kquadratico}
\left(\sum_{i}x_i\otimes  \overline 
x^i\right)^N_{(\half)}(v_\Lambda\otimes1)=v_\Lambda\otimes 
(h_{\ov\Lambda})^{\ov \tau}_{(-\frac{1}{2})}\cdot 1.
\end{equation}  Now we can complete the proof of 
\begin{prop}\label{azioneg0}  Let $\rho_\si$ be as in \eqref{rhos}. 
Then 
\begin{equation}\label{0} G^N_0(v_\Lambda\otimes 1)=v_\Lambda\otimes 
(\ov 
h_{\ov\Lambda+\rho_\sigma})^{\overline\tau}_{(-\frac{1}{2})}\cdot  1.
\end{equation}
\end{prop}
\begin{proof} Collecting all the contributions 
\eqref{Kanomaly},\eqref{Kcubico}, and \eqref{Kquadratico}, we find  
that
\begin{equation}\label{finalee} G^N_0(v_\Lambda\otimes 
1)=v_\Lambda\otimes 
\left(\ov h_{\ov\Lambda+\rho_0}+\half\sum_{j: \frac{1}{2}>s_j>0}(1-
2s_j)\overline{[x_j,x^j]}\right)^{\ov
\tau}_{(-\frac{1}{2})}\cdot  1.\end{equation}
 Now, if $\{v_i\}$ is a basis of 
$\g^{\ov t}$, then $\sum\limits_{s_i=t}(1-2s_i)[v_i,v^i]$ is 
independent of the  choice of the basis.
It follows that 
\begin{equation}\label{y}\sum_{i: s_i=t}(1- 2s_i)[x_i,x^i]=\sum_{i: 
s_i=t}\sum_{\a\in\D_{\ov
t}}(1- 2s_i)[x_{\a i},x_\a^i]
\end{equation} where $\{x_{\a i}\}$ is a basis of 
$\g_\a^{\ov t}$. Since $[x_{\a i},x_\a^i]=h_\a$, we have that  the 
l.h.s. of 
\eqref{y} equals $2(1-2t)h_{\rho_t}$. Summing over $t$ and 
substituting in 
\eqref{finalee} we get \eqref{0}.
\end{proof} Recall from \eqref{zg} the definition of $z(\g,\si)$. We 
have
\begin{prop}\label{azioneDeG} \ %
\begin{enumerate} 
\item $(L^{\ov\g})^{\ov\tau}_{(1)}\cdot 1= 
z(\g,\si)-\tfrac{1}{16}\dim\g$.
\item If $M$ is a highest weight module of $L'(\g,\si)$ with highest 
weight $\L$, then
$$ (G^N_0)^2(v_\L\otimes
1)=\half\left((\ov\L+2\rho_\si,\ov\L)+\frac{g}{12}\dim\g-2gz(\g,\si)\right)(v_\L\otimes 
1).
$$
\end{enumerate}
\end{prop}
\begin{proof}If  $M_0$ is a highest weight module with highest weight 
$-\rho_\si+k\L_0$ then, by
Proposition~\ref{azioneg0},  $G^N_0(v_{-\rho_\si+k\L_0}\otimes 1)=0$. 
Applying \eqref{gzeronuovo} and
Lemma~\ref{xixi}, we find that
$$ 
0=(-\half\Vert\rho_\si\Vert^2+kz(\g,\si)-\tfrac{1}{16}(k+\frac{g}{3})\dim\g)(v_\L\otimes
1)-(k+g)(v_\L\otimes (L^{\ov\g})^{\ov\tau}_{(1)}\cdot 1).
$$ Since this equality holds for any $k$, the coefficient of $k$ must 
vanish. This implies the first
claim of the proposition. 

Again by \eqref{gzeronuovo} and Lemma~\ref{xixi}, 
\begin{align*} (G^N_0)^2(v_\L\otimes 1)&=(\half(\ov
\L+2\rho_\si,\ov\L)+kz(\g,\si)-\tfrac{1}{16}(k+\frac{g}{3})\dim\g)(v_\L\otimes 
1)\\&-(k+g)(v_\L\otimes
(L^{\ov\g})^{\ov\tau}_{(1)}\cdot 1).
\end{align*} Using the first equality we get the second claim.
\end{proof}

We now turn to the study of the relative Dirac operator. Fix a 
subalgebra $\sa$ as in \S~\ref{repsup}
and consider
$G_\sa\in V^{k+g,1}(R^{super}(\sa))\subset V^{k+g,1}(R^{super}) $. 

Set $G_{\g,\sa}=G_\g-G_\sa$.  By \eqref{aG} and \eqref{abarG},
$$ [{G_{\g,\sa}}_\l G_{\g,\sa}]=[{G_\g}_\l  G_\g]-[{G_\sa}_\l 
{G_\sa}]. 
$$  In particular 
\begin{equation}\label{decouple} 
((G_{\g,\sa})^N_0)^2=(G^N_0)^2-((G_\sa)^N_0)^2,
\end{equation} so, by \eqref{gzeronuovo},
\begin{align}\notag  ((G_{\g,\sa})^N_0)^2&=(L^\g)^M_{(1)}\otimes 
I_{F^{\ov\tau}(\ov\g)}-
(k+g)I_M\otimes 
(L^{\ov\g}-L^{\ov\sa})^{\ov\tau}_{(1)}\\\notag&-(L^\sa)^{M\otimes
F^{\ov\tau}(\ov\p)}_{(1)}\otimes I_{F^{\ov\tau}(\ov\sa)}\\&
-\tfrac{1}{16}\left((k+\frac{g}{3})\dim\g-\sum_S(k+g-
\frac{2g_S}{3})\dim\sa_S\right)I_N\label{gzeroquadro}.
\end{align}

\begin{prop}\label{lpirhozero} If $(\ov\L+\rho_\si)_{|\h_\p}=0$ then 
$(G_{\g,\sa})^N_0(v_\L\otimes 1)=0$
and, if
$\mu$ is as in \eqref{mu}, 
\begin{align}\notag (\ov\mu+2\rho_{\aa\,\sigma},\ov\mu)&-
(\ov\L+2\rho_\si,\ov\L)\\&=\frac{g}{12}\dim\g-2gz(\g,\si)-\sum_S(\frac{g_S}{12}\dim\sa_S-2g_Sz(\sa_S,\si))\label{masternok}.
\end{align}
\end{prop}
\begin{proof} Since $N=M\otimes F^{\ov\tau}(\ov \g)=(M\otimes 
F^{\ov\tau}(\ov\p))\otimes
F^{\ov\tau}(\ov\sa)$, applying Proposition~\ref{azioneg0} to $G_\sa$ 
and using
Lemma~\ref{lemmamu}, we find that
$$ (G_{\g,\sa})^N_0(v_\L\otimes 1)=v_\L\otimes (\ov 
h_{\ov\L+\rho_\si})^{\ov\tau}_{(-\half)}\cdot 1-
v_\L\otimes (\ov 
h_{(\L_+\rho_\si)_{|\h_\sa}})^{\ov\tau}_{(-\half)}\cdot 1.
$$ Since 
$h_{\ov\L_+\rho_\si}-h_{(\L_+\rho_\si)_{|\h_\sa}}=(h_{\ov\L_+\rho_\si})_\p$, 
we obtain
\begin{equation}\label{Grelativo} (G_{\g,\sa})^N_0(v_\L\otimes 
1)=v_\L\otimes (\overline{
(h_{\ov\L +\rho_\si})}_{\p})^{\ov\tau}_{(-\half)}\cdot 1.
\end{equation} This proves the first part of the statement.

In particular, by \eqref{decouple}, 
$((G^N_0)^2-((G_\sa)^N_0)^2)(v_\L\otimes 1)=0$. Now applying
Proposition~\ref{azioneDeG} to $(G_\sa)^N_0$ we see that
$$ ((G_\sa)_0^N)^2(v_\L\otimes
1)=\half\left((\mu_{|\h_\sa}+2\rho_{\sa\si},\mu_{|\h_\sa})+\sum_S(\frac{g_S}{12}\dim\sa_S-2g_Sz(\sa_S,\si))\right),
$$ hence 
\begin{align*}
(\mu_{|\h_\sa}+2\rho_{\sa\si},\mu_{|\h_\sa})+&\sum_S(\frac{g_S}{12}\dim\sa_S-2g_Sz(\sa_S,\si))\\&=(\ov\L+2\rho_{\si},\ov\L)+
\frac{g}{12}\dim\g-2gz(\g,\si).
\end{align*}
\end{proof}
\begin{cor}\label{strange}
\begin{equation}\label{masterrho}
\Vert\rho_\si\Vert^2-
\Vert\rho_{\aa\,\sigma}\Vert^2=\frac{g}{12}\dim\g-2gz(\g,\si)-\sum_S(\frac{g_S}{12}\dim\sa_S-2g_Sz(\sa_S,\si)).
\end{equation}
\end{cor}\begin{proof}Plug $\L=-
\rho_\si$ in \eqref{masternok}.
\end{proof}

We now observe that $G_{\g,\sa}$ defines a twisted quantum field on 
$M\otimes F^{\ov\tau}(\ov\p)$. Clearly $M\otimes F^{\ov\tau}(\ov\p)$ 
is a twisted representation of $V^k(\g)\otimes F(\bar \p)$.
Recall from Proposition~\ref{isomorfi} that, if we set $\tilde 
x=x-\half\sum_i:\ov{[x,x_i]}\ov x^i:$ for $x\in\g$, then the map 
$x\mapsto \tilde x$, $\bar x\mapsto \bar x$ induces an isomorphism 
$V^k(\g)\otimes F(\bar \g)\simeq V^{k+g,1}(R^{super})$. In the next 
result we show explicitly that $G_{\g,\aa}$ is in the image of 
$V^k(\g)\otimes F(\bar \p)$ under this isomorphism.
\begin{lemma}\label{gsup}  Let $\{b_i\}$ be an orthonormal basis of 
$\p$. We  have
\begin{equation}\label{GG}G_{\g,\sa}=\sum_i:\tilde b_i\overline b_i:-
\tfrac {1}{6}\sum_{i,j}:\overline{[b_i, b_j]}_\p\overline  
b_i\overline b_j:.
\end{equation}
\end{lemma}
\begin{proof} 
\begin{align}\notag G_\g-G_\sa&=\sum_i:b_i\overline  
b_i:+\tfrac{1}{3}(\sum_{i,j}:\overline{[a_i,
b_j]}\overline a_i\overline  
b_j:\\\label{diff}&+\sum_{i,j}:\overline{[b_i, a_j]}\overline 
b_i\overline 
a_j:+\sum_{i,j}:\overline{[b_i, b_j]}\overline b_i\overline  
b_j:).\end{align}First remark that
$$\sum_{i,j}:\overline{[a_i, b_j]}\overline  a_i\overline 
b_j:=-\sum_{i,j}:\overline{[b_j,a_i]}\overline
a_i\overline  b_j:=\sum_{i,j}:\overline{[b_j,a_i]}\overline 
b_j\overline a_i:,$$where the  second equality
follows from \eqref{qc} since$$\int^0_{-T}[\overline  
a_i\phantom{.}_\l\overline
b_j]d\l=T((a_i,b_j)\va)=0.$$Now, since $[a_i,  b_j]\in \p$, using the 
invariance of the form we get the
following  relation:\begin{align*}&\sum_{i,j}:\,\overline{[a_i, 
b_j]}\overline  a_i\overline b_j:
=\sum_{i,j,k}([a_i,b_j],b_k):\overline b_k\overline  a_i\overline 
b_j: =\\&\sum_{i,j,k}:\overline
b_k([b_j,b_k],a_i)\overline  a_i\overline b_j: =\sum_{j,k}:\overline 
b_k \overline{[b_j,b_k]}_\sa\overline 
b_j:.\end{align*} Finally 
\begin{equation}\label{parziale}
\sum_{j,k}:\overline  b_k \overline{[b_j,b_k]}_\sa\overline 
b_j:=\sum_{i,j}:\overline{[b_i, 
b_j]}_\sa\overline b_i\overline b_j:.
\end{equation} Indeed by \eqref{qc} and 
\eqref{qa}
$$
\sum_{j,k}:\overline b_k :\overline{[b_j,b_k]}_\sa\overline b_j:  
:=\sum_{i,j}::\overline{[b_i,
b_j]}_\sa\overline b_i:\overline  b_j:=\sum_{i,j}:\overline{[b_i, 
b_j]}_\sa:\overline b_i\overline b_j::
$$ (here  we use several times that $([b_i,b_j]_\sa,b_k)=0$.) The 
upshot is that 
\eqref{diff} simplifies to 
\begin{equation}\label{uno} G_\g- G_\sa=\sum_i:b_i\overline 
b_i:+\sum_{i,j}:\overline{[b_i,
b_j]}_\sa\overline  b_i\overline 
b_j:+\frac{1}{3}\sum_{i,j}:\overline{[b_i, b_j]}_\p\overline  
b_i\overline b_j:.
\end{equation} Now we look at the r.h.s. of \eqref{GG}. Using 
\eqref{qa} and \eqref{parziale} we have
\begin{align*}\sum_i :\widetilde  b_i\overline b_i: &=\sum_i: 
b_i\overline b_i: -
\tfrac{1}{2}\sum_{i,j}::\overline{[b_i,x_j]}\overline x_j: \overline 
b_i:\\&= 
\sum_i: b_i\overline b_i: -
\tfrac{1}{2}\sum_{i,j}::\overline{[b_i,a_j]}\overline a_j: \overline 
b_i:-
\tfrac{1}{2}\sum_{i,j}::\overline{[b_i,b_j]}\overline b_j :\overline 
b_i:\\&= 
\sum_i: b_i\overline b_i: +\tfrac{1}{2}\sum_{i,j}::\overline  
b_j\overline{[b_i,b_j]}_\sa: \overline 
b_i:+\tfrac{1}{2}\sum_{i,j}:\overline{[b_i,b_j]}\overline b_i 
\overline b_j: 
\\&= \sum_i: b_i\overline b_i: 
+\tfrac{1}{2}\sum_{i,j}:\overline{[b_i,b_j]}_\sa 
\overline b_i\overline  
b_j:+\tfrac{1}{2}\sum_{i,j}:\overline{[b_i,b_j]}\overline b_i 
\overline b_j: 
\\&= \sum_i: b_i\overline b_i: 
+\sum_{i,j}:\overline{[b_i,b_j]}_\sa\overline  b_i \overline
b_j:+\tfrac{1}{2}\sum_{i,j}:\overline{[b_i,b_j]}_\p\overline b_i 
\overline b_j:.
\end{align*}  hence the desired  equality \eqref{GG}.
\end{proof} Note that formula~\eqref{GG} specializes to  
\eqref{diractilde} when $\sa=0$.

\begin{rem}\label{primorem}
 Set 
$$G=\frac{G_{\g,\sa}}{\sqrt{k+g}},\quad L=\frac{1}{k+g}(\tilde 
L^\g-\tilde L^\sa)-(L^{\ov\g}-L^{\ov\sa}),$$ 
where $\tilde L^\g$ is the image of $L^\g\otimes |0\rangle$ in the 
isomorphism
$V^k(\g)\otimes F(\overline \g)\cong V^{k+g,1}(R^{super})$ of 
Proposition \ref{isomorfi},
$L^\g$ is defined in \eqref{sugaw} and $L^{\ov\g}$ is defined in 
\eqref{da}.
A direct
computation (cf. \cite{KacD}) shows that $G$ and $L$ form a 
Neveu-Schwarz Lie conformal superalgebra 
$$
NS=\C[T]L+\C[T]G+\C C,$$
$$[L_\l L]=(T+2\l)L+\frac{\l^3}{12}C,\quad [L_\l 
G]=(T+\frac{3}{2}\l)G,\quad [G_\l G]=2L+\frac{\l^2}{3}C$$
with
central charge
\begin{equation}\label{cc}
C=\half\dim(\p)-\sum_S(1-\frac{g_S}{k+g})\dim(\sa_S).
\end{equation} Set $k=0$. Then \eqref{cc} vanishes if and only if the 
pair $(\g,\sa)$ is symmetric, i.e. $\sa$ is the
algebra of fixed points of an involution of $\g$. 
Indeed, if $(\g,\sa)$ is symmetric, choosing $\s$ 
as the involutive automorphism that fixes $\aa$, then 
$\g=\g^{\bar0}\oplus\g^{\ov{1/2}}$
and $\aa=\g^{\bar0}$, $\p=\g^{\ov{1/2}}$. This implies that 
$\rho_\s=\rho_{\aa\s}=\rho_0$
and that $z(\aa_s,\s)=0$ while $z(\g,\s)=\frac{1}{16}\dim\p$. 
Substituting in 
\eqref{masterrho} we find that $C=0$ in \eqref{cc}.

The reverse implication is a consequence of the ``Symmetric Space 
Theorem" by Goddard, Nahm,
Olive \cite{GNO}. We can also derive it from our previous discussion. 
Indeed choose $\s=I$. Then the vanishing of the
central charge together with the fact that $F^{\bar\tau}(\ov\p)$ is a 
unitarizable representation of the
Ramond Lie superalgebra implies the vanishing of $G$ and $L$. In 
particular $(G_{\g,\sa})^{\C\otimes F^{\ov\tau}(\ov\p)}_0=0$. 
Writing   $G_{\g,\sa}$ explicitly as in  Lemma~\ref{gsup}, 
$$
0=(G_{\g,\sa})^{\C\otimes 
F^{\ov\tau}(\ov\p)}_0=-\frac{1}{6}\sum_{i,j}:\ov{[b_i,b_j]}_\p\bar 
b^i\bar b^j:^{\C\otimes F^{\ov\tau}(\ov\p)}_0.
$$
It is easy to check, using Wick's formula, that, if $b,b'\in\p$, then
$$
-\frac{1}{6}\sum_{i,j}[:\ov{[b_i,b_j]}_\p\bar b^i\bar b^j:_\l\bar 
b]=-\half\sum_i:\ov{[b,b_i]}_\p\bar b^i:
$$
and 
$$
-\half\sum_i[:\ov{[b,b_i]}_\p\bar b^i:_\l \bar b']=\ov{[b,b']}_\p. 
$$
This implies that, if we apply $(G_{\g,\sa})^{\C\otimes 
F^{\ov\tau}(\ov\p)}_0$ to $\bar b^{\C\otimes 
F^{\ov\tau}(\ov\p)}_r\ov{b'_s}^{\C\otimes F^{\ov\tau}(\ov\p)}\cdot 
(1\otimes 1)$,
then $(\ov{[b,b']}_\p)_{r+s}\cdot(1\otimes 1)=0$ for any $r,s$. This 
in turns implies that $[b,b']_\p=0$, hence $[\p,\p]\subset \aa$. Therefore the 
pair $(\g,\aa)$ is symmetric. 
\end{rem}

\vskip 5pt

Let $\L_0$ be the element of $\ha_0^*$ defined  setting 
$\L_0(d)=\L_0(\h_0)=0$  and $\L_0(K)=1$.
Define also 
$\d\in\ha_0^*$  setting $\d(d)=1$ and $\d(\h_0)=\d(K)=0$. Set 
$\ha_\sa=\h_\sa\oplus\C d_\sa\oplus \sum_S\C K_S$.  Let $\d_\sa$  be 
the analogous element of
$\ha_\sa^*$ defined by  $\d_\sa(K_S)=0$ for all $S$, 
$\d_\sa(\h_\sa)=0,\,\d_\sa(d_\sa)=1$.

 Extend $(\cdot,\cdot)$ to all of $\ha_0^*$ by setting $(\L_0,\d)=1$ 
and 
$(\L_0,\L_0)=(\d,\d)=(\d,\h_0)=(\L_0,\h_0)=0$. Set 
\begin{equation}\label{rhoaff}
\rhat_\si=\rho_\si+g\L_0,\quad\rhat_{
\sa\,\si}=\rho_{\sa\,\si}+\sum_Sg_S\L_0^S.
\end{equation}  Then, writing 
$\L=\ov\L+k\L_0+\L(d)\d$,  we see  that
\begin{equation}\label{lambdaquadro} (\ov \L+2\rho_\si,\ov 
\L)+2(k+g)\L(d)=\Vert \L+\widehat\rho_\si\Vert^2-
\Vert\widehat\rho_\si\Vert^2.
\end{equation} 
\vskip5pt

Consider the map 
$\varphi_\sa:\h_0\oplus(\sum_S\C K_S)\oplus\C d_\sa\to\ha_0$ 
\begin{equation}\label{vp}
\varphi_\sa(h)=h\ 
\text{if $h\in\h_0$},\quad\varphi_\sa(d_\sa)=d,\quad 
\varphi_\sa(K_S)=K\ 
\text{for all $S$}.\end{equation} Since $\varphi_\sa$ is onto, 
$\varphi_\sa^*$ is an  embedding of
$\ha_0^*$ into $(\h_0\oplus(\sum_S\C K_S)\oplus\C d_\sa)^*$. We can 
therefore view $(\cdot,\cdot)$ 
as a bilinear form on
$\varphi_\sa^*(\ha_0^*)$. If $\mu\in\h^*_\sa$ we let 
$\mu_0$ be its extension to $\h_0$ defined by setting 
$\mu_0(\h_\p)=0$. In this way we can view
$\ha_\sa^*$ as a subspace of  $(\h_0\oplus(\sum_S\C K_S)\oplus\C 
d_\sa)^*$.

In view of Lemma~\ref{dzero},  we can define  the action of $d_\sa$ 
on $F^{\ov\tau}(\ov\p)$ by letting
it act as 
\begin{equation}\label{azionednormalizzata} (L^{\ov\p})^{\bar 
\tau}_{(1)}-(z(\g,\si)-z(\sa,\si)-\tfrac{1}{16}\dim\p)I_{F^{\ov\tau}(\ov\p)}.  
\end{equation} With this normalization we have that $d_\sa\cdot 1=0$.
 The reason for this particular choice will be clear  in  
Proposition~\ref{ker}. We can then let $d_\sa$ act
on
 $M\otimes F^{\ov\tau}(\ov\p)$  via $d^M\otimes I + I\otimes  
d^{\ov\tau}_\sa$. 
Given $\nu\in
(\ha_\sa)^*$, we denote by  $(M\otimes F^{\ov\tau}(\ov\p))_\nu$ its 
$\nu$-weight space.
 If   
$(M\otimes F^{\ov\tau}(\ov\p))_\nu\ne0$, then $\nu+\rhat_{\sa \si}\in 
\varphi^*_\sa(\ha_0^*)$ (indeed 
$\nu+\rhat_{\sa\si}=\varphi_\sa^*((\nu_{|\h_\sa})_0+(\rho_{\sa\si})_0+\nu(d)\d+ 
(k+g)\L_0)$) and 
\begin{equation}\label{nuquadro}(\nu_{|\h_\sa}+2\rho_{\sa\si},\nu_{|\h_\sa})+2 
(k+g)\nu(d_\sa)=\Vert
\nu+\rhat_{\sa \si}\Vert^2-
\Vert\rho_{\sa\si}\Vert^2.\end{equation}

\noindent If $M$ is a highest weight module for $\widehat L(\g,\s)$ 
set $N'=M\otimes F^{\ov\tau}(\ov\p)$. In light of Lemma~\ref{gsup}, 
we can consider the operator
$(G_{\g,\sa})^{N'}_0$. 

\begin{prop}\label{ker}  If $v\in  (M\otimes 
F^{\ov\tau}(\ov\p))_\nu,\,{\n'_\sa}\cdot v=0$   and
$\rhat_\si,\,\rhat_{\sa \si}$ are as in \eqref{rhoaff} 
then\begin{align*}((G_{\g,\sa})^{N'}_0)^2(v)=&\half(||\L+\rhat_\si||^2- 
||\nu+\widehat\rho_{\sa\si}||^2)v.
\end{align*}
\end{prop}

\begin{proof} Clearly $(G_{\g,\sa})^{N'}_0\otimes 
I_{F^{\ov\tau}(\ov\sa)}=(G_{\g,\sa})^N_0$, so applying
\eqref{gzeroquadro}, Lemma~\ref{xixi}, Lemma~\ref{xixibis}, and using 
the fact that
$L^{\ov\g}-L^{\ov\sa}=L^{\ov\p}$, we obtain that 
\begin{align*} ((G_{\g,\sa})^{N'}_0)^2(v)&=\left(\half(\ov 
\L+2\rho_\si,\ov 
\L)+kz(\g,\si)+(k+g)\L(d)\right)v\\&- (k+g)(d^M\otimes I_{F^{\ov\tau}(\ov\p)} + I_M\otimes
(L^{\ov\p})^{\bar\tau}_{(1)})(v)\\&-\left(\half(\nu_{|\h_\sa}+2\rho_{\sa\si},\nu_{|\h_\sa})+
\sum_S(k+g-g_S)z(\sa_S,\si)\right)v\\&-
\tfrac{1}{16}\left((k+\frac{g}{3})\dim\g-\sum_S(k+g-
\frac{2g_S}{3})\dim\sa_S\right)v.
\end{align*} By our normalization of the action of $d_\sa$ on 
$M\otimes F^{\ov\tau}(\ov\p)$ we have
that
$$ (d^M\otimes I_{F^{\ov\tau}(\ov\p)} +I_M\otimes 
(L^{\ov\p})^{\bar\tau}_{(1)})(v)=(\nu(d_\sa)+z(\g,\si)-z(\sa,\si)-\tfrac{1}{16}\dim\p)v,
$$ hence
\begin{align*} ((G_{\g,\sa})^{N'}_0)^2(v)&=\left(\half(\ov 
\L+2\rho_\si,\ov 
\L)+kz(\g,\si)+(k+g)\L(d)\right)v\\&-
(k+g)(\nu(d_\sa)+z(\g,\si)-z(\sa,\si)-\tfrac{1}{16}\dim\p)v\\&-\left(\half(\nu_{|\h_\sa}+2\rho_{\sa\si},\nu_{|\h_\sa})+
\sum_S(k+g-g_S)z(\sa_S,\si)\right)v\\&-
\tfrac{1}{16}\left((k+\frac{g}{3})\dim\g-\sum_S(k+g-
\frac{2g_S}{3})\dim\sa_S\right)v
\end{align*} thus
\begin{align*} ((G_{\g,\sa})^{N'}_0)^2(v)&=\half(\Vert 
\L+\rhat_\si\Vert^2-\Vert\nu+\rhat_{\sa\si}\Vert^2)v-
\half(\Vert\rho_\si\Vert^2-\Vert\rho_{\sa\si}\Vert^2)v\\ &+\half
(\frac{g\dim\g}{12}-2gz(\g,\si)-\sum_S(\frac{g_S\dim\sa_S}{12}-2g_Sz(\sa_S,\si)))v.
\end{align*} Applying Corollary~\ref{strange} we get the result.
\end{proof}

\section{Multiplets of representations}
\subsection{Kostant's theorem on mutiplets in the twisted\\ affine 
setting}
First of all we study $F^{\ov\tau}(\ov\g)$ viewed as a $\widehat 
L(\g,\si)$-module.  The action of
$\widehat L(\g,\si)$ on $F^{\ov\tau}(\ov\g)$ is obtained by letting 
$t^j\otimes x$ act via
$\theta_\g(x)^{\ov\tau}_{(j)}$ where
$\theta_\g(x)=x-\tilde x=\half\sum_i:\overline{[x,x_i]}\ov x_i:$. In 
our framework this action corresponds
to the pair $(\g\oplus \g, \g)$ where $\g$ embeds diagonally in 
$\g\oplus\g$ and the automorphism of
$\g\oplus\g$ is
$\si\oplus\si$.\par
Recall from \cite[Prop.~6.3]{Kac} that the choice of a set of 
positive roots for $\D_0$ induces the choice
of a set of positive roots $\Dap$ for $\widehat L(\g,\si)$.  Let 
$\Pia=\{\a_0,\cdots,\a_n\}$ denote the
corresponding set of simple roots.

\begin{lemma}\label{pesispin}  $F^{\ov\tau}(\ov\g)$ is completely 
reducible as a
$\ha_0$-module and the set of weights of $F^{\ov\tau}(\ov\g)$ is 
$\rhat_\si-S$ where
\begin{align}\notag S=&\{\l\in\ha_0^*\mid 
\l=\sum_{\a\in\Dap}n(\a)\a,\ \text{where all but  finitely many
$n(\a)$ are zero}\\\label{esse} &\text{and each $n(\a)\le \mathrm{ 
mult}\,\a$}\}.
\end{align}
\end{lemma}
\begin{proof}Fix a basis $\{h_i\}$ of $\h_0$ such that 
$(h_i,h_{n-j+1})=\d_{ij}$. Choose for any 
$\a\in\h^*_0$ a basis $\{x_{i\a}\}$ of $\g_\a$. If $\gamma\in-\Dp_0$, 
let  $y_\gamma$ be a root
vector in
$\g^{\ov 0}_{\gamma}$.  Fix any order in $\D_0$ and in the set 
$\{(i,\a)\mid \g_\a\ne0,\,1\le i\le
\dim\g_\a\}$. Then a basis of
$F^{\ov\tau}(\ov\g)$ is given by the set of vectors
\begin{equation}\label{vectorbasis} (\ov 
h_{i_1})^{\ov\tau}_{(-\half)}\cdots (\ov
h_{i_r})^{\ov\tau}_{(-\half)}(\ov 
y_{\gamma_1})^{\ov\tau}_{(-\half)}\cdots (\ov
y_{\gamma_s})^{\ov\tau}_{(-\half)}(\ov 
x_{h_1\be_1})^{\ov\tau}_{(j_1)}\cdots (\ov
x_{h_t\be_t})^{\ov\tau}_{(j_t)}( 1)
\end{equation} where
$
\left[\frac{n}{2}\right]<i_1<i_2<\cdots<i_r\le n$, 
$\gamma_1<\cdots<\gamma_s$,
$j_p<-\half$ for any $p$, and 
$ j_1\le\cdots\le j_t$  with $(h_p,\be_p)<(h_{p+1},\be_{p+1})$ when 
$j_p=j_{p+1}$. Moreover
$\be_{j_p}\in\D_{j_p+\half}$.

If $h\in\h_0$, by \eqref{thetalambdabar}, 
$[(\theta_\g(h))^{\ov\tau}_{(0)},(\ov
a)^{\ov\tau}_{(n)}]=(\overline{[h,a]})^{\ov\tau}_{(n)}$. By 
Lemma~\ref{lemmamu} the vector in
\eqref{vectorbasis} is therefore a weight vector  for $\h_0$ with 
weight
$$
\rho_\si+\sum_i \gamma_i+\sum_p \be_{j_p}.
$$

By \eqref{azionednormalizzata} the action of $d$ is given by
$(L^{\ov\g})^{\ov\tau}_{(1)}-(z(\g,\si)-\tfrac{1}{16}\dim\g)I_{F^{\ov \tau}(\ov\g)}$ and, 
by \eqref{isprimary},
$[(L^{\ov\g})^{\ov\tau}_{(1)},(\ov x)^{\ov\tau}_{(n)}]=(n+\half)(\ov 
x)^{\ov\tau}_{(n)}$. Since $d\cdot
1=0$ we obtain that the vector in \eqref{vectorbasis} is an 
eigenvector for the action of $d$ with
eigenvalue 
$
\sum_p(j_p+\half).
$

Finally, by Lemma~\ref{lemmamu}, $K$ acts by $2g-g=g$. Summarizing we 
have that the vector in
\eqref{vectorbasis} is a weight vector for $\ha_0$ whose weight is
$$
\rho_\si+g\L_0+\sum_i \gamma_i+\sum_p 
((j_p+\half)\d+\be_{j_p})=\rhat_\si-\eta,
$$ with $\eta=-\sum_i \gamma_i-\sum_p ((j_p+\half)\d+\be_{j_p})$. 
Since $(j_p+\half)\d+\be_{j_p}$ can
only occur
$\dim\g_{\be_{j_p}}$ times in the sum, we have that $\eta\in S$.
\end{proof}

\begin{lemma}\label{fazero} Choose a simple root $\a_i=s_i\d+\ov\a_i$ 
for $\widehat L(\g,\si)$ and
$x_i\in\g^{-s_i}_{-\ov\a_i}$. Then
$$ (G_{\g,{\h_0}})^{N'}_0(\ov x_{i})^{\ov\tau}_{(-s_i-\half)}(1)=0.
$$
\end{lemma}
\begin{proof} We start by computing $[{G_{\g,{\h_0}}}_\l\ov x_{i}]$. 
By \eqref{abarG} and
skewsymmetry of the $\l$-bracket, we have $[{G_{\g}}_\l\ov 
x_{i}]=x_i$. On the other hand, since
${\h_0}$ is commutative, $G_{\h_0}=\sum_j:h_j\ov h_j:$, where 
$\{h_j\}$ is an orthonormal basis of
$\h_0$. It follows from Wick's formula and skewsymmetry that
$[{G_{\h_0}}_\l \ov x_i]=\sum_j:\overline{[x_i,h_j]}\ov h_j:$. Thus
$$ [{G_{\g,{\h_0}}}_\l\ov x_{i}]=x_i-\sum_j:\overline{[x_i,h_j]}\ov 
h_j:,
$$ or, by writing $x_i=\tilde x_i+(x_i-\tilde x_i)$,
$$ [{G_{\g,{\h_0}}}_\l\ov x_{i}]=\tilde x_i+\half 
\sum_t:\overline{[x_i,y_t]}\ov
y^t:-\sum_j:\overline{[x_i,h_j]}\ov h_j:,
$$ where, as usual, $\{y_t\}$ and $\{y^t\}$ are a pair of dual basis 
for $\g$. Choosing a suitable  basis
$\{y_i^{\a,j}\}$ of $\p\cap\g^j_\a$, we can assume that
$(y_i^{\a,r},y_j^{-\a,-r})=\d_{ij}$. We choose as basis of $\g$ the 
set $(\cup_{\a,t,r}\{y_t^{\a,r}\})\cup
\{h_j\}$. With this particular choice of basis we can write
\begin{equation}\label{perkoszul} [{G_{\g,{\h_0}}}_\l\ov 
x_{i}]=\tilde x_i+\half
\sum_{\a,t,r}:\overline{[x_i,y_t^{\a,r}]}\ov 
y_t^{-\a,-r}:-\half\sum_j:\overline{[x_i,h_j]}\ov h_j:.
\end{equation} Since $\a_i$ is a real root we have that 
$\dim\g^{s_i}_{\a_i}=1$, hence
$$ [{G_{\g,{\h_0}}}_\l\ov x_{i}]=\tilde x_i+\half 
\sum_{(r,\a)\ne(s_i,\a_i)}:\overline{[x_i,y_t^{\a,r}]}\ov
y_t^{-\a,-r}:.
$$  We have used the fact that $:\overline{[x_i,y_1^{(\a_i,s_i)}]}\ov 
y_1^{-\a_i,-s_i}:=-:\ov h_{\a_i}\ov
x_i:=\sum_j:\overline{[x_i,h_j]}\ov h_j:$. In particular 
$$[(G_{\g,{\h_0}})^{N'}_0,(\ov x_{i})^{\ov\tau}_{(-s_i-\half)}]=\half
\sum_{(r,\a)\ne(s_i,\a_i)}:\overline{[x_i,y_t^{\a,r}]}\ov 
y_t^{-\a,-r}:^{\ov\tau}_{(-s_i)}.
$$ 
 Since $\h_\p=0$, by applying \eqref{Grelativo}, we have that 
$(G_{\g,{\h_0}})^{N'}_0(1)=0$, thus we
are left with showing that
$$
\half 
\sum_t\sum_{(r,\a)\ne(s_i,\a_i)}\left(:\overline{[x_i,y_t^{\a,r}]}\ov
y_t^{-\a,-r}:\right)^{\ov\tau}_{(-s_i)}(1)=0.
$$ Since $[\overline{[x_i,y_t^{\a,r}]}_\l\ov y_t^{-\a,-r}]=0$, we 
have that
$$ Y^{\ov\tau}(:\overline{[x_i,y_t^{\a,r}]}\ov y_t^{-\a,-r}:,z)=
:Y^{\ov\tau}(\overline{[x_i,y_t^{\a,r}]},z)Y^{\ov\tau}(\ov 
y_t^{-\a,-r},z):.
$$ By expanding the r.h.s. of the previous equation and picking the 
coefficient of $z^{s_i-1}$ we find that
\begin{align*}
\left(:\overline{[x_i,y_t^{\a,r}]}\ov
y_t^{-\a,-r}:\right)^{\ov\tau}_{(-s_i)}&=\sum_{n<0}\overline{[x_i,y_t^{\a,r}]}^{\ov\tau}_{(-s_i+r+\half+n)}(\ov
y_t^{-\a,-r})^{\ov\tau}_{(-n-r-1-\half)}\\&-\sum_{n\ge0}(\ov
y_t^{-\a,-r})^{\ov\tau}_{(-n-r-1-\half)}\overline{[x_i,y_t^{\a,r}]}^{\ov\tau}_{(-s_i+r+\half+n)}.
\end{align*}
 If $r\ne0$, we can choose $r\in(-1,0)$ so $(\ov 
y_t^{-\a,-r})^{\ov\tau}_{(-n-r-1-\half)}(1)=0$ for $n<0$,
and 
 $\overline{[x_i,y_t^{\a,r}]}^{\ov\tau}_{(-s_i+r+\half+n)}(1)=0$ for 
$n>0$. It follows that
 $$
 \left(:\overline{[x_i,y_t^{\a,r}]}\ov 
y_t^{-\a,-r}:\right)^{\ov\tau}_{(-s_i)}(1)=-(\ov
y_t^{-\a,-r})^{\ov\tau}_{(-r-1-\half)}\overline{[x_i,y_t^{\a,r}]}^{\ov\tau}_{(-s_i+r+\half)}(1).
$$ Since $(r+1)\d+\a$ is a positive root and $\a_i$ is simple, we 
have that either $[x_i,y_t^{\a,r}]=0$ or
$(r+1-s_i)\d+\a-\ov\a_i$ is still positive. In both cases
$\overline{[x_i,y_t^{\a,r}]}^{\ov\tau}_{(-s_i+r+\half)}(1)=0$. 

It remains to deal with the case $r=0$, i.e. $\a\in\D_0$. If 
$\a\in\Dp_0$ then either $[x_i,y_t^{\a,0}]=0$
or $s_i=0$, otherwise
$-s_i\d+\a-\ov\a_i$ would be a positive root. In both cases
$\overline{[x_i,y_t^{\a,0}]}^{\ov\tau}_{(-s_i+\half)}(1)=0$.
 If $\a\in-\Dp_0$ then 
$$ (\ov
y_t^{-\a,0})^{\ov\tau}_{(-\half)}\overline{[x_i,y_t^{\a,0}]}^{\ov\tau}_{(-s_i-\half)}=-\overline{[x_i,y_t^{\a,0}]}^{\ov\tau}_{(-s_i-\half)}(\ov
y_t^{-\a,0})^{\ov\tau}_{(-\half)}
$$ and $(\ov y_t^{-\a,0})^{\ov\tau}_{(-\half)}(1)=0$.
\end{proof}

\begin{lemma}\label{rhorho}
$$2\frac{(\rhat_\si,\a_i)}{(\a_i,\a_i)}=1.
$$
\end{lemma}
\begin{proof}
 Suppose that $\a_i=s_i\d+\ov\a_i$. Observe that
$(\ov x_{-\ov\a_i})^{\ov\tau}_{(-s_i-\half)}\cdot 1\in F^{\ov 
\tau}(\ov\p)$ and that this vector is annihilated by $\n'_{\h_0}$. From Lemma
\ref{fazero} and   Proposition
\ref{ker} (applied with $\sa=\h_0,\,M=\C,\,\L=0$) we deduce that
$||\rhat_\si||^2=||\nu+\rhat_{{\h_0}\,\si}||^2$, $\nu$ being the 
weight of $(\ov
x_{-\ov\a_i})^{\ov\tau}_{(-s_i-\half)}\cdot 1$. 

Since $\h_0$ is commutative, $(\tilde h)_{\h_0}=h$ hence 
$(\tilde h)_{\h_0}-\tilde h=\theta_\g(h)$.
 Observe that $(\ov
x_{-\ov\a_i})^{\ov\tau}_{(-s_i-\half)}\cdot 1$ is a vector of the 
form given in \eqref{vectorbasis}. Thus, arguing as in 
Lemma~\ref{pesispin}, its weight is 
$\nu=\rhat_\si-s_i\d-\ov\a_i$. We therefore obtain
$$||\rho_\si||^2=||\rho_\si||^2+||\a_i||^2-2(\rho_\si\,\ov\a_i)-2g 
s_i$$ or
\begin{equation}\label{rhoalfai} 2(\rho_\si,\ov\a_i)=(\a_i,\a_i)-2g 
s_i.
\end{equation} It follows that
$$ 
2\frac{(\rhat_\si,\a_i)}{(\a_i,\a_i)}=1+\frac{2gs_i}{(\a_i,\a_i)}-\frac{2gs_i}{(\a_i,\a_i)}
$$ hence the result.
\end{proof}

In the rest of this section we assume that $\sa^{\ov 0}$ is an equal rank 
subalgebra of $\g^{\ov 0}$. If  $\a$ is a
root of 
 $\widehat L( \sa,\si)$, then $\a=k\d_\sa+\ov\a$. Thus
$\a=\varphi^*_\sa(k\d+\ov\a)$ ($\varphi_\sa$ is as in \eqref{vp}) and 
a root vector in
$\widehat L(\sa,\si)$ for
$\a$ is a root vector for
$k\d+\ov\a$ in $\widehat L(\g,\si)$. It follows that 
$(\varphi_\sa^*)^{-1}$ maps the set of roots
$\Da(\sa)$ of $\widehat L( \sa,\si)$ into the set of roots $\Da$ of 
$\widehat L(\g,\si)$. For simplicity we
identify $\Da(\sa)$ and
$(\varphi_\sa^*)^{-1}(\Da(\sa))$, thus viewing  $\Da(\sa)$ as a 
subset of $\Da$.  Let $\Wa$ be the Weyl 
group of 
$\widehat L(\g,\sigma)$ and let $\Wa_\sa$ be the subgroup generated 
by the reflections $s_\a$ with
$\a\in\Da(\sa)$. Denote by 
$\widehat W'$  the set of minimal right coset representatives of 
$\Wa_\sa$ in $\Wa$.

If $\L\in\ha_0^*$ is dominant and integral, we let $L(\L)$ denote the 
irreducible highest weight module
for $\widehat L(\g,\si)$ with highest weight $\L$.  We set 
$\Dap_\sa=\Dap\cap\Da(\sa)$. If $\xi\in(\h_0\oplus(\sum_S\C 
K_S)\oplus\C d_\sa)^*$ is dominant and
integral, denote by
$V(\xi)$ the irreducible $\widehat L(\sa,\si)$-module with highest 
weight $\xi$.
\par The following result  is a generalization of  Theorem 16 in 
Landweber's paper
 \cite{land}, where the case $\sigma=I_\g$ is treated.
\vskip5pt
\begin{theorem}\label{multiplets}  Assume that $\sa^{\ov 0}$ is an 
equal rank subalgebra of $\g^{\ov 0}$
and that 
$\L$ is a dominant integral weight for $\widehat L(\g,\si)$. Set
$X=L(\L)\otimes  F^{\ov\tau}(\ov\p)$. Then 
\begin{equation}\label{formulona}Ker\,(G_{\g,\sa})_0^{X}=\bigoplus_{w\in
\widehat 
W'}V(\varphi^*_\sa(w(\L+\rhat_\si))-\rhat_{\sa\si}).\end{equation}
\end{theorem}
\begin{proof}  Suppose that $V(\xi)$ occurs in 
$Ker\,((G_{\g,\aa})_0^X)^2$.
  Then $\xi=\gamma+\beta$  where $\gamma$ is a weight of 
$F^{\ov\tau}(\ov\p)$ and $\beta$ is a
weight of $L(\L)$. By  Lemma \ref{pesispin}, we know the form of the 
weights of $F^{\ov\tau}(\ov\g)$.
Since we are assuming that $\rank(\sa^{\ov 0})=\rank(\g^{\ov 0})$  the 
weights of $F^{\ov\tau}(\ov\p)\otimes 1\subset F^{\bar\tau}(\g)$ are also of
this form. Hence we can write
$\xi+\rhat_{\sa\,\si}=\varphi^*_\sa(-
\nu+\beta+\rhat_\si)$ where $\nu\in S$ (cf. \eqref{esse}). By 
Proposition \ref{ker} we have that  $||-
\nu+\beta+\rhat_\si||^2=||\L+\rhat_\si||^2$.  Lemma \ref{rhorho} 
tells us that $\rhat_\si$ is what is usually
denoted by $\rho$ for
$\widehat L(\g,\si)$. Hence we  can use \cite[Lemma 3.2.4]{Kumar}  to 
deduce the existence of 
 $w\in\Wa$ such that 
$\xi=\varphi^*_\sa(w(\L+\rhat_\si))-\rhat_{\sa\,\si}$. We claim that 
$w\in \widehat W'$ and
that for any 
$w\in  \widehat W'$ the corresponding submodule occurs with 
multiplicity one in 
$Ker\,((G_{\g,\aa})_0^X)^2$. The proof of all these statements can be 
done  along the lines of Kostant's argument
in the finite dimensional case, as  extended to the affine case by 
Kumar in  \cite[Theorem
3.2.7]{Kumar}. There is   only one difference with Kumar's setting:  
$\widehat W_{\sa}$ is a reflection
subgroup of
$\widehat W$ and not, in general, a standard parabolic subgroup. But  
Kumar's proof relies on a
description of $\widehat W'$ (see \cite[Exercise  1.3.E]{Kumar}) 
which holds in our weaker
hypothesis too. This concludes the proof since, by Proposition 
\ref{unitarity} below,
$(G_{\g,\aa})_0^X$ is self-adjoint (in our hypothesis), hence
$Ker\,(G_{\g,\aa})_0^X=Ker\,((G_{\g,\aa})_0^X)^2$.  
\end{proof}
\subsection{Applications}\label{applications}
We want to discuss in our setting some consequences of 
Theorem~\ref{multiplets} which
are the analogues of Theorems 4.17 and 4.24 of \cite{Kold} in the 
finite dimensional
case. As in the finite dimensional case, we name  {\it  multiplet}
the set of $\widehat L(\aa,\s)$-modules occurring in the 
decomposition \eqref{formulona}. As discussed in the Introduction,
these were discovered in \cite{GKRS}  (in the finite dimensional 
equal rank case), where it is shown that
they possess remarkable properties. First, the Casimir element acts 
by the same scalar on all the
representations in the multiplets. This fact has a direct analogue in 
the affine case. Indeed 
$C(\aa)=(L^\aa_0)^{X}+(k+g)d^{X}_\aa$ can be considered as  (one 
half of the) Casimir element  
for $\widehat L(\aa,\s)$ acting on $X=L(\L)\otimes 
F^{\bar\tau}(\bar\p)$: this follows e.g. from the formula displayed 
in 
\cite[Exercise 7.16]{Kac}, noting that in our context the central 
elements $K_S$
specialize to the levels $k+g-g_S$.
We shall deduce the above remarkable property by a formula for the square of the Dirac operator 
acting on $L(\L)\otimes 
F^{\bar\tau}(\bar\p)$, which holds in the framework of Section 4 (i.e., $(\g,\aa)$ a reductive pair, $\L$ any weight of $\widehat L(\g,\si)$). The following result is a
twisted affine analog of  \cite[Theorem 2.13]{Kold}.
\begin{prop}\label{quadro} Set 
$N=M\otimes F$, where $M$ is any level $k$ highest 
weight module  for $\widehat L(\g,\s)$ and $F=F^{\ov\tau}(\ov \p)$. Set $C(\g)=(L^\g_0)^{M}+(k+g)d^{M}$, 
$C(\aa)=(L^\aa_0)^{N}+(k+g)d^{N}_\aa.$ Then 
\begin{equation}
\label{masterG}((G_{\g,\sa})^N_0)^2=C(\g)\otimes I_F-C(\aa)+\left[\half(\Vert\rho_\s\Vert^2-\Vert\rho_{\aa\s}\Vert^2)+c(k)\right]I_N,\end{equation}
 where $c(k)=-kz(\g,\s)+\sum_S(k+g-g_S)z(\aa_S,\s)$. \end{prop}
\begin{proof}ÊCombine formulas \eqref{gzeroquadro}, \eqref{azionednormalizzata}, and \eqref{masterrho}.\end{proof}
\begin{cor}\label{casimiraref} Under the hypothesis of Theorem \ref{multiplets}, $C(\aa)$ acts on all $\widehat
L(\aa,\sigma)$-modules 
$V(\varphi^*_\sa(w(\L+\rhat_\si))-\rhat_{\sa\si})$ of the multiplet
\eqref{formulona} by
the scalar
\begin{equation}\label{casimira}
\half\left(\Vert\L+\rhat_\si\Vert^2-\Vert\rho_{\sa\si}\Vert^2\right)+\sum_S(k+g-g_S)z(\aa_S,\sigma).
\end{equation}
\end{cor}
\begin{proof} By Theorem \ref{multiplets} and formula \eqref{masterG}, $C(\aa)$ acts as
$$
C(\g)\otimes I_F+\left[\half(\Vert\rho_\s\Vert^2-\Vert\rho_{\aa\s}\Vert^2)+c(k)\right]I_N
$$
on the multiplet. Combine Lemma \ref{xixibis} and formula \eqref{lambdaquadro} to compute the action of $C(\g)$.
\end{proof}
\vskip10pt
The second property discovered in \cite{GKRS} involves the dimensions 
of the representations in 
the multiplets. In the finite dimensional case the multiplets are 
canonically indexed by the 
set $W'$ of
minimal length  representatives of (right) cosets of the Weyl group of $\aa$ 
(a reductive subalgebra 
of $\g$ of the same rank) in the Weyl 
group $W$ of $\g$.
 If we
let $V_w$ denote the $\g$-module indexed by $w\in W'$, then  
\begin{equation}\label{altdim}
\sum_{w\in  W'} (-1)^{\ell(w)}\dim V_w=0,
\end{equation}
where $\ell(\cdot)$ is the length function on $W$.
 
Obviously this result cannot hold true in the affine case for the 
representations involved
 are infinite dimensional. However in Proposition~\ref{asdim} below 
we obtain an analog of \eqref{altdim}
involving the asymptotic dimensions of the representations
$V(\varphi^*_\sa(w(\L+\rhat_\si))-\rhat_{\sa\si})$ as defined in 
\cite[Ch. 13]{Kac}. \par\noindent
\vskip5pt
To prove this fact we need several preliminary considerations.
Let $V$ be  a complex vector space endowed with a symmetric bilinear 
form $(\ ,\ )$. Fix $\sigma\in O(V)$ and assume that $\sigma$ is 
diagonalizable with modulus $1$ eigenvalues. Suppose also that the 
set of $\sigma$-fixed points is even dimensional. Set $\tilde V = V 
\oplus \C$ and extend $(\ ,\ )$ to $\tilde V$ by setting 
$(v,1)=0,\,(1,1)=1$. Then
$so(V)$ embeds in $so(\tilde V)$ and $(so(\tilde V),so( V))$  is a 
reductive pair. Indeed, for $v\in V$ define $X_v\in so(\tilde V)$ 
by $X_v(w+c)=cv-(v,w)$. If we endow $so(\tilde V)$ with the invariant 
form
$\langle X,Y\rangle =\half tr(XY)$, then we have that 
$so(V)^\perp=\{X_v\mid v\in V\}$. Note that, if $A\in so(V)$ and 
$v\in V$, then $[A,X_v]=X_{A(v)}$.
Thus, identifying $V$ with $\{X_v\mid v\in V\}$, we see that the 
adjoint action of $so(V)$ on its orthogonal complement gets 
identified with the natural action of $so(V)$ on $V$.\par
Extend $\si$ to an automorphism of $\tilde V$ by letting $\si(1)=1$. 
Then $\si X_v \si^{-1}=X_{\si(v)}$. Let $\ov V$ be the space $V$ 
viewed as an odd space, and set $\ov\tau=-\si$. By applying
our machinery to the reductive pair $(so(\tilde V),so( V))$ we can 
turn
$F^{\bar\tau}(\ov V)$ into a $\widehat L(so(V),Ad(\si))$-module. Let
$\si_0\in End(\tilde V)$ be defined by $\si_0(v)=v,\,\si_0(1)=-1$. 
Then the decomposition $so(\tilde V)=so(V)\oplus V$ is precisely the 
eigenspace decomposition of $\si_0$. Since
$\si_0^2=I$, the pair $(so(\tilde V),so( V))$ is actually symmetric. 
As observed in Remark
\ref{primorem}, we have that $(G_{so(\tilde V),so( V)})_0$ acts 
trivially on  $L(\ov V,\ov\tau)$. Thus
$L(\ov V,\ov\tau)$ decomposes as a  $\widehat L(so(V),Ad(\si))$-module
as prescribed by Theorem \ref{multiplets}. Hence we need to find the 
set of minimal 
right coset representatives of the Weyl group of  $\widehat 
L(so(V),Ad(\si))$ in the Weyl 
group of  $\widehat L(so(\tilde V),Ad(\si))$-module. For this we need 
to distinguish two cases.
Suppose first that $\det\si=1$. It 
follows that $Ad(\si)$ is an automorphism of $so(V)$ of inner 
type. 
Choose a Cartan subalgebra $\h_{so(V)}$ of $so(V)$ which is fixed by 
$Ad(\si)$. Since $\det\s=1$, $\dim(V)$ is even (recall that we are assuming that the set of $\s$-fixed vectors in $V$ is even dimensional), hence  
$\h_{so(V)}$  is a Cartan subalgebra of  $so(\tilde V)$.
Choose $h\in \h_{so(V)}$ such that $Ad(\si)=e^{2\pi i\,ad(h)}$. Let 
$\{\be_1,\ldots,\be_n\}$ be a set of simple roots for $so(\tilde V)$. 
Let $\th$ be the corresponding highest root for 
$so(\tilde V)$. Then we can assume that  $\be_i(h)\geq 
0\  i=1,\ldots,n,\, \th(h)\leq 1$. It follows that
the map $n\d+\a\mapsto (n+\a(h))\d+\a$ is a bijection between the set 
of real roots of
$\widehat L(so(\tilde V),I)$ and the set of real roots of $\widehat 
L(so(\tilde V),Ad(\si))$ that maps
$\{\d-\th,\be_1,\ldots,\be_n\}$ to the set $\{\a_0,\dots,\a_n\}$ of 
simple roots for $\widehat L(so(\tilde V),Ad(\si))$
and the set $\{\d-\th,\be_1,\ldots,\be_{n-1},s_{\be_n}(\be_{n-1})\}$ 
of simple roots of $\widehat L(so( V),I)$ to the  set of simple 
roots of $\widehat L(so(V),Ad(\si))$. Then the Weyl groups of 
$\widehat L(so(\tilde V),Ad(\si))$ and $\widehat L(so(V),Ad(\si))$ 
are isomorphic to the Weyl groups of $\widehat L(so(\tilde V), I)$ 
and $\widehat L(so(V), I)$, respectively,
the index of the latter in the former is two, and the set of minimal 
length coset representatives is $\{1,s_{\a_n}\}$. Note finally that, 
in this case,  $\varphi_{so(V)}$ is the identity. Thus, according to 
Theorem~\ref{multiplets},
\begin{equation}\label{decfermionicinner}
F^{\bar\tau}(\ov V)=V(\rhat_{Ad(\s)}-\rhat_{so(V),Ad(\s)})\oplus 
V(s_{\a_n}(\rhat_{Ad(\s)}-\rhat_{so(V),Ad(\s)})).
\end{equation}
(see \eqref{rhos}, \eqref{rhoa} for notation). Here we used the fact that 
$s_{\a_n}(\rhat_{so(V),Ad(\s)})=\rhat_{so(V),Ad(\s)}$.

If instead $\det\s=-1$ then $\dim V$ is odd and $Ad(\s)$ is of inner 
type for $so(V)$ but not for $so(\tilde V)$.  Let, as above,  
$\h_{so(V)}$ be a Cartan subalgebra of $so(V)$ fixed pointwise by 
$\s$. Then there is an element $h$ of $\h_{so(V)}$ such that 
$Ad(\s)=Ad(\s_0)e^{2\pi i ad(h)}$. 
Arguing as in the $\det\s=1$ case, we find that the index of the Weyl 
group of $\widehat L(so(V),Ad(\s))$ in the Weyl group of $\widehat 
L(so(\tilde V),Ad(\s))$ equals the index of the Weyl group of 
$\widehat L(so(V),I)$ in the Weyl group of $\widehat L(so(\tilde 
V),Ad(\s_0))$ and that the set of minimal length representatives is 
$\{I,s_{\a_0}\}$.
Hence, in this case 
\begin{equation}\label{decfermionicouter}
F^{\bar\tau}(\ov V)=V(\rhat_{Ad(\s)}-\rhat_{so(V),Ad(\s)})\oplus 
V(s_{\a_0}(\rhat_{Ad(\s)}-\rhat_{so(V),Ad(\s)})).
\end{equation}
On the algebra $Cl(L(\ov V,\bar\tau))$ there is a unique  involutive 
automorphism such that  $x\mapsto-x$ for $x\in L(\ov V,\s)$.  Then, 
denoting by $Cl(L(\ov V,\bar\tau))^\pm$ the $\pm1$ eigenspace  for 
this
automorphism, we can write $$Cl(L(\ov V,\bar\tau))=Cl(L(\ov 
V,\bar\tau))^+\oplus Cl(L(\ov V,\bar\tau))^-.$$ 
Recall from \S~\ref{twisted fermionic} that
$$F^{\bar\tau}(\ov V)=Cl(L(\ov V,\bar\tau))/Cl(L(\ov 
V,\bar\tau))L^+(\ov V,\bar\tau).$$
It follows
that 
$$
F^{\bar\tau}(\ov V)=F^{\bar\tau}(\ov V)^+\oplus F^{\bar\tau}(\ov V)^-,
$$
where
$F^{\bar\tau}(\ov V)^\pm=Cl(L(\ov V,\bar\tau))^\pm/(Cl(L(\ov 
V,\bar\tau))L^+(\ov V,\bar\tau)\cap
Cl(L(\ov V,\bar\tau))^\pm$. Moreover $Cl(L(\ov V,\bar\tau))^+$ acts 
naturally on
$F^{\bar\tau}(\ov V)^\pm$.  

Set $\theta_{so(V)}(X)=\half\sum :\ov{X(b_i)}\bar b^i:$, where 
$\{b_i\}$, $\{b^i\}$ are  bases of $V$ dual to each other.
We note that the action of $X_r\in\widehat L(so(V),Ad(\s))$ is given by 
$(\theta_{so(V)}(X))^{F^{\bar\tau}(\ov V)}_{r}$. It follows that 
$F^{\bar\tau}(\ov V)^\pm$ are stable under the action of $\widehat 
L(so(V),Ad(\s))$. Thus, by the decompositions 
\eqref{decfermionicinner}, \eqref{decfermionicouter}, we obtain that 
$F^{\bar\tau}(\ov V)^\pm$ are both irreducible $\widehat 
L(so(V),Ad(\s))$-modules whose highest weights are switched by an 
involution $s$ of the Dynkin diagram
of $\widehat L(so(V),Ad(\s))$.

In particular, if $(\g,\aa)$ is any reductive pair and $\s$ is an automorphism such that $\rank\, \aa^{\bar0}=\rank\, \g^{\bar 0}$, we can apply the 
above discussion to $F^{\bar\tau}(\bar\p)$, turning it into a $\widehat 
L(so(\p),Ad(\s))$-module. Note that we can see $
L'(\aa,\s)$ as a subalgebra of $
L'(so(\p),Ad(\s))$ by embedding $\aa$ in $so(\p)$ via $ad_{\p}$ and 
that the action of $
L'(\aa,\s)$ on $F^{\bar\tau}(\bar\p)$ is just the restriction of the 
action of  $L'(so(\p),Ad(\s))$.
Since, by Wick's formula,
$$
[\theta_{so(\p)}(X)_\l L^{\ov\p} ]=-\l \theta_{so(V)}(X),
$$
letting $d_{so(V)}$ act as $(L^{\ov\p})^{\bar 
\tau}_{0}-(z(\g,\si)-z(\sa,\si)-\tfrac{1}{16}\dim\p)I$, we can 
extend this action to $\widehat L(so(\p),Ad(\s))$ in such a way that the 
action of $d_{so(\p)}$ equals the action of $d_\aa$. 
\vskip10pt
Now we observe that in our setting the so-called ``homogeneous 
Weyl-Kac" formula holds
(recall that $\aa^{\ov 0}$ is assumed to have the same rank of $\g$).
Indeed, in the (completed)
representation ring of
$\widehat L(\aa,\s)$, we have
\begin{equation}\label{hwk}
L(\L)\otimes F^{\bar \tau}(\bar \p)^+-L(\L)\otimes F^{\bar 
\tau}(\bar\p)^-=
\sum_{w\in \widehat 
W'}(-1)^{\ell(w)}V(\varphi^*_\sa(w(\L+\rhat_\si))-\rhat_{\sa\si}).
\end{equation}
This relation can be proved exactly as in the finite dimensional case 
(or affine $\sigma=I$ case, 
see \cite[Theorem 4]{land}), using Lemma~\ref{pesispin} to evaluate
$F^{\ov\tau}(\ov\p)^+- F^{\ov\tau}(\ov\p)^-$.\par

\begin{rem}\label{nominimal}
If $\mu\in\ha_\aa^*$ is dominant integral for $\Dap_\aa$ and $w\in\Wa_\aa$, we set $V(w(\mu+\rhat_{\aa,\s})-\rhat_{\aa,\s})=(-1)^{\ell(w)}V(\mu)$. Then, with this definition, we can rewrite \eqref{hwk} as
\begin{equation}\label{hwknominimal}
L(\L)\otimes F^{\bar \tau}(\bar \p)^+-L(\L)\otimes F^{\bar 
\tau}(\bar\p)^-=
\sum_{x\in \Wa_\aa\backslash\Wa}(-1)^{\ell(w_x)}V(\varphi^*_\sa(w_x(\L+\rhat_\si))-\rhat_{\sa\si}),
\end{equation}
where $\Wa_\aa\backslash\Wa$ is any set of right coset representatives and $w_x$ is any element from the coset  $x$.
\end{rem}
Following \cite[Ch. 13]{Kac} or  \cite[(2.2.1)]{KacW}, we 
recall the asymptotics of the character
of an integrable highest weight module $V$ over the affine algebra 
$\widehat L(\aa,\s)$, where $\aa$ is a
simple or abelian Lie algebra. Recall that  the series
$$ch_V(\tau,h)=tr_V e^{2\pi i (-\tau d_\aa+h)}$$
converges to an analytic function of the complex variable $\tau$, if 
$Im\,\tau>0$, for each $h\in\h_\aa$.
The asymptotics of this function, as $\tau\downarrow 0$ (i.e. $\tau=it,\,t\in \R^+,t\to 0$), is as 
follows:
\begin{equation}\label{cinquenove}ch_V(\tau,h)\approx 
a(\L)e^{\frac{\pi i c(k)}{12\tau}},\end{equation}
where $V=V(\L)$ is a highest weight module with highest weight $\L$,
  $c(k)$ is the conformal anomaly (= Sugawara central charge,
see \cite[(1.4.2)]{KacW}), which depends only on the level $k=\L(K)$ 
of $V$, and 
\begin{equation}\label{cinquedieci} a(\L)=b(k)\prod_{\a\in
R^+} \sin\pi\frac{(\bar\L+\rho_\aa,\a)}{k+g_\aa}.\end{equation}

Here $b(k)$ is a positive constant, depending only on $k$
(one can find in \cite{Kac} a simple formula for $b(k)$, which is unimportant
for the present paper) and  $R^+$ denotes the set of positive
roots (resp. coroots) of $\aa$ if $\widehat L(\aa,\s)$ is of type 
$X_n^{(1)}$ or $A_{2n}^{(2)}$ (resp. all
other types). Moreover,  $\rho_\aa=\half\sum\limits_{\a\in R^+}\a$ and
$g_\aa$ is  half of the value of the Casimir operator  on $\aa$.\par Now
let
$\aa$ be a reductive Lie algebra and let  $\aa=\bigoplus\limits_{j=0}^s
\aa_j$ be the decomposition of $\aa$  in the direct sum of an abelian Lie
algebra $\aa_0$ and simple components $\aa_j$, $j \geq 1$. Let $V$ be an
integrable $\widehat  L(\aa,\sigma)$-module form the category 
$\mathcal O$ of level $k=(k_0,\ldots,k_s)$ (i.e., $K_j\in\widehat 
L(\aa_j,\sigma)$ acts on $V$ via the scalar
$k_j$). Motivated by the above discussion, we define the {\it 
asymptotic dimension} of $V$ by
\begin{equation}\label{dimensioneasintotica}
\asdim(V)=\lim_{\tau\downarrow 0}e^{-\frac{\pi 
i}{12\tau}\sum\limits_{j=0}^s
c(k_j)}ch_V(\tau,h).\end{equation} If $V$ is irreducible, then it is 
an outer tensor product of irreducible $\widehat
L(\aa_j,\sigma)$-modules with highest weights $\L^j$ of level 
$k_j,\,j=0,\ldots,s,$ and it follows from \eqref{cinquenove} that 
$$\asdim(V)=\prod_{j=1}^s a(\L^j),$$
where $a(\L)$ is given by \eqref{cinquedieci}. We stipulate that $\asdim(V)=1$ if $s=0$, i.e. if $\aa$ is abelian.\par

\begin{rem}\label{extendasdim}
 Note that the asymptotic dimension is a positive
real number, which has all properties of the usual dimension. Namely, 
the asymptotic dimension
of the tensor product of modules equals the product of asymptotic 
dimensions of the factors, and the
asymptotic dimension of a finite direct sum of modules of the same 
level equals the sum of asymptotic
dimensions of summands. In particular we can extend the asymptotic dimension by linearity to the subring of the representation ring of $\widehat L(\aa,\s)$ generated by the integrable highest weight modules. 
\end{rem} 
\vskip5pt
Note that $\widehat W_\aa$ has finite index in $\Wa$ iff $\aa^{\ov 
0}$ is a semisimple equal rank Lie subalgebra
of $\g^{\ov 0}$.  To deal with the reductive case we need some preliminaries and more notation.
Let  $M\subset \h_0$
be  the lattice which indexes  the translations in the Weyl group of $\widehat L(\g,\s)$ (see
\cite[(6.5.8)]{Kac})  and set $M_0=\aa^{\ov 0}_0\cap M$ (recall that $\aa^{\ov 0}_0$ is the center of
$\aa^{\ov 0}$). Let $P_0$ be the lattice in $\aa^{\ov 0}_0$ dual to $M_0$ and let $L_0$ the lattice
corresponding to $M_0$ in $\h_0^*$ under the identification induced by the bilinear form. Let $T_{L_0}=\{t_\a\mid\a\in L_0\}$ be the set
of translations by elements of $L_0$ (see \cite[(6.5.2)]{Kac}).  Let $\Wa'_{fin}$ be a set of representatives for the right cosets of $T_{L_0}\times \Wa_\aa$ in $\Wa$. Note that in the semisimple case $\Wa'_{fin}$ is a set of 
right coset representatives for $\Wa_\aa$ in $\Wa$.

 Set 
$$\ha_c^*=(\sa_0^{\ov
0})^*\oplus\C\L^0_0,\quad\ha_{ss}^*=(\h_0\cap
\sum_{S>0}\aa_S)^*\oplus\sum_{S>0}\C\L_0^S,$$
so that any $\l\in\ha_\sa^*$ can be uniquely written as
$\l=\l_c+\l_{ss}+a\d_\aa,\,\l_c\in\ha_c^*,\l_{ss}\in\ha_{ss}^*,a\in\C$.
Set  $r=\dim\aa^{\ov 0}_0$.

\begin{prop}\label{asdim} If $\aa^{\bar0}$ is a reductive 
equal rank subalgebra of
$\g^{\bar0}$ such that $\aa^{\ov 0}_0=Span_{\C}M_0$, then $\Wa'_{fin}$ is finite and 
 $$
\sum_{w\in \widehat
W'_{fin}}(-1)^{\ell(w)}\asdim(V(\varphi^*_\sa(w(\L+\rhat_\si))-\rhat_{\sa\si}))=0.
$$
(here $\asdim(V(\varphi^*_\sa(w(\L+\rhat_\si))-\rhat_{\sa\si}))$ is defined as in Remark~\ref{extendasdim}).
\end{prop}
\begin{proof} Let $M_\aa$ 
be  the lattice which indexes  the translations in the Weyl group of $\widehat L(\aa,\s)$. Since $\h_0\subset\aa$,   $\rank( M_0\oplus M_\aa)=\rank\, \aa=\dim \h_0=\rank\, M$, hence $T_{L_0}\times\Wa_\sa$ has finite index in $\Wa$. This proves the first claim.

Let $\L^\pm$
be the highest weights of 
$F^{\bar\tau}(\bar\p)^\pm$ and let $s$ be the 
involution of the Dynkin diagram of $\widehat L(so(\p),Ad(\s))$ such 
that $s\L^+=\L^-$. 
Let  $h$ be an element in a Cartan subalgebra $\h_{so(\p)}$ of $so(\p)$  such 
that $d_{so(\p)}+h$ is the unique element
of $\C d+\h_{so(\p)}$ such that $\gamma(d_{so(\p)}+h)=1$ for any 
simple root $\gamma$ of $\widehat
L(so(\p),Ad(\s))$.  Let $ch^\pm$ be the character of  
$F^{\bar\tau}(\bar\p)^\pm$ as $\widehat
L(so(\p),Ad(\s))$-modules. Since $s(d_{so(\p)}+h)=d_{so(\p)}+h$ we 
see that 
$$
(ch^+-ch^-)(\tau,h)=0.
$$
 In particular we see that 
$$\lim_{\tau\downarrow0}(ch^+-ch^-)(\tau,h)=0.
$$
This limit is independent of $h$,  thus $(ch^+-ch^-)(\tau,0)\approx 0$ 
as $\tau\downarrow 0$.
Since $d_{so(\p)}$ acts as $d_\aa$ we get 
for the $\widehat L(\aa,\s)$-modules:
\begin{equation}\label{passaggio}
(ch_{L(\L)\otimes F^{\bar \tau}(\bar \p)^+}-ch_{L(\L)\otimes F^{\bar 
\tau}(\bar\p)^-})
(\tau,0)\approx0,
\end{equation}
as $\tau\downarrow0$. \par

Define $m_w=w(\L+\rhat_\si)(d)$ and $\L^w=w(\L+\rhat_\si)$. 
  Observe
that, for $\a\in L_0$ there exists  constants $c_1,c_2$ such that 
$t_\a(\l)=\l+c_1\l(K)\a-c_2((\l,K)+\half|\a|^2)\d$. Hence  we have
\begin{align*}\varphi^*_\sa(t_\a \L^w)&=
\varphi^*_\sa(\L^w)_c+c_1(k+g)\a-c_2
((\varphi^*_\sa(\L^w)_c,\a)+\half|\a|^2)\d_\aa
\\&+\varphi^*_\sa(\L^w)_{ss}+m_w\d_\aa.\end{align*}
Setting, for $\mu\in\ha^*_c$,  $\dot{t}_\a(\mu)=
\mu+c_1(k+g)\a-c_2((\mu,\a)+\half|\a|^2)\d_\aa$, 
we can write

$$ch(V(\varphi^*_\sa(t_\a \L^w)-\rhat_{\sa\si})=
\frac{e^{\dot{t}_\a(\varphi^*_\sa(\L^w)_c))}}{\eta(\tau)^r} ch
V(\varphi^*_\sa(\L^w)_{ss}+m_w\d_\aa-\rhat_{\aa\s}),
$$
$\eta$ being the Dedekind $\eta$-function. 

Since $T_{L_0}\Wa'_{fin}$ is a set of coset representatives for $\Wa_\aa$ in $\Wa$ and observing that multiplying any element $w\in\Wa$ by a translation does not change
the parity of $\ell(w)$, we can rewrite 
\eqref{hwknominimal} as
\begin{align*}
&ch(L(\L)\otimes F^{\bar \tau}(\bar \p)^+)-ch(L(\L)\otimes F^{\bar 
\tau}(\bar\p)^-)=\\
&\sum_{w\in \widehat 
W'_{fin}}\sum_{\a\in L_0} (-1)^{\ell(t_\a w)}ch(V(\varphi^*_\sa(t_\a \L^w)-\rhat_{\sa\si}))=\\
&\sum_{w\in \widehat 
W'_{fin}} (-1)^{\ell(w)}\frac{\Theta(\L^w_c)}{\eta(\tau)^r} ch(V(\varphi^*_\sa(\L^w)_{ss}-\rhat_{\sa\si}+m_w\d_\sa))
\end{align*}
where $\Theta(\mu)=\sum\limits_{\a\in L_0}e^{\dot{t}_\a(\mu)}$. Hence, by the asympotics 
given in \cite[(2.4.3)]{KW} and \eqref{passaggio}, we find that, taking the limit as $\tau\downarrow 0$, 
$$0=\sum_{w\in\Wa'_{fin}}(-1)^{\ell(w)}|P_0/L_0|^{-\half}(k+g)^{-\frac{r}{2}}
\asdim(V(\varphi^*_\sa(\L^w)_{ss}+m_w\d_\sa-\rhat_{\sa\si}).$$
Since, by definition, $\asdim(V(\varphi^*_\sa(\L^w)_{ss}+m_w\d_\sa-\rhat_{\sa\si}))=
\asdim(V(\varphi^*_\sa(\L^w)-\rhat_{\sa\si}))$, simplifying the constant  we are done.
\end{proof}

Turning to the finite dimensional case, if  $\aa$ is an equal rank 
reductive  subalgebra
 of a finite dimensional reductive Lie algebra $\g$, then  
\eqref{altdim} can be extended to $q$-dimensions. If $V$ is a finite
dimensional representation of $\aa$ and $ch(V)$ is its character, 
then the $q$-dimension of $V$ is 
\begin{equation}\label{qdimred}\dim_q
V=ch(V)(e^{r^\vee_\aa}),\end{equation}
where $r^\vee_\aa$ is an element of $\h$ such that $\a(r^\vee_\aa)=1$ for any simple 
root $\a$ of $\aa$. Note that $r^\vee_\aa$ coincides with
$\rho^\vee_\aa$ (the half  sum of the positive coroots of 
$\aa$) if $\aa$ is semisimple. In this case formula \eqref{qdimensione} gives 
an explicit expression for $\dim_q V$.
For a general reductive $\aa$ the element $r^\vee_\aa$ is not unique and the definition of $q$-dimension depends on this choice. We shall choose $r^\vee_\aa$ as in the following lemma.
\begin{lemma}\label{tecnico} Let $\g$ be a simple Lie algebra  and let
$\aa$ be a reductive equal rank subalgebra of $\g$. Let $\h$ be a common Cartan subalgebra and
let $\D$ (resp $\D_\aa$) be the set of roots of $\g$ (resp. $\aa$). 
Then there exists an element $r\in\h$, for which $\a(r)=1$ for each  simple root $\a$
of $\aa$ and $\beta(r)\in\ganz$ for some $\beta\in \D\setminus \D_\aa$.

\end{lemma}
\begin{proof} Assume $\aa$ is semisimple. Then we have to show that 
$\a(\rho^\vee_\aa)\in\ganz$ for some $\a$ in $\D\setminus \D_\aa$. 
Let $\a_1,\ldots,\a_n$ be a set of simple roots for 
$\D$ and let
$\theta=\sum_{i=1}^n a_i\a_i$ be the highest root. Fix an index 
$i_0,\,1\leq i_0\leq n$;
then the set $\{-\th,\a_i,\ldots,\a_{i_0-1},$ 
$\a_{i_0+1},\ldots,\a_n\}$ is a set of simple roots
of a semisimple subalgebra $\aa^{i_0}$ of $\g$. It is a Theorem  of 
Dynkin and Borel --  de Siebenthal that any semisimple $\aa$ can be 
obtained in this way by repeating the procedure several times, and 
the maximal (equal rank)  subalgebras are exactly
the $\aa^{i_0}$ 
with $a_{i_0}$  a prime number. We may assume that $\aa=\aa^{i_0}$ 
with $a_{i_0}$ prime. Then we have
$$\rho^\vee_\aa(\a_{i_0})=1-\frac{h}{a_{i_0}},$$
where  $h= 1+\sum_{i=1}^n a_i$ is the Coxeter number. Since  $a_{i_0}$ 
always divides $h$ if it is prime, we can take 
$\a=\a_{i_0}$.\par
If $\aa$ is reductive, start from $\rho_\aa^\vee+th$ with $t\in\C$ and $h$ in the center of $\aa$, and choose $t$ is such a way that  $\rho_\aa^\vee+th$ is integer valued on some root in 
$\D\setminus \D_\aa$.\end{proof}
Now we can prove the following
\begin{prop}\label{qdimens} Let $L$ be a finite dimensional irreducible
module with\break  highest weight $\lambda$ over a reductive Lie algebra $\g$.
Let $\aa$ be a reductive subalgebra\break  of $\g$, and $r=r_\aa^\vee$
be as in Lemma \ref{tecnico}. Let $V_w$ denote the irreducible $\aa$-module
with highest weight $ w(\lambda + \rho_\g)-\rho_\aa$, where $w \in W'$
and $W'$ is as in \eqref {in1}. Then 
\begin{equation}\label{altqdim}
\sum_{w\in  W'} (-1)^{\ell(w)}\dim_qV_w=0.
\end{equation}
\end{prop}

\begin{proof}
Let $F^+$ and $F^-$ be the even and odd components of the spin 
representation of $\p$. (Recall that, since $\aa$ is equal rank in
$\g$, $\p$ is even dimensional). Then, by \eqref{in1},
 $$
 ch(L\otimes F^+)-ch(L\otimes F^-)=\sum_{w\in  W'} 
(-1)^{\ell(w)}ch(V_w),
 $$
hence, to prove our claim, we need only to check that 
$$(ch(F^+)-ch(F^-))(e^{r^\vee_\aa})=0.$$ Let $\D$ be the set of 
roots of $\g$ and 
choose a positive set of roots $\Dp$ for $\g$. Then  $\D$ is the 
disjoint union of $\D_\aa$ and $\D_\p$, where $\D_\aa$ (resp. 
$\D_\p$) is the set of roots in $\D$ such that $\g_\a\subset\aa$ 
(resp. $\g_\a\subset\p$). Let $\Dp_\p=\Dp\cap\D_\p$ and set 
$\p^\pm=\sum_{\pm \a\in\Dp_\p}\g_\a$.  If $\a\in\D_\p$, choose 
$x_\a\in\g_\a$  in such a way that $(x_\a,x_\be)=\d_{\a,-\beta}$. 
Define  $E_{\a,\be}\in End(\p)$  by 
$E_{\a,\be}(x_\gamma)=\d_{\a,\gamma}x_\be$. A Cartan subalgebra of 
$so(\p)$ is $\h_{so(\p)}=\sum_{\a\in\Dp_\p}(E_{\a,\a}-E_{-\a,-\a})$. 
Define
$\epsilon_\a\in(\h_{so(\p)})^*$  by 
$\epsilon_\a(E_{\be,\be}-E_{-\be,-\be})=\d_{\a,\be}$. The $so(\p)$ 
character of $F^+-F^-$ is 
$$
e^{\half\sum_{\a\in\Dp_\p}\epsilon_\a}\prod_{\a\in\Dp_\p}(1-e^{-\epsilon_\a}).
$$
The Cartan subalgebra $\h$ of $\g$ is a Cartan subalgebra of $\aa$, 
hence it embeds in $so(\p)$ via $ad_{|\p}$. It follows that $h\in\h$ 
embeds as $\sum_{\a\in\Dp_\p}\a(h)(E_{\a,\a}-E_{-\a,-\a})$. In 
particular, to check that 
$ch(F^+-F^-)(e^{r^\vee_\aa})=0$, it is enough to find a root 
$\a\in\D_\p$ such that $\a(r^\vee_\aa)=0$. 
This root can be found as follows. By Lemma \ref{tecnico} we have that 
$\a(r^\vee_\aa)\in\ganz$ for some $\a\in\D_\p$. We may assume that 
$\a(r^\vee_\aa)>0$. Then there is a simple root $\be$ in $\Dp_\aa$ 
such that $(\a,\be)>0$, hence $\a-\be$ is a root.
Since the pair is reductive $[\aa,\p]\subset\p$, hence 
$\a-\be\in\D_\p$. Now $(\a-\be)(r^\vee_\aa)=\a(r^\vee_\aa)-1$.
If $(\a-\be)(r^\vee_\aa)=0$ we are done. Otherwise repeat the 
argument substituting $\a$ with
$\a-\be$.
\end{proof}

\section{Interlude: the very strange formula}
 We want to show that formula \eqref{masterrho} becomes the  ``very 
strange formula" (cf.
\cite[(13.15.4)]{Kac}) when $\si$ is an automorphism of order
$m$ (hence it affords a generalization of it, holding for infinite 
order automorphisms too). Let $\g$ be a
simple finite dimensional Lie algebra and $\h$ a Cartan subalgebra. 
Let $\Pi=\{\eta_1,\ldots,\eta_n\}$
be a set of simple roots for $\g$, $\Dp$ the associated subset of 
positive roots and $\rho$ the corresponding
half sum of positive roots. 
\begin{prop} Let $\kappa(\cdot,\cdot)$ denote the Killing form on $\g$
 and $\sigma$ be an automorphism of order $m$ of type 
$(s_0,s_1,\ldots,s_n;1)$ (\cite[Chapter 8]{Kac}). Let 
$\g=\oplus_{\bar j\in\ganz/m\ganz}\,\g^{\ov j}$ be the eigenspace 
decomposition with respect to $\s$. Define  $\l_s\in\h^*$
by $\kappa(\l_s,\eta_i)=\frac{s_i}{2m},\,1\leq i\leq n$. Then 
\begin{equation}\label{vsf} \kappa(\rho-\l_s,\rho-\l_s)=
\frac{\dim\g}{24}-\frac{1}{4m^2}\sum_{j=1}^{m-1}j(m-j)\dim \g^{\ov j}.
\end{equation}
\end{prop}
\begin{proof} Plug $\sa=0$ in formula \eqref{masterrho} and choose 
$\kappa(\cdot,\cdot)$ as 
invariant form on $\g$. Then we have $g=\half$ and the r.h.s. of 
\eqref{masterrho} becomes the r.h.s. of
\eqref{vsf}. Let  $\th$ denote the highest root of $\D$. We claim now 
that if we choose as simple roots in
$\g^{\ov 0}$ the set
$$
\Pi_0=\begin{cases}\{\eta_i\mid s_i=0\}\cup\{-\th\}\quad&\text{if 
$s_0=0$},\\
\{\eta_i\mid s_i=0\} \quad&\text{otherwise.}\end{cases}$$ then 
$\l_s=\rho-\rho_\si$.  To prove this we
use the following fact, which is not difficult to prove: if we put
$\beta_0=\frac{s_0}{m}\d-\theta,\,\beta_i=\frac{s_i}{m}\d+\eta_i$, 
then 
$\{\beta_0,\ldots,\beta_n\}$ is a set of simple roots for $\widehat 
L(\g,\sigma)$.  Note that
$\overline{\beta_i}=\eta_i$; using \eqref{rhoalfai} in our context  
we have
$$\kappa(\rho-\rho_\sigma,\eta_i)=\frac{\kappa(\eta_i,\eta_i)}{2}-\frac{\kappa(\eta_i,\eta_i)}{2}+\frac{s_i}{2m}$$
as desired.
\end{proof}
\section{Non-vanishing Dirac cohomology} 

If $M$ is a highest weight module of $\widehat L(\g,\s)$ and 
$N'=M\otimes F^{\bar\tau}(\bar\p)$, the Dirac cohomology of the 
affine Dirac operator is
\begin{equation}\label{Diraccohomology} H((G_{\g,\sa})_0^{N'})=Ker 
(G_{\g,\sa})_0^{N'}\slash (Im
(G_{\g,\sa})_0^{N'}\cap Ker (G_{\g,\sa})_0^{N'}).
\end{equation}
In this section we obtain a non-vanishing result for 
$H((G_{\g,\sa})_0^{N'})$ similar to Theorem
3.15 of \cite{K3}. Although the general strategy is parallel to 
Kostant's one, we obtain the crucial
result (Lemma \ref{fund})  using vertex algebra techniques. The 
setting is as in Section 4;  in particular,
when we consider the module $L(\L)\otimes F^{\ov\tau}(\ov\p)$,  the 
highest weight
$\L$ is  not assumed to be dominant integral. Note that if 
$\L$ is  dominant integral we have already determined the cohomology: 
in this case $N'=X=L(\L)\otimes F^{\bar\tau}(\bar\p)$ and 
$H((G_{\g,\sa})_0^X)=Ker (G_{\g,\sa})_0^X$. 
Indeed, by Proposition \ref{unitarity}, $Im
(G_{\g,\sa})_0^X\cap Ker (G_{\g,\sa})_0^X=0$. 
\vskip5pt 
Suppose that $M$ is a highest weight module for $\widehat 
L(\g,\si)$. Recall that
$\mathfrak n'_-=\mathfrak  n_-+\sum_{j<0}t^j\otimes\g^j\subset 
\widehat L(\g,\si)$. We set
$\overline{\mathfrak n'}_-=t^{-\half}\otimes\overline{\mathfrak 
n_-\cap\p}+\sum_{j<0}t^{j-\half}\otimes\overline{\g^j\cap\p}\subset 
L(\ov\p,\ov\tau)$. Set
$$ \mathcal M=\mathfrak n'_-M\otimes F^{\ov\tau}(\ov\p)+M\otimes 
\overline{\mathfrak
n'}_-F^{\ov\tau}(\ov\p).
$$
\begin{lemma}\label{fund}
$$ (G_{\g,\sa})^{N'}_0(\mathcal M)\subset \mathcal M.
$$
\end{lemma}
\begin{proof} We need the following formulae: if $x\in\g$ then, in 
$V^{k+g,1}(R^{super})$,
$$ [\tilde x_\l G_{\g,\sa}]=\sum_i:\widetilde{[x,b_i]}\ov 
b^i:+\l\,k\ov x_\p,
$$ while, if $x\in\p$,
$$ [\ov x_\l G_{\g,\sa}]=\tilde 
x+\half\sum_i:\overline{[x,b_i]}_\p\ov b^i:.
$$ These formulae are easily computed using Wick formula and the 
explicit form for $G_{\g,\sa}$ given
in Lemma~\ref{gsup}. Using \eqref{rep2}, we deduce that, on  
$M\otimes F^{\ov\tau}(\ov\p)$, we have, if $x\in\g$,
\begin{equation}\label{conxtilde} [\tilde x_{(n)}, 
(G_{\g,\sa})^{N'}_0]=\sum_i:\widetilde{[x,b_i]}
\ov b^i:^{\tau}_{(n+\half)}+n\,k\,(\ov x_\p)^{\tau}_{(n-\half)}
\end{equation} and, if $x\in\p$,
\begin{equation}\label{conxbar} [\ov x_{(n)}, 
(G_{\g,\sa})^{N'}_0]=\tilde
x^\tau_{(n+\half)}+\half\sum_i:\overline{[x,b_i]}_\p\ov 
b^i:^{\tau}_{(n+\half)}.
\end{equation} If we compute explicitly the normal order using 
\eqref{tensore}, then
$$ 
\sum_i:\widetilde{[x,b_i]}\ov 
b^i:^{\tau}_{(n+\half)}=\sum_{i,r}\widetilde{[x,b_i]}^\tau_{(r)} (\ov
b^i)^{\tau}_{(-r+n-\half)}
$$

Suppose now that $v=(t^n\otimes x)(w)\otimes u\in \mathcal M$ with 
$t^n\otimes x\in \mathfrak n'_-$. Then
either $n<0$ or $n=0$ and $x\in \mathfrak n_-$.

Since  $v=\tilde x^\tau_{(n)}(w\otimes u)$, we have that
$$ (G_{\g,\sa})^{N'}_0(v)=-[\tilde 
x^\tau_{(n)},(G_{\g,\sa})^{N'}_{0}](v)+\tilde
x_{(n)}(G_{\g,\sa})^{N'}_{0}(v),
$$
 so we need only to check that
$\sum_{i,r}\widetilde{[x,b_i]}^\tau_{(r)} (\ov 
b^i)^{\tau}_{(-r+n-\half)}(w\otimes u)\in \mathcal M$ and that $(\ov
x_\p)^{\tau}_{(n-\half)}(w\otimes u)\in \mathcal M$. This is obvious 
if $n<0$. For $n= 0$ then $x_\p\in \mathfrak
n_-\cap\p$ so $t^{-\half}\otimes\ov x_\p\in\overline{\mathfrak 
n'}_-$. Moreover $t^r\otimes[x,b_i]\in
\mathfrak n'_-$ if $r<0$ and $t^{-r-\half}\otimes\ov b^i\in 
\overline{\mathfrak n'}_-$ if $r>0$. It remains
only to check the term $r=0$. If $b_i\in \mathfrak n_-\cap\p\oplus 
\h_\p$ then $[x,b_i]\in \mathfrak n_-$,
while if $b_i\in \mathfrak n\cap \p$ then $ t^{-\half}\otimes\ov 
b^i\in\overline{\mathfrak n'}_-$. 

Suppose now $v=w\otimes \ov x_{(n)}u$ with either $n<-\half$ or 
$n=-\half$ and
$x\in\overline{\mathfrak n_-\cap\p}$. Arguing as above we need only 
to check that
$$ (\tilde x)^\tau_{(n+\half)}(w\otimes 
u)+\sum_i:\overline{[x,b_i]}_\p\ov b^i:^{\tau}_{(n+\half)}(w\otimes
u)\in \mathcal M,
$$ or, equivalently, that $\half\sum_i:\overline{[x,b_i]}_\p\ov 
b^i:^{\tau}_{(n+\half)}(u)\in
\overline{\mathfrak n'}_-F^{\ov\tau}(\p)$.

Write $u=(\Pi_{i=1}^mt^{r_i}\otimes\ov y_i)\cdot 1$ with either 
$r_i<-\half$ or $r_i=-\half$ and $y_i\in \mathfrak n_-\cap\p+\h_\p$.
. We introduce the following notation: 
if $x\in\p$, write
$$\gamma(x)=\half\sum_i:\overline{[x,b_i]}_\p\ov b^i:,$$ 
$\{b_i\},\{b^i\}$ being dual bases of $\p$. 
If
$y\in\p$, by Wick formula, 
$$[\ov y_\l\gamma(x)]=\overline{[y,x]}_\p,
$$
 thus, if $u'=(\Pi_{i=2}^mt^{r_i}\otimes\ov y_i)\cdot 1$, then
$\gamma(x)^\tau_{(n+\half)}(u)=(\ov
y)^{\tau}_{(r)}\gamma(x)^\tau_{(n+\half)}(u')+(\overline{[y,x]}_\p)^\tau_{(r+n+\half)}(u')$.  
Since $[(\ov
y)^\tau_{(r)},(\ov a)^\tau_{(s)}]=0$ if  $t^s\otimes\ov 
a\in\overline{\mathfrak n'}_-$, we have that $(\ov
y)^{\tau}_{(r)}\mathcal M\subset \mathcal M$. Hence we are left with 
checking that $\gamma(x)^\tau_{(n+\half)}(u')\in \mathcal
M$.  By an obvious induction on $m$, we reduce ourselves to check 
that 
$
\gamma(x)^\tau_{(n+\half)}(1)\in \mathcal M$.
Computing explicitly the normal order and using the fact that, if 
$\{b_i\}$ is a basis of $\g^j\cap \p$ then
$\sum_i([x,b_i],b^i)=0$ if $x\in \mathfrak n_-\cap\p$, we find that
\begin{align*}
        \sum_{i}:(\overline{[x,b_i]}_\p)\ov
b^i:^{\tau}_{(n+\half)}&=\sum_{i,r<n}(\overline{[x,b_i]}_\p)^\tau_{(r)}(\ov
b^i)^{\tau}_{(-r+n-\half)}\\&-\sum_{i,r\ge n}(\ov 
b^i)^{\tau}_{(-r+n-\half)}(\overline{[x,b_i]}_\p)^\tau_{(r)}.
\end{align*} Thus $\sum_{i}:(\overline{[x,b_i]}_\p)\ov 
b^i:^{\tau}_{(n+\half)}(1)=-\sum_{i,r\ge n}(\ov
b^i)^{\tau}_{(-r+n-\half)}(\overline{[x,b_i]}_\p)^\tau_{(r)}(1)$. The 
terms that are not obviously in $\mathcal M$ are
those with $n=r$ and $b^i\in\mathfrak n\cap\p+\h_\p$. But in this 
case $[x,b_i]_\p\in\mathfrak
n_-\cap\p$ and we are done.
\end{proof}

\begin{cor}\label{hw}Fix $\L\in\ha^*_0$ and let $M$ be a highest weight 
module for $\widehat L(\g,\si)$
of highest weight $\L$. Set $N'=L(\L)\otimes F^{\bar\tau}(\bar\p)$. 
If $(\L+\rho_\si)_{|\h_\p}=0$ then 
$H((G_{\g,\sa})_0^{N'})\ne 0$.
\end{cor}
\begin{proof} From the explicit description of the basis of 
$F^{\ov\tau}(\ov\p)$ given in
Lemma~\ref{pesispin}, we see that
$F^{\ov\tau}(\ov\p)=\mathcal M\oplus (v_\L\otimes 
Cl(t^{-\half}\otimes \ov\h_\p)\cdot 1)$. To conclude the proof
we need only to show that $(G_{\g,\sa})^{N'}_0(v_\L\otimes 
Cl(t^{-\half}\otimes \ov\h_\p)\cdot 1)=0$.
For this, it suffices to show that
$$ (G_{\g,\sa})^{N'}_0(\ov h_{i_1})^{\ov\tau}_{(-\half)}\cdots (\ov
h_{i_r})^{\ov\tau}_{(-\half)}(v_\L\otimes 1)=0
$$ with $h_{i_j}$ chosen as in \eqref{vectorbasis}. We prove this by 
induction on $r$. If $r=0$, then, by
Proposition~\ref{lpirhozero} and our assumption that 
$(\L+\rho_\si)_{|\h_\p}=0$, we have $(G_{\g,\sa})^{N'}_0(v_\L\otimes 
1)=0$. If $r>0$ then, by the induction hypothesis and \eqref{conxbar},
\begin{align*}
        (G_{\g,\sa})^{N'}_0(\ov h_{i_1})^{\ov\tau}_{(-\half)}&\cdots (\ov
h_{i_r})^{\ov\tau}_{(-\half)}(v_\L\otimes 1)=\\&=((\tilde
h_{i_1})^{\ov\tau}_{(0)}+\gamma(h_{i_1})^\tau_{(0)})(\ov 
h_{i_2})^{\ov\tau}_{(-\half)}\cdots (\ov
h_{i_r})^{\ov\tau}_{(-\half)}(v_\L\otimes 1).
\end{align*} Since $[\gamma(h_{i_1})_\l \ov h]=\ov{[h_{i_1},h]}=0$ 
for any $h\in\h_\p$, we can rewrite the
last formula as
\begin{align*}
        (G_{\g,\sa})^{N'}_0(\ov h_{i_1})^{\ov\tau}_{(-\half)}\cdots (\ov
h_{i_r})&^{\ov\tau}_{(-\half)}(v_\L\otimes 1)=\\&=\L(h_{i_1})(\ov 
h_{i_2})^{\ov\tau}_{(-\half)}\cdots
(\ov h_{i_r})^{\ov\tau}_{(-\half)}(v_\L\otimes 1)\\&+(\ov 
h_{i_2})^{\ov\tau}_{(-\half)}\cdots (\ov
h_{i_r})^{\ov\tau}_{(-\half)}(v_\L\otimes 
\gamma(h_{i_1})^\tau_{(0)}(1)).
\end{align*} But  $\gamma(h)^\tau_{(0)}(1)=\rho_\si(h)\cdot1$ (see 
formula \eqref{prima} further on)
whence 
\begin{align*}
(G_{\g,\sa})^{N'}_0(\ov h_{i_1})^{\ov\tau}_{(-\half)}&\cdots (\ov
h_{i_r})^{\ov\tau}_{(-\half)}(v_\L\otimes 
1)=\\&=(\L+\rho_\s)(h_{i_1})\left((\ov 
h_{i_2})^{\ov\tau}_{(-\half)}\cdots
(\ov h_{i_r})^{\ov\tau}_{(-\half)}(v_\L\otimes 1)\right)=0
\end{align*} 
by our assumption that $(\L+\rho_\s)_{|\h_\p}=0$.
\end{proof}
\section{An analogue of a conjecture of Vogan in affine setting}
\subsection{The ``Vogan conjecture" and its generalization}
Recall from \eqref{vp}   the function
$\varphi_\aa:\h_0\oplus\sum_S\C K_S\oplus\C d_\aa\to\ha_0$ and that, 
extending functionals by zero
on $\h_\p$, we can view $\ha_\aa^*$ as a subspace of 
$(\h_0\oplus\sum_S\C K_S\oplus\C d_\aa)^*$. 
Recall that   the complexified Tits cone of an affine algebra 
$\widehat L(\g,\si)$ is the set
\begin{equation}\label{tits}C_\g=\{\l\in\ha^*_0\mid 
Re\,\l(K)>0\}.\end{equation}
If
$f$ is a function on $C_\g$ we denote by $f_{|\ha_\aa^*}$  the 
function on
$\varphi_\aa^*(C_\g)\cap\ha_\aa^*$ defined by 
$f_{|\ha_\aa^*}(\l)=(f\circ(\varphi_\aa^*)^{-1})(\l)$.
Suppose that $M$ is a highest weight module  for $\widehat L(\g,\si)$ 
with highest weight $\L\in C_\g$.
We already observed in
\S~\ref{diracoperators} that, given $\nu\in (\ha_\sa)^*$ such that
$(M\otimes F^{\ov\tau}(\ov\p))_\nu\ne0$, then $\nu+\rhat_{\sa \si}\in 
\varphi^*_\sa(C_\g)$ so
$f_{|\ha_\aa^*}(\nu+\rhat_{\sa \si})$ makes sense. \par
\vskip5pt
The following result is our affine analog of Vogan's conjecture. It 
appears, in a slightly different formulation,  as Theorem 
\ref{primovogan} in the Introduction. In order
to use Kac's results from 
\cite{Kacpnas} on the
holomorphic center of a suitable completion of $U(\widehat 
L(\g,\sigma))$, we need a technical hypothesis  on the
pair $(\g,\aa)$ (which is satisfied in  important cases, e.g. when
$\aa^{\ov 0}$ is an equal rank  subalgebra of $\g^{\ov 0}$ or the set 
of fixed
points of a diagonalizable automorphism of $\g$). 
\begin{theorem}\label{Vogan1} Assume that the centralizer $C(\h_\aa)$ 
of $\h_\aa$ in $\g^{\bar0}$ equals $\h_0$. Fix
$\L\in
\widehat
\h^*_0$ such that
$\L+\rhat_\si\in C_\g$ and let
$M$ be a highest weight module for $\widehat L(\g,\si)$ with highest 
weight $\L$. Let $f$ be a
holomorphic
$\Wa$-invariant function on $C_\g$. Suppose that a highest weight 
$\widehat L(\sa,\si)$-module $Y$ of
highest weight
$\mu$
 occurs in the Dirac cohomology $H((G_{\g,\sa})_0^{N'})$ of
$N'=M\otimes  F^{\ov\tau}(\ov\p)$. Then 
\begin{equation}\label{fine}f_{|{\ha}_\sa^*}(\mu+\rhat_{\aa\si})= 
f(\L+\widehat\rho_{\si}).\end{equation}
\end{theorem}
\begin{rem}  Assume that $\aa^{\bar 0}$ is an equal rank subalgebra 
of $\g^{\bar 0}$ and that $\L$ is dominant integral.
Then \eqref{fine} is a fairly trivial consequence of Theorem 
\ref{multiplets}. Indeed, in the integral dominant
case, we have
$H((G_{\g,\sa})_0^X)=Ker((G_{\g,\sa})_0^X)$, and formula 
\eqref{formulona} tells us
  that the only weights appearing in the decomposition  are  
precisely those in the orbit of $\L+\rhat_\sigma$.
Since
$f$ is
$\Wa$-invariant, the claim follows. 
\par
\vskip10pt
The following  sections are devoted to the proof of 
Theorem~\ref{Vogan1}. 
\end{rem}
\vskip10pt
\subsection{The basic setup} \label{basicsetup}
Consider the  algebra $\mathcal 
W=U(L'(\g,\si))\otimes
Cl(L(\ov\g,\ov\tau))$ and its quotient algebra $\mathcal W^k=\mathcal 
W/\mathcal W(K-k)$. We view $L(\g,\s)\oplus L(\bar\g,\bar\tau)$ as a 
subspace of $L'(\g,\si))\oplus L(\ov\g,\ov\tau)$ and identify it with 
its image in $\mathcal W^k$.

Assume that the centralizer of $\h_\aa$ in $\g^{\bar 0}$ is the Cartan subalgebra $\h_0$. Therefore we can fix  $f_\aa\in\h_\aa$ such that the centralizer 
of 
$f_\aa$ in $\g^{\bar0}$ is  $\h_0$. We can
choose $\Dp_0$ so that  $\a(f_\aa)>0$ if $\a\in\Dp_0$.  Choose 
$c\in\R$ such that
$\a_i(cd+f_\aa)>0$ for all $i$. 
 Set $\hat
f_\aa=cd+f_\aa$. With this choice, we have that if $\a\in\Dap$, then 
$\a(\hat f_\aa)> 0$.\par
We define a $\R$-grading $\deg$
on
$L'(\g,\si)$ and on
$L(\bar
\g,\bar\tau)$ by setting,  for $x\in\g^{\bar r}_{\a}$,
\begin{equation}\label{gradoprincipale}\deg(t^r\otimes 
x)=\deg(t^{r-\half}\otimes
x)=(r\d+\a)(\hat f_\aa),\quad \deg(K)=0\end{equation} and write
$L'(\g,\s)=\oplus_jL'(\g,\s)_j$, 
$L(\bar\g,\bar\tau)=\oplus_jL(\bar\g,\bar\tau)_j$ for the
corresponding decomposition into homogeneous components.

We
can extend $\deg$ to $\mathcal W$ thus defining a $\R$-grading  
$\mathcal W=\oplus_j\mathcal W_j$.
Since $K-k$ is homogeneous, $\deg$ induces a $\R$-grading on 
$\mathcal W^k$ that we 
again denote
by $\deg$.

Observe that $\g^{\bar r}=\oplus_{\mu\in\h_\aa}\g^{\mu,\bar r}$ 
where  $\g^{\mu,\bar r}$ is the $\h_\aa$-weight space of weight 
$\mu$. 
Note that $\g^{\mu,\bar r}$ is homogeneous with respect to $\deg$ and 
$\g^{\mu,\bar r}=\aa^{\mu,\bar r}\oplus\p^{\mu,\bar r}$,
where $\aa^{\mu,\bar r}=\g^{\mu,\bar r}\cap\aa$ and $\p^{\mu,\bar 
r}=\g^{\mu,\bar r}\cap\p$. It follows that  $L(\aa,\s)$,
$L(\p,\s)$, $L(\bar\aa,\bar\tau)$, and $L(\bar\p,\bar\tau)$ all 
inherit a grading $\deg$ from the grading $\deg$
 on $L(\g,\s)$,
$L(\bar\g,\bar\tau)$ and
\begin{equation}\label{ass}
L(\g,\s)_j=L(\aa,\s)_j\oplus L(\p,\s)_j\quad L(\bar 
\g,\bar\tau)_j=L(\bar\aa,\bar\tau)_j\oplus L(\bar\p,\bar\tau)_j.
\end{equation}
 Recall the triangular decompositions
$L'(\g,\si)=\n'_-\oplus\h'_0\oplus\n'$ and $L(\ov 
\g,\ov\tau)=\ov\n'_-\oplus\ov\h_0\oplus\ov \n'$. Set
$$\mathcal W^+=U(\n')\otimes Cl(\ov \n')$$
and observe that $\mathcal W^+=\bigoplus\limits_{n\in\R^+}\mathcal 
W^+_n$.\par
Let $\mathcal F$ be the algebra of holomorphic functions on 
$(\h_0^*\oplus\C\d)\times(\h_\aa^*\oplus\C\d_\aa)$. 
 Observe that one can define a product on $\mathcal W\otimes \mathcal 
F$ by setting, for $x\in\g^{\ov
r}_\a$, $y\in\p^{\mu,\ov r}$, $z\in\aa^{\mu,\ov r}$,  and 
$f,g\in\mathcal F$,
\begin{equation}\label{prodotto}
f (t^r\otimes x)= (t^r\otimes x)f_{\a+r\d, 
\a_{|\h_\aa}+r\d_\aa},\quad f(t^{r-\half}\otimes\ov
y)=(t^{r-\half}\otimes\ov y)f_{0,\mu+r\d_\aa}
\end{equation}
and 
\begin{equation}\label{prodottoa}
[(t^{r-\half}\otimes\ov z),f]=[f,g]=[f,K]=0,
\end{equation}
 where $f_{\a,\be}(\l,\mu)=f(\l+\a,\mu+\be)$. We will refer to 
$f_{\a,\be}$ as a {\it translate} of $f$.
Since $(1\otimes f)\mathcal W(K-k)\subset\mathcal W(K-k)\otimes 
\mathcal F$ the product on
$\mathcal W\otimes\mathcal F$ factors to define a product on 
$\mathcal W^k\otimes \mathcal F$.

Set $j_0=0$ and define  recursively, for $N\in\nat$, 
\begin{equation}\label{defjp}j_N=\min(j>j_{N-1}\mid \mathcal 
W^+_j\ne\{0\}).\end{equation} To see that  $j_N$ is well
defined we argue as follows:
 if $x\in\mathcal W^+_j$ then $x$ is a sum of products $x^1\cdots 
x^m$ with $x^i= t^{r_i}\otimes y^i$ or 
$x^i=t^{r_i-\frac{1}{2}}\otimes\ov y^i$ with
$y^i\in
\g^{\bar r_i}_{\gamma_i}$. Then $j=\l(\hat f_\aa)$ where $\l=\sum 
(r_i\d+\gamma_i)$. Since $\l=\sum n_i\a_i$ with $n_i\in\nat$
and since $\a_i(\hat f_\aa)\ge0$, we see that, fixing $R>0$, then the 
set $\{j<R\mid \mathcal W^+_j\ne\{0\}\}$ is finite. This in
turn implies that $\{j>j_{N-1}\mid \mathcal W^+_j\ne\{0\}\}$ has a 
minimum.

If $N\in\nat$ set 
\begin{align*} F^N=\sum_{n\in \R,\,n\ge j_N}\mathcal W^+_n.
\end{align*} 
The set $\{(\mathcal W^k\otimes\mathcal F)(F^N\otimes\mathcal 
F)\}_{N\in\mathbb N}$ is, by PBW
theorem, a fundamental system of neighborhoods of $0$ in $\mathcal 
W^k\otimes\mathcal F$. Let
$(\mathcal W^k)_{\mathcal F}^{com}$ be the corresponding completion. 
Set $$\ov{\mathcal
W^k_{\mathcal F}}=\oplus_j\ov{\mathcal W^k_j\otimes \mathcal F}.$$ 

The product we just defined on $\mathcal W^k\otimes \mathcal F$ can 
be extended to a product on
$\ov{\mathcal W^k_{\mathcal F}}$.  Indeed, suppose that $u\in 
\ov{\mathcal W^k_i\otimes \mathcal F}$
and $w\in\ov{\mathcal W^k_j\otimes \mathcal F}$. 
 Writing explicitly $u,w$ as a sequence $u^n\otimes f_n$ and 
$w^n\otimes g_n$, we now check that 
$(u^n\otimes f_n)(w^n\otimes g_n)$ converges (i.e. is a Cauchy 
sequence). We need to check that for each $N$ there is $p$ such
that, if $n,m>p$ then $(u^n\otimes f_n)(w^n\otimes g_n)-(u^m\otimes 
f_m)(w^m\otimes g_m)\in
(\mathcal W^k\otimes\mathcal F)(F^N\otimes\mathcal F)$. Observe that 
\begin{align*}
        (u^n\otimes f_n)(w^n\otimes g_n)-(u^m\otimes f_m)&(w^m\otimes 
g_m)=(u^n\otimes f_n-u^m\otimes
f_m)(w^n\otimes g_n)\\&+(u^m\otimes f_m)(w^n\otimes g_n-w^m\otimes 
g_m).
\end{align*} Since $w$ is a Cauchy sequence, there is $N_1$ such 
that, for $n,m>N_1$, $(w^n\otimes
g_n-w^m\otimes g_m)\in(\mathcal W^k\otimes\mathcal 
F)(F^N\otimes\mathcal F)$.  Since $u$ is a
Cauchy sequence, there is $N_2$ such that, for $n,m>N_2$, 
$(u^n\otimes f_n-u^m\otimes
f_m)\in(\mathcal W^k\otimes\mathcal F)(F^{M}\otimes\mathcal F)$ with 
$j_M>j_N-j$. By \eqref{ass}  we have that 
\begin{equation}\label{degfezero} (1\otimes\mathcal F)(\mathcal 
W^k_j)=\mathcal W^k_j\otimes\mathcal F,
\end{equation} hence $(u^n\otimes f_n-u^m\otimes f_m)(w^m\otimes 
g_m)\in(\mathcal
W^k\otimes\mathcal F)(F^{M}\mathcal W^k_j\otimes\mathcal F)$. By PBW 
theorem, $F^{M}\mathcal
W^k_j\subset\sum_{\substack{r\le 0\\r+s\ge j_N}}\mathcal 
W^k_r\mathcal W^+_s\subset \mathcal W^kF^N$.
Thus $(u^n\otimes f_n-u^m\otimes f_m)(w^m\otimes g_m)\in(\mathcal 
W^k\otimes\mathcal
F)(F^{N}\otimes\mathcal F)$, for $n,m>N_2$. Setting 
$p\ge\max(N_1,N_2)$ we are done.

 Note that \eqref{degfezero} implies that defining $\deg(\mathcal 
F)=0$ we can extend $\deg$ to a grading on
$\ov{\mathcal W^k_{\mathcal F}}$, by declaring $\deg(u)=j$ if $u\in 
\ov{\mathcal W^k_j\otimes \mathcal F}$.
\vskip 5pt

Assume that $k+g>0$ and choose $\L\in\ha_0^*$ a weight of level $k$.  
If $M$ is a highest weight module for $\widehat
L(\g,\si)$ of highest weight $\L$, then $N=M\otimes F^{\ov \tau}(\ov 
\g)$ is a representation of 
$\mathcal W^k$. Since, for any $v\in N$, there is $p>0$ such that 
$F^M\cdot v=0$  for $M\ge p$, we can extend the action of $\mathcal W^k$ to $\ov{\mathcal 
W^k\otimes
1}\subset\ov{\mathcal W^k_{\mathcal F}}$.
\begin{lemma}\label{faithful}Let $u\in \ov{\mathcal W^k\otimes 1}$ be 
such that, for any highest weight
module $M$ of level $k$ and any $v\in N=M\otimes F^{\ov\tau}(\ov 
\g)$,  relation $u\cdot v=0$ holds. Then $u=0$.
\end{lemma}
The proof is postponed to Section~\ref{lemma85}.
 \vskip10pt

If $\ov r\in \R/\ganz$ and $r\in\ov r$, let ${}^{\ov 
r}V^{k+g,1}(R^{super})$ be the $e^{2\pi i
r}$-eigenspace of $\tau$ in $V^{k+g,1}(R^{super})$.  Set $H=\frac{1}{k+g}\tilde 
L^\g_{(1)}-L^{\ov \g}_{(1)}$.  Since $\tau$  fixes
$\frac{1}{k+g}\tilde L^\g -L^{\ov \g}$, the 
$\R/\ganz$-grading induced on  $V^{k+g,1}(R^{super})$ by $\tau$ is 
compatible with $H$, i.\ e.{} we can
write  $$V^{k+g}(R^{super})=\bigoplus_{\substack{\ov r\in\R/\ganz\\ 
\D\in\R}}{}^{\ov r}V^{k+g,1}(R^{super})[\D],$$ where ${}^{\ov 
r}V^{k+g}(R^{super})[\D]$ are the joint
eigenspaces for $\tau$ and
$H$. If $N$ is a $\tau$-twisted representation of 
$V^{k+g,1}(R^{super})$ and $a\in {}^{\ov
r}V^{k+g}(R^{super})[\D]$ then, for
$r\in\ov r-\D$, we set $a_r^N=a^N_{(r+\D-1)}$. If $a\in  {}^{\ov 
r}V^{k+g}(R^{super})[\D]$ we set $\ov
r_a=\ov r$ and
$\D_a=\D$. $\D_a$ is called the conformal weight of $a$.

 Set  $\si_{f_\aa}=e^{ad(f_\aa)}$. 
 As in \S~\ref{repsup} we can extend   $\si_{f_\aa}$ to  an 
automorphism $\si_{f_\aa}$ of
$V^{k+g,1}(R^{super})$. If $v\in V^{k+g,1}(R^{super})$ is an 
eigenvector for  $ \si_{f_\aa}$ with
eigenvalue $s$ we set $\deg(v)=\log(s)$.

Since $f_\aa\in\h_0$, $\si_{f_\aa}$ commutes with $\tau$  and clearly 
fixes
$\frac{1}{k+g}\tilde L^\g -L^{\ov \g}$, thus   ${}^{\ov 
r}V^{k+g}(R^{super})[\D]$ is spanned by elements that
 are homogeneous with respect to $\deg$.

\begin{lemma}\label{fieldsinW} For any $a\in {}^{\ov 
r}V^{k+g}(R^{super})[\D]$ and $r\in\ov r-\D$,
there is a unique element $a_r\in\ov{\mathcal W^k\otimes 1}$ such 
that, for any highest weight module
$M$ of level $k$ and any $v\in N=M\otimes F^{\ov\tau}(\ov \g)$, 
$$a_r\cdot v=a^N_r\cdot v.$$

Moreover 
\begin{equation}\label{degrepres}
\deg(a_r)=cr+\deg(a).
\end{equation}
\end{lemma}

Uniqueness follows at once from Lemma~\ref{faithful}; the existence 
part of the proof as well as the proof of \eqref{degrepres} is found 
in 
Section~\ref{lemma86}. 
\vskip10pt
Recall that, for  $x\in\aa$, $\theta(x)=(\tilde x)_\aa-\tilde x$. 
\begin{lemma}\label{gradtoCdiag}If $h\in\h_\aa$,
we have
$\theta(h)_0-(\rho_\s-\rho_{\aa\s})(h)\in\ov{\mathcal 
W^kF^1\otimes1}$.
Moreover, if  
$D=L^{\ov\p}_0-(z(\g,\si)-z(\sa,\si)-\tfrac{1}{16}\dim\p)I$ (cf.
\eqref{azionednormalizzata}) then $D\in\ov{\mathcal W^kF^1\otimes1}$.
\end{lemma}
\begin{proof}
Arguing as in the proof of Lemma~\ref{lemmamu}, if $\{x_i\}$ is a 
basis of $\n$ and $\{x^i\}$ is the basis of $\n_-$ dual to $\{x_i\}$, 
then
$$\theta(h)_0= 
\sum_{i} \ov{[h,x^i]}_0(\ov x_i)_0+ (\rho_\s-\rho_{\aa\s})(h)+u$$ 
with $u\in\ov{\mathcal W^kF^1\otimes1}$.
We can choose $x_i,x^i$ as root vectors for $\g^{\bar0}$. Recall  
that we have chosen $\deg$ in such a way that  $\deg(x_i)\ge0$ and 
$\deg(x_i)=0$ if and only if $x_i\in\n\cap C(\h_\aa)$. Hence
$$
\theta(h)_0= 
\sum_{i:x_i\in C(\h_\aa)} \ov{[h,x^i]}_0(\ov x_i)_0+ 
(\rho_\s-\rho_{\aa\s})(h)+u'
$$ 
with $u'\in\ov{\mathcal W^kF^1\otimes1}$.
If $\a$ is a root that occurs in $C(\h_\aa)$, then $\a_{|\h_\aa}=0$. 
Thus 
$$\theta(h)_0= 
 (\rho_\s-\rho_{\aa\s})(h)+u'$$
  as wished.
  
In order to show that $D\in\ov{\mathcal W^kF^1\otimes1}$, we first 
show that there is a constant $d_0\in\C$ such that $D=d_0+u$ with 
$u\in\ov{\mathcal W^kF^1\otimes1}$. Recall  that, up to a constant, 
$D=L^{\bar \p}_0$. Writing explicitly $L^{\bar\p}_0=-\half\sum 
:T(\bar{b_i})\bar{b^i}:_0$, we see that 
$$
D=\half\(\sum_i(\sum_{r<s_i}(r+\half)(\bar b_i)_r\bar 
b^i_{-r}-\sum_{r\ge s_i}(r+\half)\bar b^i_{-r}(\bar b_i)_r)\)+const.
$$
Choosing $s_i\in[0,1)$ we see that 
$$
D=-\frac{1}{4}\(\sum_i(\sum_{i: s_i=0}\bar b^i_{0}(\bar 
b_i)_0)\)+u+const
$$
with $u\in\ov{\mathcal W^kF^1\otimes1}$. We can choose $\{b_i\mid 
s_i=0\}$ to be an orthonormal basis of $\p\cap\g^{\bar 0}$ so, since 
$((\bar b_i)_0)^2=\half$,
 $D=u+d_0$
with $d_0\in\C$ and $u\in\ov{\mathcal W^kF^1\otimes1}$. To conclude 
the proof it is enough to recall that the normalization in 
\eqref{azionednormalizzata} was chosen precisely to obtain that 
$D^{\bar \tau}\cdot1=0$, hence $d_0=0$.
\end{proof}

 We can define a linear
map from $\h_0\times\h_\aa\times\C(d-d_\aa)$ to $\ov{\mathcal 
W^k_{\mathcal F}}$ by mapping 
 $h\in\h_0$ to $\tilde h_0$,  $h\in\h_\sa$ to $(\tilde h)_\sa{}_0$,
and $d_\aa-d$ to $D= 
L^{\ov\p}_0-(z(\g,\si)-z(\sa,\si)-\tfrac{1}{16}\dim\p)I$ (cf.
\eqref{azionednormalizzata}).

Remark that formula \eqref{rep1} says that, if $a,b\in 
V^{k+g,1}(R^{super})$, 
\begin{equation}\label{brakett} [a^N_s,b^N_r]=\sum_{j\ge 
0}\binom{s+\D_{a}-1}{j}(a_{(j)}b)^N_{s+r}.
\end{equation}
Hence, by 
Lemma~\ref{faithful},
 \begin{equation}\label{brachkett} [a_s,b_r]=\sum_{j\ge 
0}\binom{s+\D_{a}-1}{j}(a_{(j)}b)_{s+r}.
\end{equation}
In view of Lemma~\ref{dzero} and \eqref{brachkett}, $\tilde h_0, 
(\tilde h)_{\aa 0}$ and $D$ commute with each other, hence we can
extend the map defined above to an algebra map  
$L:S(\h_0\times\h_\aa\times\C(d-d_\aa))\to \ov{\mathcal W^k_{\mathcal 
F}}$.

 Clearly we can look upon $S(\h_0\times\h_\aa\times\C(d-d_\aa))$ as a 
subset of $\mathcal F$ by
setting $h(\l,\mu)=\l(h)$ for $h\in\h_0$, $h(\l,\mu)=\mu(h)$ for 
$h\in\h_\aa$ and
$(d-d_\aa)(\l,\mu)=\l(d)-\mu(d_\aa)$. This embedding induces an 
algebra map
$R:S(\h_0\times\h_\aa\times\C(d-d_\aa))\to \ov{\mathcal W^k_{\mathcal 
F}}$. 
 
Let $\mathcal I$ be the (two-sided) ideal in $\ov{\mathcal 
W^k_{\mathcal F}}$ generated by $\{L(f)-R(f)\mid f\in
S(\h_0\times\h_\aa\times\C(d-d_\aa))\}$.  We need to describe the 
ideal $\mathcal I$ more carefully.
Note that the product in $\mathcal W^k_{\mathcal F}$ has been devised 
in such a way that, if $x\in\g^{\ov r}_\a$, then 
\begin{equation}\label{tildeandideal}
\tilde x_r(L(f)-R(f))=(L(f')-R(f'))\tilde x_r
\end{equation}
 and 
 \begin{equation}\label{barandideal}
 \ov x_r(L(f)-R(f))=(L(f'')-R(f''))(\ov x_\p)_r+(L(f)-R(f))(\ov 
x_\aa)_r
 \end{equation}
 with $f'$ and $f''$ suitable translates of $f$. Moreover, if $u\in 
U(\h_0)\otimes Cl(\bar\h_0)\otimes\mathcal F$, 
 \begin{equation}\label{cartanandideal}
 [u,L(f)-R(f)]=0.
 \end{equation}
 
 Observe that \eqref{tildeandideal} and \eqref{barandideal} imply 
that, if $u\in \mathcal W^+$, then 
$u(L(f)-R(f)=\sum_i(L(f_i)-R(f_i))u_i$ with $u_i\in\mathcal W^+$. 
Since $\deg(L(f))=\deg(R(f))=0$, if $u\in F^N\otimes\mathcal F$, then
 \begin{equation}\label{FNandideal}
 u(L(f)-R(f))=\sum_i(L(f_i)-R(f_i))u_i
 \end{equation}
 with $u_i\in F^N\otimes\mathcal F$.
 
 It follows that, if $u\in\ov{\mathcal W^k_{\mathcal F}}$, then, for 
each $N\in\nat$, we can find $u'\in\mathcal W^k_{\mathcal F}$ such 
that 
 $$
 u(L(f)-R(f))=u'(L(f)-R(f))+w
 $$
 with $w\in\ov{\mathcal W^k_{\mathcal F}(F^N\otimes\mathcal F)}$. 
Writing $u'=\sum_iu^-_iu^+_i$ with $u^-_i\in U(\n'_-\oplus 
\h_0)\otimes Cl(\bar\n'_-)$ and $u^+_i\in (1\otimes\mathcal 
F)U(\n')\otimes Cl(\bar\n'\oplus\bar\h_0)$, then
 $$
 u(L(f)-R(f))=\sum_iu^-_i(L(f_i)-R(f_i))u^+_i+w.
 $$
 Observe finally that \eqref{tildeandideal}, \eqref{barandideal}, and 
\eqref{cartanandideal} imply that $\mathcal I$ is the left ideal 
generated by $\{L(f)-R(f)\mid f\in
S(\h_0\times\h_\aa\times\C(d-d_\aa))\}$. Summarizing, we can conclude 
that, given $u\in\mathcal I$ and $N\in\nat$, we can always write $u$ 
as
\begin{equation}\label{idealandtriangular}
u=\sum_iu^-_i(L(f_i)-R(f_i))u^+_i+w
\end{equation}
 with $u^-_i\in U(\n'_-\oplus \h_0)\otimes Cl(\bar\n'_-)$, $u^+_i\in 
(1\otimes\mathcal F)U(\n')\otimes Cl(\bar\n'\oplus\bar\h_0)$, and 
$w\in\ov{\mathcal W^k_{\mathcal F}(F^N\otimes\mathcal F)}$.

\subsection{The main filtration}\label{mainfiltration}

Define 
$$
\Aa=\ov{\mathcal W^k_{\mathcal F}}/\mathcal I.
$$

By \eqref{degrepres}, $\mathcal I$ is generated by homogeneous 
elements, hence we can define a
grading 
$$
\Aa=\oplus_j\Aa_j
$$  where $\Aa_j=\ov{\mathcal W^k_j\otimes \mathcal F}/(\mathcal 
I\cap \ov{\mathcal W^k_j\otimes
\mathcal F})$.
Set, for $p\in\nat$
 $$
 \Aa^p=\left(\ov{\mathcal W^k_{\mathcal F}(F^p\otimes\mathcal 
F)}+\mathcal I\right)/\mathcal I.
 $$

We use $\{\Aa^p\}_{p\in\mathbb N}$ as a fundamental system of 
neighborhoods of $0$, and let $\Aa^{com}$
 be the corresponding completion. Set
$$
\ov\Aa=\oplus_j(\ov{\Aa_j}).
$$

Similarly to what we did with $\mathcal W$, the product on $\Aa$ can 
be extended to $\ov\Aa$.
\vskip20pt
\noindent We now start the construction of a sort of PBW-basis for 
$\Aa$.
If  $\{x^1,x^2,\ldots\},$ $\{y^1,y^2,\ldots\}$ are bases of 
$\n',\,\n'_-$
respectively and  $I=\{i_1,i_2,\ldots\}$ is a multi-index with a 
finite number of non zero
elements, we set 
\begin{equation}\label{xy}x^I=(x^1)^{i_1}(x^2)^{i_2}\cdots,\quad
y^I=(y^1)^{i_1}(y^2)^{i_2}\cdots.\end{equation}
  Here $x^I,y^I$ are viewed
as elements of the symmetric algebra of $L(\g,\si)$.

For $I$  a multi-index with $i_j\leq 1$ for all $j$ and a finite 
number of non zero elements, set
$$\wedge^I x=(x^1)^{i_1}\wedge(x^2)^{i_2}\cdots,\qquad \wedge^I 
y=(y^1)^{i_1}\wedge(
y^2)^{i_2}\cdots.$$
Here $\wedge^Ix,\wedge^Iy$ are viewed as elements of the exterior 
algebra $\wedge(L(\g,\si))$.

 Finally, fix a  basis $\{h^i\}$ of $\h_0$. If
 $l=\dim(\h_0)$ and 
  $S=(s_1,\ldots,s_{l})\in\nat^{l}$, we set $h^S=(h^1)^{s_1}\cdots 
(h^l)^{s_l}$ and, if $s_i\le 1$ for all $i$,  $(\wedge^S h)=( 
h^1)^{s_1}\wedge\cdots\wedge( h^l)^{s_l}$. Here $h^S$ is an element 
of the symmetric algebra $S(\h_0)$ and
$\wedge^Sh$ is an element of the exterior algebra $\wedge\h_0$.

Clearly we can choose $x^i=t^r\otimes x$ with $x\in  \g^{\bar r}$ and 
$y^i=t^{-r}\otimes y$ with $y\in  \g^{-\bar r}$. Moreover we can 
assume that $x^i,y^i$ are homogeneous with respect to $\deg$ and that 
$\deg(y^i)=-\deg(x^i)$. 
If $x^j=t^r\otimes x$ with $x\in  \g^{\bar r}$, then we set $\tilde 
x^j=\tilde x_r$ and $\ov x^j=\ov x_r$. Similarly we
define $\tilde y^j$, $\ov y^j$. We also set $\ov h^i=\ov 
h^i_0,\,\tilde h^i=\tilde h^i_0$.

For $I$  a multi-index with a finite number of non zero elements and 
$S\in\nat^l$, define, as in \eqref{xy},
$
\tilde x^I=(\tilde x^1)^{i_1}(\tilde x^2)^{i_2}\cdots,\quad \tilde 
y^I=(\tilde y^1)^{i_1}(\tilde
y^2)^{i_2}\cdots\quad \tilde h^S=(\tilde h^1)^{s_1}\cdots (\tilde
h^l)^{s_l}.$ If  $i_j\leq 1$, similarly define
$\ov x^I,\,\ov y^I$.
If $s_i\le 1$ for $i=1,\dots,l$, set  $\ov h^S=(\ov 
h^1)^{s_1}\cdots(\ov h^l)^{s_l}$.
Here $\tilde x^I$, $\tilde y^I,\ldots$ etc. are seen as elements of 
$\ov{\mathcal W^k_{\mathcal F}}$ or of $\ov \Aa$.
Finally define
$$\deg(I)=\sum_j i_j\deg(x^j),\qquad |I|=\sum_j i_j.
$$
\vskip5pt
If $u\otimes f\in\mathcal W^k_{\mathcal F}$ and $a\otimes g\in 
F^N\otimes \mathcal F$, then, by the triangular decomposition, we can 
write $u$ as a sum of terms of type $u^-u^+$ with $u^-\in 
U(\n'_-\oplus\h_0)\otimes Cl(\bar\n'_-\oplus \bar\h_0)$ and 
$u^+\in\mathcal W^+$. Note also that, by \eqref{prodotto}, $(1\otimes 
\mathcal F)\mathcal W^+=\mathcal W^+\otimes \mathcal F$, hence, by 
\eqref{degfezero}, $a\otimes g\in(1\otimes \mathcal F)F^N$. It 
follows that we can write $(u\otimes f)(a\otimes g)$ as a sum of 
terms of type $u^-f'u^+a^+$ with $u^+\in\mathcal W^+$ and $a^+\in 
F^N$. By applying PBW-theorem, we can write any element of $\mathcal 
W^k_{\mathcal F}(F^N\otimes \mathcal F)$ as a sum of terms of type
\begin{equation}\label{monomials}
\tilde y^I\bar y^L\tilde h^S f\bar h^T\tilde x^H\bar x^K
\end{equation}
with $\deg(H+K)\ge j_N$ and $f\in\mathcal F$. Therefore the elements 
of $\ov{\mathcal W^k_{\mathcal F}(F^N\otimes \mathcal F)}$ are series 
$\sum_{n=N}^\infty u_n$ with $\deg(u_n)$ bounded and $u_n$ a finite 
sum of monomials as in \eqref{monomials} with $\deg(H+K)\ge j_n$.

Using the fact that $\tilde h^S=L(h^S)$, we see that an element of 
$\Aa^p$ is a series $\sum_{n=p}^\infty u_n$ with $\deg(u_n)$ bounded 
and $u_n$ a finite sum
\begin{equation}\label{termofdegn}
u_n=\sum_{\substack{I,L,T,H,K\\ \deg(H+K)\ge j_n}}\tilde y^I\bar y^L 
f_{I,L,T,H,K}\bar h^T\tilde x^H\bar x^K.
\end{equation}
This writing is, however, not unique. To afford uniqueness we need to 
restrict the functions to a suitable subspace of 
$(\h_0^*\oplus\C\d)\times(\h_\aa^*\oplus\C\d_\aa)$.
Define
\begin{align*}
C_{diag}=\{(\L,\l)\in 
&(\h_0^*\oplus\C\d)\times(\h_\aa^*\oplus\C\d_\aa),\mid\\
&(\L+\rho_\s)(h)=(\l+\rho_{\aa\s})(h)\,\forall\,h\in\h_\aa\oplus\C 
d_\aa\},\end{align*}
\begin{align*}I_{diag}=\{f\in
\mathcal F\mid f_{|C_{diag}}=0\}.\end{align*} 
Define $K^p\subset   S(\n'\oplus \n'_-)\otimes
\wedge(\n'\oplus
        \n'_-)$ as
$$K^p=Span\left(y^Ix^H\otimes \wedge^Ly\wedge^Kx\mid  
\deg(H+K)=j_p\right)
        $$ 
($j_p$ is defined in \eqref{defjp}).

The following Lemma gives a sort of PBW-theorem for $\ov{\Aa}$.
 \begin{lemma}\label{Kpisom} The
 map
$$
        y^Ix^H\otimes\wedge^Ly\wedge^Kx\otimes 
f\otimes\wedge^Th\mapsto\tilde y^I \ov y^Lf\ov h^T
\tilde x^H \ov x^K+\ov{\Aa^{p+1}}
$$ extends to
 a linear onto map 
\begin{equation}\mathcal S: K^p \otimes \mathcal F\otimes 
\wedge(\h_0)\to \ov{\mathcal
A^p}/\ov{\mathcal A^{p+1}}.
\end{equation} Moreover $Ker\,\mathcal S=K^p\otimes 
I_{diag}\otimes\wedge(\h_0)$
thus $\mathcal S$ induces a linear isomorphism, still denoted by 
$\mathcal S$,
\begin{equation}\label{S}K^p \otimes \mathcal F_{|C_{diag}}\otimes
\wedge(\h_0)\longrightarrow\ov{\mathcal A^p}/\ov{\mathcal 
A^{p+1}}.\end{equation}
\end{lemma}
\begin{proof}
The first statement amounts to proving that the set $$\{\tilde y^I 
\ov y^Lf\ov h^T
\tilde x^H \ov x^K+\ov{\Aa^{p+1}}\mid \deg(H+K)=j_p\}$$
spans $\ov{\Aa^p}/\ov{\Aa^{p+1}}$, so assume that 
$a+\ov{\Aa^{p+1}}\in\ov{\Aa^p}/\ov{\Aa^{p+1}}$. We can assume that 
$a\in\Aa^p$, hence $a=\sum_{n=p}^\infty u_n$ with $u_n$ as in 
\eqref{termofdegn}. Thus 
$$
a+\ov{\Aa^{p+1}}=u_p+\ov{\Aa^{p+1}}=\sum_{\substack{I,L,T,H,K\\ 
\deg(H+K)= j_p}}\tilde y^I\bar y^L f_{I,L,T,H,K}\bar h^T\tilde 
x^H\bar x^K+\ov{\Aa^{p+1}}.
$$
as desired.

In order to prove the second statement suppose that $$a=\sum \tilde 
y^I\ov y^Lf_{I,L,T,H,K}\ov h^T\tilde x^H\ov x^H\in\ov{\Aa^{p+1}}.$$
Then  $a=\lim a_i$ with $a_i\in\Aa^{p+1}$, hence, since $a\in\Aa$, we 
have that $a\in\Aa^{p+1}$. This says that $a\in(\ov{\mathcal 
W^{k}_{\mathcal F}(F^{p+1}\otimes \mathcal F)}+\mathcal I)/\mathcal 
I$, 
 hence  we can write
$$\tilde y^I\bar y^L f_{I,L,T,H,K}\bar h^T\tilde x^H\bar x^K=u+w'$$ 
with $u\in\mathcal I$
 and $w'\in\ov{\mathcal W^{k}_{\mathcal F}(F^{p+1}\otimes \mathcal 
F)}$. By  \eqref{idealandtriangular} we can write 
$$
u=\sum_iu^-_i(L(f_i)-R(f_i))u^+_i+w
$$
 with $u^-_i\in U(\n'_-\oplus \h_0)\otimes Cl(\bar\n'_-)$, $u^+_i\in 
(1\otimes\mathcal F)U(\n')\otimes Cl(\bar\n'\oplus\bar\h_0)$, and 
$w\in\ov{\mathcal W^k_{\mathcal F}(F^{p+1}\otimes\mathcal F)}$.
 Writing $u^-_i$ as a linear combination of terms $\tilde y^I\ov 
y^L\tilde h^S$, $u^+_i$ as a linear combination of terms $f_i\ov 
h^T\tilde x^H\ov x^K$ and setting $w''=w+w'$ we can write
\begin{align}
        \sum &\tilde y^I\ov y^Lf_{I,L,T,H,K}\ov h^T\tilde x^H\ov x^K\notag
\\&=\sum \tilde y^I\ov y^L\tilde 
h^S(L(p_{I,L,S,T,H,K,i})-R(p_{I,L,S,T,H,K,i}) )
f_i\ov h^T\tilde x^H\ov x^K+w''\label{inI},
\end{align}
with $w''\in\ov{\mathcal W^{k}_{\mathcal F}(F^{p+1}\otimes \mathcal 
F)}$. 

Since $\tilde h^S=L(h^S)$, by substituting $\tilde h^S(L(p)-R(p))$ 
with
$L(h^Sp)-R(h^Sp))$ $-(L(h^S)-R(h^S))R(p)
$, we can assume that $S=0$, thus \eqref{inI} 
simplifies to 
\begin{align}
        \sum &\tilde y^I\ov y^Lf_{I,L,T,H,K}\ov h^T\tilde x^H\ov x^K\notag
\\&=\sum \tilde y^I\ov y^L(L(p_{I,K,J,H,S,i})-R(p_{I,K,J,H,S,i}) )
f_i\ov h^T\tilde x^H\ov x^K+w''\label{inImodificata}.
\end{align}

Fix a basis $\{v_i\}$ of $\h_\aa$. By performing the affine change of 
variables
$(u,v,d-d_\aa)\to (u',v',d-d_\aa)$, where 
$$
\begin{cases}
        u'=u&\text{if $u\in\h_0$,}\\
        v'=(-v,v,0)-(\rho_\s-\rho_{\s\aa})(v)&\text{if $v\in\h_\aa$,}   
\end{cases}
$$
and setting $x_0=d-d_\aa$ and $x_i=v'_i$ for $i>0$,
we can write a polynomial $P\in S(\h_0\times\h_\aa\times\C(d-d_\aa))$ 
as
\begin{equation}\label{decomposeP}      P=P^0+\sum_i x_i{} P_i
\end{equation}
with $P^0,P_i$ suitable polynomials and $P^0=P^0(u',0,0)$ depending 
only on $(u',0,0)$.
By Lemma~\ref{gradtoCdiag}, $L(x_i)\in\ov{\mathcal W^k_{\mathcal 
F}F^{1}}$ for all $i$. Hence 
$$
L(P)=L(P^0)+w^{>0}$$ 
with $w^{>0}\in\ov{\mathcal W^k_{\mathcal F}F^{1}}$. 
We can therefore write that
\begin{align*}
\sum &\tilde y^I\ov y^L(w^{>0}_{I,L,T,H,K,i} )f_i\ov h^T\tilde x^H\ov 
x^K+w''
\\&=\sum \tilde y^I\ov y^Lf_{I,L,T,H,K}\ov h^T\tilde x^H\ov x^K-\sum 
\tilde y^I\ov y^L(L(p^0_{I,L,T,H,K,i}) )f_i\ov h^T\tilde x^H\ov 
x^K\\&+\sum \tilde y^I\ov y^LR(p_{I,L,T,H,K,i})f_i\ov h^T\tilde 
x^H\ov x^K.
\end{align*}
Let $LHS$ denote  the left hand side of the above equation and $RHS$ 
the right hand side. Set
$p'=\inf\{q\mid j_q=\deg(H+K),\ p_{I,L,T,H,K,i}\ne 0\}$. If $p'<p$ 
then $LHS\in\ov{\mathcal
W^k_{\mathcal F}(F^{p'+1})}$. Since $RHS\in\mathcal W^k_{\mathcal F}$ 
we have that
$RHS\in\mathcal W^k_{\mathcal F}F^{p'+1}$. If $\deg(H+K)=j_{p'}$, 
using PBW and comparing
terms, we find that $L(p^0_{I,L,T,H,K,i})$ must be a constant. We can 
clearly assume
this constant 
to be zero,  obtaining  that  $p^0_{I,L,T,H,K,i}=0$, or, 
equivalently, that
$R(p_{I,L,T,H,K,i})\in I_{diag}$. Thus, if we set $P=p_{I,L,T,H,K,i}$ 
in \eqref{decomposeP}
we can write $p_{I,L,T,H,K,i}=\sum_j x_jP_j$ with $P_j$ suitable 
polynomials.
Now remark  that 
 $L(x_jP_j)-R(x_jP_j)=(L(P_j)-R(P_j))L(x_j)+(L(x_j)-R(x_j))R(P_j)$ 
and that $L(x_j)\in\ov{\mathcal W^k_{\mathcal F}F^{1}}$, so, if 
$\deg(H+K)=j_{p'}$,
\begin{align*}
\tilde y^I&\ov y^L(L(P)-R(P) )
f_i\ov h^T\tilde x^H=\sum_j\tilde y^I\ov y^L(L(x_j)-R(x_j))f'_j\ov 
h^T\tilde x^H\\&+
\sum_{\deg(H+K)>j_{p'}}\tilde y^I\ov 
y^L(L(p_{I,K,J,H,S,i})-R(p_{I,K,J,H,S,i}) )
f_i\ov h^T\tilde x^H.
\end{align*}
Collecting terms, \eqref{inImodificata} gets rewritten as
\begin{align}
        \sum &\tilde y^I\ov y^Kf_{I,K,J,H,S}\ov h^S\tilde x^J\ov x^H\notag
\\&=\sum_{\deg(J+K)=j_{p'}}\!\!\sum_j\tilde y^I\ov y^K(L(x_j)-R(x_j) 
)f'_j \ov h^S\tilde x^J\ov x^H\label{inImodificatabis}\\&
+\sum_{\deg(J+K)>j_{p'}} \tilde y^I\ov y^K(L(p_{I,K,J,H,S,i})
-R(p_{I,K,J,H,S,i}) )g_i\ov h^S\tilde x^J\ov x^H+w''.\notag
\end{align}
Moreover degree considerations imply that  we must have that
\begin{equation*}\label{epoi}
\sum_jR(x_j)f'_j=0.
\end{equation*} 
For $j>0$ write $f'_j=p_j+R(x_0)q_j$ with $p_j,q_j\in\mathcal F$ and 
$p_j$ independent of $R(x_0)$.
Then $\sum_{j>0}R(x_j)p_j=0$ and $f'_0=-\sum_{i>0}R(x_j)q_j$. 
Substituting, we find that
\begin{align*}  
\sum_j(&L(x_j)-R(x_j))f'_j\\&=\sum_{j>0}(L(x_j)-R(x_j))R(x_0)
q_j-\sum_{j>0}(L(x_0)-R(x_0))R(x_j)q_j\\&+\sum_{j>0}(L(x_j)-R(x_j))p_j\\        
&=\sum_{j>0}(L(x_j)R(x_0)-L(x_0)R(x_j))q_j+\sum_{j>0}(L(x_j)-R(x_j))p_j\\       
&=\sum_{j>0}L(x_0)(L(x_j)-R(x_j))q_j-\sum_{j>0}L(x_j)(L(x_0)-R(x_0))q_j\\       
&+\sum_{j>0}(L(x_j)-R(x_j))p_j.
\end{align*}

Since $L(x_j)\in\ov{\mathcal W^k_{\mathcal F}F^{1}}$ we obtain that 
\eqref{inImodificatabis} becomes 
\begin{align*}
        \sum &\tilde y^I\ov y^Kf_{I,K,J,H,S}\ov h^S\tilde x^J\ov x^H\notag
\\&=\sum_{\deg(J+K)=j_{p'}}\, \sum_{j>0}\tilde y^I\ov 
y^K(L(x_j)-R(x_j) )p_j\ov h^S\tilde x^J\ov x^H\\&
+\sum_{\deg(J+K)>j_{p'}} \tilde y^I\ov 
y^K(L(p'_{I,K,J,H,S,i})-R(p'_{I,K,J,H,S,i}) )g'_i\ov h^S\tilde x^J\ov 
x^H+w''.\notag
\end{align*}
Repeating the argument for all the variables $x_j$ we can rewrite 
\eqref{inImodificatabis} as
\begin{align}
        &\sum \tilde y^I\ov y^Kf_{I,K,J,H,S}\ov h^S\tilde x^J\ov x^H\notag
\\&=
\sum_{\deg(J+K)>j_{p'}} \tilde y^I\ov 
y^K(L(p''_{I,K,J,H,S,i})-R(p''_{I,K,J,H,S,i}) )g''_i\ov h^S\tilde 
x^J\ov x^H+w''.\notag
\end{align}
 We can therefore assume that $p'\ge p$.
Again we deduce that $RHS\in\mathcal W^k_{\mathcal F}F^{p+1}$. 

Comparing terms we find that, if $\deg(H+K)=j_p$, then 
$R(p_{I,K,J,H,S,i})\in I_{diag}$ and 
$$
f_{I,K,J,H,S}=\sum_iR(p_{I,K,J,H,S,i})f_i\in I_{diag}
$$
as desired.
\end{proof}

We
introduce an increasing filtration on $\ov{\mathcal A^p}/\ov{\mathcal 
A^{p+1}}$
by
\begin{align}\label{filtr2}(\ov{\mathcal A^p}/\ov{\mathcal 
A^{p+1}})_n=Span\left(
\ov h^T\tilde y^J\ov y^Rf \tilde x^H\ov
x^K+\ov{\mathcal A^{p+1}}\mid
|J|+|H|\leq n\right).
\end{align}

Note that the associated graded vector space is given by 
$$
Gr(\ov{\mathcal
A^p}/\ov{\mathcal A^{p+1}})=\bigoplus\limits_n Gr_n(\ov{\mathcal
A^p}/\ov{\mathcal A^{p+1}})$$
where, setting $K^p_n=(S^n(\n'\oplus 
\n'_-)\otimes\wedge(\n'\oplus\n'_-))\cap K^p$,
\begin{align*} Gr_n(\ov{\mathcal
A^p}/\ov{\mathcal A^{p+1}}) =\mathcal S(K^p_n\otimes \mathcal 
F\otimes\wedge\h_0).
\end{align*}

\begin{rem}\label{forrestricting}
Set $J^p=Span(\ov y^J\tilde y^I \ov h^U
\tilde x^H \ov x^K\mid  \deg(H+K)=j_p)$. Lemma~\ref{Kpisom} implies 
that an element $u$ in the closure
 of $\sum_p J^p$   can be  written  uniquely as a series 
$\sum_{i=0}^\infty u_i$ with $u_i\in J^i$. In
particular $u\in\ov{\Aa^p}$ if and only if $u_i=0$ for $i<p$.
\end{rem}

\subsection{The basic complex}
Let ${\bf d}:\ov{\mathcal A}\to \ov{\mathcal A}$ be the
superbracket operator
${\bf d}(a)=[(G_{\g,\aa})_0,a]$. 
Set $$\ov{\mathcal A}^{inv}=\{a\in\ov{\mathcal A}\mid 
[d_\aa,a]=0\}.$$ 
\begin{lemma}\
\begin{enumerate}
\item $\ov{\mathcal A}^{inv}$ is $\bf d$-stable.
\item ${\bf d}^2=0$ on $\ov{\mathcal A}^{inv}$.
\end{enumerate}
\end{lemma}
\begin{proof}By \eqref{Sugawaraction}, $[\tilde L^\g_0,\tilde 
x_r]=-(k+g)r\tilde x_r$ 
and $[\tilde L^\g_0,\ov x_r]=0$. By the definition of the product in 
$\Aa$ and 
\eqref{Sugawaraexpansion} we see that $[\tilde L^\g_0,f]=0$ for $f$ 
in  $\mathcal F$. On the other hand
in \eqref{prodotto} we set  $[d,\tilde x_r]=r\tilde x_r$, $[d,\ov 
x_r]=0$ and $[d,f]=0$ for  $f$ in  $\mathcal
F$.  Thus bracketing with $d$ and $\tilde L^\g_0$ stabilizes the 
subalgebra of $\ov{\Aa}$ generated by $\tilde
x_r,\ov x_r,f$ and if $x$ is in this subalgebra, then $[\tilde 
L^\g_0,x]=-(k+g)[d,x]$.  By  Lemma~\ref{Kpisom} this
subalgebra is dense in $\ov{\Aa}$, hence $$[\tilde 
L^\g_0,x]=-(k+g)[d,x]$$ for all  $x\in\ov{\Aa}$.

To prove the first statement it is enough to show that 
$(G_{\g,\aa})_0\in \ov{\mathcal A}^{inv}$. 
To check that
$[(G_{\g,\aa})_0,d_\aa]=0$ we recall that $d_\aa=d+D$, so
$[(G_{\g,\aa})_0,d_\aa]=[(G_{\g,\aa})_0,d+L^{\ov\p}_0]=[(G_{\g,\aa})_0,-\frac{1}{k+g}\tilde 
L^\g_0+
L^{\ov\p}_0]=[(G_{\g,\aa})_0,-\frac{1}{k+g}(\tilde L^\g_0-\tilde 
L^\aa_0)+L^{\ov\p}_0]$.
By \eqref{gzeroquadro},
$[(G_{\g,\aa})_0,-\frac{1}{k+g}(\tilde L^\g_0-\tilde 
L^\aa_0)+L^{\ov\p}_0]=[(G_{\g,\aa})_0,-
\frac{1}{k+g}(G_{\g,\aa})^2_0]=0$, 
$(G_{\g,\aa})^2_0$ being even.

To prove the second statement, we first notice that ${\bf 
d}^2(a)=[(G_{\g,\aa})_0^2,a]$, ($G_{\g,\aa})_0$ being odd.
 Arguing as above, we see that, if $a\in \ov{\Aa}^{inv}$, 
$[(G_{\g,\aa})_0^2,a]=-(k+g)[d_\aa,a]=0$.
\end{proof}

By \eqref{ass} we can find a basis of $\p$ homogeneous w.r.t. $\deg$. 
Then \eqref{GG} implies that
$(G_{\g,\aa})_0$ has degree $0$. It follows that 
$[(G_{\g,\aa})_0,\ov{\Aa^p}]\subset
\ov{\Aa^p}$. Thus $$[(G_{\g,\aa})_0,\ov{\mathcal 
A^p}\cap\ov{\Aa}^{inv}]\subset \ov{\mathcal
A^p}\cap\ov{\Aa}^{inv}.
 $$
 Set $\ov{\mathcal A^p_{inv}}=\ov{\mathcal A^p}\cap\ov{\Aa}^{inv}$. 
We can therefore  define a
differential 
$$\ov d_p: \ov{\mathcal A^p_{inv}}/\ov{\mathcal A^{p+1}_{inv}}\to 
\ov{\mathcal A^p_{inv}}/\ov{\mathcal
A^{p+1}_{inv}}
$$ by $$\ov d_p(x+\ov{\mathcal A^{p+1}_{inv}})={\bf 
d}(x)+\ov{\mathcal A^{p+1}_{inv}}.$$
\vskip5pt

By \eqref{ass}, $[d_\aa,\ov{\Aa^p}]\subset \ov{\Aa^p}$, hence we can 
define an action of $d_\aa$ on $\ov{\Aa^p}/\ov{\Aa^{p+1}}$. If 
$x\in\ov{\Aa^p}/\ov{\Aa^{p+1}}$, let us write $d_\aa\cdot x$ for this 
action. 
Using the isomorphism $\mathcal S$ we can lift this action to 
$K^p\otimes \mathcal F_{|C_{diag}}\otimes
\wedge\h_0$. In order to describe explicitly this action, set 
$$\n'_\p=\n\cap\p\oplus\sum_{r>0}t^r\otimes(\g^{\ov
r}\cap\p),\quad\n'_\aa=\n\cap\aa\oplus\sum_{r>0}t^r\otimes(\g^{\ov 
r}\cap\aa).
$$
Analogously define $(\n'_-)_\p$ and $(\n'_-)_\aa$.
Since $[d_\aa,\tilde x_r]=r\tilde x_r$ for $x\in\g^{\ov r}$, 
$[d_\aa,\ov x_r]=r\ov x_r$ for $x\in\g^{\ov r}\cap\p$, $[d_\aa,\ov 
x_r]=0$ for $x\in\g^{\ov r}\cap\aa$, and $[d_\aa,f\ov h^T]=0$, we see 
that $d_\aa$ acts 
on $K^p\otimes \mathcal F_{|C_{diag}}\otimes \wedge\h_0$ as the 
derivation on $S(\n'\oplus\n'_-)$ such that
$d_\aa\cdot(t^r\otimes x)=r(t^r\otimes x)$ and as the even derivation 
on $\wedge(\n'_\p\oplus
        (\n'_-)_\p)$ such that $d_\aa\cdot (t^r\otimes x)=r(t^r\otimes x)$. 
Moreover $d_\aa$ acts trivially on $\wedge(\n'_\aa\oplus
(\n'_-)_\aa)$ and on $\mathcal F_{|C_{diag}}\otimes \wedge \h_0$. 

By \eqref{ass}, we can choose the basis $\{x^i\}$ as  the union of a 
basis of $\n'_\p$ with
a basis of $\n'_\aa$. Likewise, the basis $\{y^i\}$ can be chosen as 
the union of bases
of $(\n'_-)_\p$ and $(\n'_-)_\aa$. With this choice of bases, it is 
clear that the monomials
$y^Ix^J\otimes \w^Hy\w^Kx$ are eigenvectors for the action of 
$d_\aa$. In particular the
action of $d_\aa$ is semisimple.
 We can therefore write $K^p=\oplus_sK^{p,s}$ where
$K^{p,s}$ denotes the $s$-eigenspace of $K^p$ under the action of 
$d_\aa$. Let $(\ov{\Aa^p}/\ov{\Aa^{p+1}})_{inv}$ be the 
$d_\aa$-invariant subspace of $\ov{\Aa^p}/\ov{\Aa^{p+1}}$.  By 
construction $(\ov{\Aa^p}/\ov{\Aa^{p+1}})_{inv}=
\mathcal S(K^{p,0}\otimes
\mathcal F\otimes \wedge \h_0)$. On the other hand 
$\ov{\Aa^p_{inv}}/\ov{\Aa^{p+1}_{inv}}$ embeds in
$(\ov{\Aa^p}/\ov{\Aa^{p+1}})_{inv}$. Moreover, if 
$y^Ix^J\otimes\wedge^Hy\wedge^Kx\otimes f\otimes \wedge^T h\in
K^{p,0}\otimes\mathcal F\otimes
\wedge\h_0$, then, clearly, $\ov y^I\tilde y^Hf\ov h^T\tilde x^J\ov 
x^K$ commutes with $d_\aa$, hence the
embedding of $\ov{\Aa^p_{inv}}/\ov{\Aa^{p+1}_{inv}}$ in 
$(\ov{\Aa^p}/\ov{\Aa^{p+1}})_{inv}$ is onto. In
particular we have an isomorphism
        \begin{equation}\label{Kpzero}
        \mathcal S:K^{p,0}\otimes\mathcal F_{|C_{diag}}\otimes \wedge\h_0\to 
\ov{\Aa^p_{inv}}/\ov{\Aa^{p+1}_{inv}}.
        \end{equation}

Since $d_\aa\cdot S^n(\n'\oplus\n'_-)\subset S^n(\n'\oplus\n'_-)$ we 
can write $K^{p,r}=\oplus_n K^{p,r}_n$, where $K^{p,r}_n=K^{p,r}\cap 
K^p_n$. Set 
$(\ov{\Aa^p_{inv}}/\ov{\Aa^{p+1}_{inv}})_n=(\ov{\Aa^p_{inv}}/\ov{\Aa^{p+1}_{inv}})\cap(\ov{\Aa^p}/\ov{\Aa^{p+1}})_n$, 
so the isomorphism in  \eqref{Kpzero} gives an isomorphism  
$$\oplus_{m\le n}(K^{p,0}_m\otimes\mathcal
F_{|C_{diag}}\otimes\wedge \h_0)\to 
(\ov{\Aa^p_{inv}}/\ov{\Aa^{p+1}_{inv}})_n.$$ Setting
$Gr_n(\ov{\Aa^p_{inv}}/\ov{\Aa^{p+1}_{inv}})= 
(\ov{\Aa^p_{inv}}/\ov{\Aa^{p+1}_{inv}})_n
\slash (\ov{\Aa^p_{inv}}/\ov{\Aa^{p+1}_{inv}})_{n-1}$, the map 
$\mathcal S$ induces an isomorphism
\begin{equation}\label{Kpzeron}
        \mathcal S:K^{p,0}_n\otimes(\mathcal F_{|C_{diag}}\otimes 
\wedge\h_0)\to
Gr_n(\ov{\Aa^p_{inv}}/\ov{\Aa^{p+1}_{inv}}),
        \end{equation}
        and, setting 
$Gr((\ov{\Aa^p_{inv}}/\ov{\Aa^{p+1}_{inv}}))=\oplus_nGr_n(\ov{\Aa^p_{inv}}/\ov{\Aa^{p+1}_{inv}})$,        
an isomorphism
        \begin{equation}\label{iso}
        \mathcal S:K^{p,0}\otimes(\mathcal F_{|C_{diag}}\otimes 
\wedge\h_0)\to
Gr(\ov{\Aa^p_{inv}}/\ov{\Aa^{p+1}_{inv}}).
        \end{equation}
By the decompositions $\n'=\n'_\p\oplus\n'_\aa$, 
$\n'_-=(\n'_-)_\p\oplus(\n'_-)_\aa$, we have
\begin{align*}
        S(\n'\oplus \n'_-)&=S(\n'_\p\oplus (\n'_-)_\p)\otimes 
S(\n'_\aa\oplus (\n'_-)_\aa),\\
        \wedge(\n'\oplus \n'_- )&=\wedge(\n'_\p\oplus (\n'_-)_\p)\otimes 
\wedge(\n'_\aa\oplus
        (\n'_-)_\aa).
\end{align*}

Set
\begin{align*}
&K^p_\p=K^p\cap (S(\n'_\p\oplus (\n'_-)_\p)\otimes 
\w(\n'_\p\oplus
        (\n'_-)_\p)),\\&K^p_\aa=K^p\cap (S(\n'_\aa\oplus (\n'_-)_\aa)\otimes 
\w(\n'_\aa\oplus
        (\n'_-)_\aa)),
        \end{align*}
        and define $K^{p,s}_\aa$, $K^{p,s}_\p$, $(K^{p,s}_n)_\aa$, 
$(K^{p,s}_n)_\p$  analogously.
        Since $d_\aa\cdot K^p_\aa\subset K^p_\aa$ and $d_\aa\cdot 
K^p_\p\subset K^p_\p$, we 
  can write 
  \begin{equation}\label{decompaandp}
  K^{p,s}=\sum_{\substack{r+t=p\\ a+b=s}}(K^{r,a})_\p\otimes 
(K^{t,b})_\aa.
  \end{equation} 

Let $\partial_{\p_-}$, $\partial_{\p_+}$ be  the Koszul differentials 
on $S((\n'_-)_\p)\otimes \w (\n'_-)_\p$ and $S(\n'_\p)\otimes \w 
\n'_\p$ respectively.
Consider the complex $C^u=S((\n'_-)_\p)\otimes \w^u (\n'_-)_\p\otimes 
S((\n'_-)_\aa)$ endowed with the Koszul
differential
$\partial_{\p_-}\otimes I$ (denoted for for shortness 
$\partial_{\p_-}$ in the following) and the complex
$D^v=(S(\n'_\p)\otimes
\w^v
(\n'_\p))\otimes
\w\left((\n'_-)_\aa\oplus\n'_\aa\oplus\h_0\right)\otimes\mathcal 
F_{|C_{diag}}\otimes S((\n')_\aa)$ endowed with
the ``signed" Koszul differential 
$\partial^{sign}_{\p_+}$ defined as $\partial_{\p_+}\otimes 
(-1)^{j}I\otimes I\otimes I\otimes I$
on $(S(\n'_\p)\otimes \w^v (\n'_\p))\otimes
\w^j\left((\n'_-)_\aa\oplus\h_0\right)\otimes\w\n'_\aa\otimes\mathcal 
F_{|C_{diag}}\otimes S((\n')_\aa)$. This endows 
$S(\n'\oplus\n'_-)\otimes \w(\n'\oplus\n'_-)\otimes \mathcal 
F_{|C_{diag}}\otimes \wedge\h_0=C\otimes D$ with the differential 
$\partial_{\p_-}\otimes \partial^{sign}_{\p_+}$. By K\"unneth formula 
this differential 
is exact except in degree zero and $$H_0(C\otimes 
D)=S(\n'_\aa\oplus(\n'_-)_\aa)\otimes
\w(\n'_\aa\oplus(\n'_-)_\aa)\otimes \mathcal F_{|C_{diag}}\otimes 
\wedge\h_0.$$

Since $K^p\otimes \mathcal F_{|C_{diag}}\otimes \wedge \h_0$ and 
$K^{p,s}\otimes \mathcal
 F_{|C_{diag}}\otimes \wedge \h_0$ are  subcomplexes of 
$S(\n'\oplus\n'_-)\otimes \w(\n'\oplus\n'_-)\otimes\mathcal
F_{|C_{diag}}\otimes\w\h_0$, the differential $\partial_{\p_-}\otimes
\partial^{sign}_{\p_+}$ restricts to differentials on $K^p\otimes 
\mathcal
F_{|C_{diag}}\otimes
\wedge \h_0$, $K^{p,s}\otimes \mathcal F_{|C_{diag}}\otimes \wedge 
\h_0$, which we denote by $\partial_p$, $\partial_{p,s}$
respectively. Since
$S(\n'\oplus\n'_-)\otimes
\w(\n'\oplus\n'_-)\otimes\mathcal 
F_{|C_{diag}}\otimes\w\h_0=\oplus_{p,s}(K^{p,s}\otimes \mathcal 
F_{|C_{diag}}\otimes \wedge
\h_0)$, we have that
$\partial_{p,s}$  is exact except in degree zero and, by 
\eqref{decompaandp},
\begin{equation}\label{coomologiazero}H_0(\partial_{p,s})=K^{p,s}_\aa\otimes\mathcal 
F_{|C_{diag}}
\otimes \w\h_0.\end{equation}

\begin{lemma}\label{dizero} We have 
\begin{equation}\label{G0phi}
[(G_{\g,\aa})_0,f]\in \ov{\mathcal A^1}.
\end{equation}
Set $\ov h^{T_j}=\ov h^{t_1}_0\cdots\widehat{\ov h^{t_j}_0}\cdots\ov 
h^{t_k}_0$. Then, if $h^{t_j}\in\h_\p$ for each $j$,
\begin{equation}\label{G0barh}
[(G_{\g,\aa})_0,\ov h^T]=\sum_j (-1)^j(\tilde h_0^{t_j}+
\rho_\si(h^{t_j})) 
\ov h^{T_j} + \ov u.
\end{equation}
with $\ov u\in\ov{\Aa^1}$.
\end{lemma} 
\begin{proof}
The first relation is proven via a direct computation.\par We write 
$(G_{\g,\aa})_0$ explicitly using \eqref{GG} and \eqref{tensore}.
 The cubic term turns out to be
\begin{align*}
&\sum_{i,j}\left(\sum_{h,k=0}^\infty(\ov{ 
[b_i,b_j]_\p}_{s_i+s_j-h-1}\bar b^i_{-s_i-k-1}\bar 
b^j_{-s_j+h+k+2}-\right.\\&\ov{ [b_i,b_j]_\p}_{s_i+s_j-h-1}\bar 
b^j_{-s_j+h-k+1}\bar b^i_{-s_i+k})
+\sum_{h,k=0}^\infty(\bar b^i_{-s_i-k-1}\bar 
b^j_{-s_j-h+k+1}\ov{[b_i,b_j]_\p}_{s_i+s_j+h}\\&\left.-\bar 
b^j_{-s_j-h-k}\bar b^i_{-s_i+k}\ov{[b_i,b_j]_\p}_{s_i+s_j+h})\right)
+3\sum_{i}s_i\ov{[b_i,b^i]_\p}_0. 
\end{align*}
Hence the bracket of $f$ with the cubic term is $0$. Let us now 
consider the quadratic term. This turns out to be
$$
\sum_i \sum_{h\in\ganz}(\tilde b_i)_{s_i-h}\bar b^i_{-s_i+h}.
$$ 
If $\deg(\bar b^i_{-s_i+h})>0$ then, since $\deg(f)=0$, $\deg([f,\bar 
b^i_{-s_i+h}])>0$. It follows that $[f,(\tilde b_i)_{s_i-h}\bar 
b^i_{-s_i+h}]=[f,(\tilde b_i)_{s_i-h}]\bar b^i_{-s_i+h}+(\tilde 
b_i)_{s_i-h}[f,\bar b^i_{-s_i+h}]\in\ov{\Aa^1}$. If $\deg(\bar 
b^i_{-s_i+h})<0$ then $\deg((\tilde b_i)_{s_i-h})>0$, hence 
$\deg([f,(\tilde b_i)_{s_i-h})])>0$ and $[f,(\tilde b_i)_{s_i-h}\bar 
b^i_{-s_i+h}]=\bar b^i_{-s_i+h}[f,(\tilde b_i)_{s_i-h}]+[f,\bar 
b^i_{-s_i+h}](\tilde b_i)_{s_i-h}\in\ov{\Aa^1}$. Finally, if 
$\deg(\bar b^i_{-s_i+h})=0$, then, since we are assuming that the centralizer of $\h_\aa$ in $\g^{\bar 0}$ is $\h_0$, we have that $(\tilde b_i)_{s_i-h}\in\h_0$ and 
$\bar b^i_{-s_i+h}\in\bar\h_0$. Thus, in this case $[f,(\tilde 
b_i)_{s_i-h}\bar b^i_{-s_i+h}]=0$.

For the second equality we can argue by induction on $|T|$, the 
result being obvious if $|T|=0$. 
To deal with the case  $|T|>0$, recall that for  $x\in\p$, we write
$\gamma(x)=\half\sum_i:\overline{[x,b_i]}_\p\ov b^i:$. Then, by \eqref{conxbar}, we have $$
[(G_{\g,\aa})_0,\ov h^T]=(\tilde h_0^{t_1}+\gamma(h^{t_1})_{0})
(\ov h_0^{t_2}\cdots \ov h_0^{t_k})+(
\ov h_0^{t_1})[(G_{\g,\aa})_0,\ov h_0^{t_2}\cdots \ov h_0^{t_k}],$$ 
so, by the induction hypothesis,
we need only to check that
$\gamma(h)_{(0)}\equiv\rho_{\si}(h)\mod \ov{\Aa^1}$. We have
\begin{align}\label{prima}
\gamma(h)_{0}&\equiv -\half\sum_{i}(\ov b^i)_0
(\overline{[h,b_i]}_\p)_0-\half\sum_i(s_i+\half)([h,b_i],b^i)\\\notag&\equiv
\half(-\sum_{b_i\in\mathfrak n_-\cap\p}( 
[h,b_i],b^i)-\sum_i(s_i+\half)([h,b_i],b^i))\\\notag
&\equiv\half(-\sum_{x_i\in\mathfrak n_-}( 
[h,x_i],x^i)+\sum_{i,0<s_i\le\half}(1-2s_i)([h,x_i],x^i))\\\notag
&\equiv(h,\half(-\sum_{x_i\in\mathfrak n_-} 
[x_i,x^i]+\sum_{i,0<s_i\le\half}(1-2s_i)[x_i,x^i]))\\\notag
&\equiv\rho_\si(h).
\end{align}

\end{proof}
\begin{lemma}\label{deDp} The differential $\ov d_p$ maps 
$(\ov{\mathcal
A^p_{inv}}/\ov{\mathcal A^{p+1}_{inv}})_n$ to  $(\ov{\mathcal
A^p_{inv}}/\ov{\mathcal A^{p+1}_{inv}})_{n+1}$ and the induced 
differential  on $Gr(\ov{\mathcal
A^p_{inv}}/\ov{\mathcal A^{p+1}_{inv}})$ 
is $\partial_{p,0}$ (under the identification
\eqref{iso}).
\end{lemma}
\begin{proof}Fix $p\in\nat$. It suffices to prove that,  if 
$z+\ov{\Aa^{p+1}}=\mathcal S(z')$ with $z'\in K^{p}_n\otimes \mathcal
F_{|C_{diag}}\otimes \wedge\h_0$, then  
$[(G_{\g,\aa})_0,z]+\ov{\Aa^{p+1}}=\mathcal S(\partial_{p}(z'))+w$ 
with
$w\in(\ov{\mathcal A^p}/\ov{\mathcal A^{p+1}})_n$.  \par 
Note that, if $w\in\mathcal W^k_{\mathcal F}$, we can find 
$u_1\in\mathcal W^k_{\mathcal F}$ and $u_2\in\ov{\mathcal 
W^k_{\mathcal F}F^{p+1}}$ such that $uw=u_1w+u_2$. This is easily 
checked using \eqref{degfezero}. 

For $n\in\nat$, let $\mathcal W^k_{\mathcal F}(p,n)$ be the span of 
monomials $\tilde
y^I\bar y^J\tilde h^Sf\bar h^T\tilde x^H\bar x^K$, with $\deg(H+K)\ge 
j_p$ and $|I|+|H|\le
n$. Set also $\mathcal W^k_{\mathcal F}(n)=\mathcal W^k_{\mathcal 
F}(0,n)$. The ordinary PBW
theorem for $\mathcal W$ combined with \eqref{prodotto} shows that, 
if $u_1\in\mathcal
W^k_{\mathcal F}(m)$ and $\deg(H+K)\ge j_p$, then
$$
u_1(\tilde y^I\bar y^J\tilde h^Sf\bar h^T\tilde x^H\bar 
x^K)\in\mathcal W^k_{\mathcal F}(p,m+|I|+|H|).
$$
Moreover, if $\deg(u_1)=j_{q_1}$, $\deg(H_1+H_2+K_1+K_2)=j_{q_2}$, 
then 
$$
\tilde y^I\bar y^J\tilde h^Sf\bar h^T\tilde x^{H_1}\bar 
x^{K_1}u_1\tilde x^{H_2}\bar x^{K_2}\in\mathcal W^k_{\mathcal 
F}(q_1+q_2,m+|H_1|+|H_2|+|I|).
$$

We now apply the above observations to the computation of 
$$[(G_{\g,\aa})_0,\tilde y^I\ov y^Jf\ov h^T \tilde x^H\ov x^K]$$
with $\deg(H+K)=j_p$ and $|H|+|I|=n$.
Clearly
\begin{align*}
        &[(G_{\g,\aa})_0,\tilde y^I\ov y^J f\ov h^T \tilde x^H \ov x^K]
=
[(G_{\g,\aa})_0,\tilde y^I]\ov y^J f \ov h^T\tilde x^H \ov x^K+\\
&\tilde y^I[(G_{\g,\aa})_0,\ov y^J] f \ov h^T\tilde x^H \ov x^K+
(-1)^{|J|}\tilde y^I\ov y^J[(G_{\g,\aa})_0, f \ov h^T]
\tilde x^H\ov x^K+\\
&+(-1)^{|J|+|T|}\tilde y^I\ov y^J f \ov h^T[(G_{\g,\aa})_0,\tilde 
x^H] \ov x^K+(-1)^{|J|+|T|}\tilde y^I\ov y^J f \ov
h^T\tilde x^H [(G_{\g,\aa})_0,\ov x^K].
        \end{align*}
 By the above observations, applying formula \eqref{conxtilde}, we 
find that
$$[(G_{\g,\aa})_0,\tilde y^I]\ov y^J f \ov h^T\tilde x^H \ov 
x^K\text{ and }\tilde y^I\ov y^J f \ov h^T[(G_{\g,\aa})_0,\tilde x^H] 
\ov x^K\in \mathcal W^k_{\mathcal F}(p,n)+ \ov{\mathcal W^k_{\mathcal 
F}F^{p+1}},$$
 while Lemma~\ref{dizero} implies that
 $$\tilde y^I\ov y^J[(G_{\g,\aa})_0, f \ov h^T]
\tilde x^H\ov x^K\in\mathcal W^k_{\mathcal F}(p,n)+ \ov{\mathcal 
W^k_{\mathcal F}F^{p+1}}.
$$
It follows that
\begin{align*}
        &[(G_{\g,\aa})_0,\tilde y^I\ov y^J\ov y^S f \ov h^T\tilde x^H\ov x^K]
+\ov{\Aa^{p+1}}\equiv
\tilde y^I[(G_{\g,\aa})_0,\ov y^J] f \ov h^T\tilde x^H\ov x^K+\\
&(-1)^{|J|+|T|}\tilde y^I\ov y^J f\ov h^T \tilde x^H
 [(G_{\g,\aa})_0,\ov x^K]+
\ov{\Aa^{p+1}}\mod (\ov{\Aa^p}/\ov{\Aa^{p+1}})_n.
        \end{align*}
Writing explicitly $\ov x^K=\ov x^{n_1}\cdots\ov x^{n_t}$ and $\ov 
y^J=\ov y^{m_1}\cdots\ov y^{m_t}$ we set
$\ov x^{K_i}=\ov x^{n_1}\cdots\widehat{\ov x^{n_i}}\cdots\ov x^{n_t}$ 
and $\ov y^{J_i}=\ov
y^{m_1}\cdots\widehat{\ov y^{n_i}}\cdots\ov y^{m_t}$. The above 
observations imply, by \eqref{conxbar} and the fact that 
$[(G_{\g,\aa})_0,\bar a]=0$ if $a\in\aa$, that
\begin{align*}
        [(G_{\g,\aa})_0&,\tilde y^I\ov y^J f h^T\tilde x^H\ov 
x^K]+\ov{\Aa^{p+1}}
        \equiv\sum_{i:y^{m_i}\in(\n'_-)_\p}(-1)^i\tilde y^I \tilde 
y^{m_i}\ov y^{J_i}f \ov
h^T\tilde x^H \ov x^K\\
        &+\sum_{i:x^{m_i}\in(\n')_\p}(-1)^{|J|+|T|+i}\tilde y^I\ov y^{J} f 
\ov h^T\tilde x^{n_i} \tilde x^H \ov
x^{K_i}+\ov{\Aa^{p+1}}\mod(\ov{\Aa^p}/\ov{\Aa^{p+1}})_n
        \end{align*}
as wished.
        \end{proof}

We also need the following  (possibly known) fact. It can be proved 
either generalizing Lemmas 4.3, 4.4  in \cite{KacD} to the series
case, 
or extending (as we do in Section~\ref{newKoszul}) the standard 
homotopy argument which proves the
exactness of the Koszul complex. Let $\mathcal E(z_1,\ldots,z_n)$ 
denote the algebra of entire functions in $n$
complex variables
$z_1,\ldots,z_n$. Let $\xi$ be an odd variable and set $\xi_i=z_i\xi$.
\begin{lemma}\label{koszulholomorphic} Fix $h_0\in\C^n$. Consider the 
complex 
$$R^p=\mathcal E(z_1,\ldots,z_n)\otimes \w^p(\xi_1,\dots,\xi_n)$$ 
endowed with the Koszul
 differential \begin{equation}\label{kzd}\partial_{h_0}(f\otimes \xi_1\w\cdots\w 
\xi_p)=\sum_{k=1}^{p}(-1)^k
(z_k- z_k(h_0))f\otimes  \xi_1\w\cdots\w \widehat{\xi_k}\w\cdots\w 
\xi_p.\end{equation}
Then $H_j(R^\bullet)=0$ for $j>0$.  Moreover, given $f\in\mathcal 
E(z_1,\ldots,z_n)$, then $f=f(h_0)+\partial_{h_0}(g)$ with $g\in R^1$.
\end{lemma}

If $x^j=x_r$ with $x\in\aa^{\ov r}$, then we set $\tilde 
x^j_\aa=(\tilde x)_{\aa r}$ and
 $\theta(x^j)=\theta(x)_r$ ($\theta$ is as in \eqref{thetaref}). For 
a multi-index $J=(j_1,j_2,\dots)$ 
we set $\tilde x^J_\aa=x^{j_1}_\aa x^{j_2}_\aa\dots$. We define 
$y^j_\aa$ and $y^J_\aa$ similarly.
We shall need the following technical lemma, whose proof is in 
Section~\ref{lemmarestrizione}.
\begin{lemma}\label{restriction}

\item\label{restriction2} If $\deg(H+K)=j_p$ and $|J+H|=n$ then
        $$\ov y^I\tilde y^J_\aa f\ov h^T\tilde x^H_\aa\ov x^K=\ov y^I\tilde 
y^Jf\ov h^T\tilde x^H\ov x^K+u$$
        with $u\in \ov{\Aa^p}$ and 
$u+\ov{\Aa^{p+1}}\in\left(\ov{\Aa^p}/\ov{\Aa^{p+1}}\right)_{n-1}$.

\end{lemma}

The main use we will make of Lemma~\ref{restriction2} is highlighted 
in the following Corollary.
\begin{cor}\label{restrictionKoszul}
If $a=p\otimes f\otimes \w^T h$ with $p\in (K^{p,0}_{n})_\aa$, 
$f\in(\mathcal F_{|C_{diag}}\otimes\w\h_\p)$, and $\w^T h\in\w\h_\aa$ 
then
$$
\bar d_p(\mathcal S(a))=\mathcal S(p\otimes 
\partial_{-(\rho_\s)_{|\h_\p}} (f)\otimes\w^T h)+\bar d_p(v)+v'
$$
with $v,v'\in(\ov{\mathcal
        A^p_{inv}}/\ov{\mathcal A^{p+1}_{inv}})_{n-1}$. (See \eqref{kzd} Êfor the definition
of $\partial_{-(\rho_\s)_{|\h_\p}}$). 
\end{cor}
\begin{proof}
Write explicitly $\mathcal S(a)$ as a sum of terms of type $\ov 
y^I\tilde y^Jf\ov h^T\ov h^S\tilde x^H\ov x^K+\ov{\Aa^{p+1}}$ with 
$f\in\mathcal F_{|C_{diag}}$, $\w^T h\in \w\h_\p$, and $\w^S h\in 
\w\h_\aa$.
 Then, by  Lemma~\ref{restriction}, we have that $\ov y^I\tilde 
y^Jf\ov h^T\ov h^S\tilde x^H\ov
x^K+\ov{\Aa^{p+1}}=\ov y^I\tilde y^J_\aa f\ov h^T\tilde x^H_\aa\ov 
x^K+\ov{\Aa^{p+1}}+v$, with $v\in(\ov{\mathcal
        A^p}/\ov{\mathcal A^{p+1}})_{n-1}$. If $x\in\aa$ then  
$[d_\aa,(\tilde x)_{\aa r}]=r(\tilde x)_{\aa r}$, hence we have that 
$[d_\aa,\ov y^I\tilde y^J_\aa f\ov h^T\tilde x^H_\aa\ov x^K]=0$, thus 
$v\in(\ov{\mathcal
        A^p_{inv}}/\ov{\mathcal A^{p+1}_{inv}})_{n-1}$. Hence $\ov 
d_p(\mathcal S(a))$ is a sum of terms of type
$\ov y^I\tilde y^J_\aa [(G_{\g,\aa})_0,f\ov h^T]\ov h^S\tilde 
x^H_\aa\ov x^K+\ov{\Aa^{p+1}}+\ov d_p(v)$. By Lemma~\ref{deDp} we 
obtain that 
$\ov d_p(\mathcal S(a))$ is a sum of terms of type $\ov
y^I\tilde y^J_\aa \partial_{-(\rho_\s)_{|\h_\p}}(f\ov h^T)\ov 
h^S\tilde x_\aa^H\ov x^K+\ov{\Aa^{p+1}}+\ov d_p(v)$.
Applying Lemma~\ref{restriction}  again, we find that $\ov 
d_p(\mathcal S(a))$ is a sum of terms of type $\ov
y^I\tilde y^J \partial_{-(\rho_\s)_{|\h_\p}}(f\ov h^T)\ov h^S\tilde 
x^H\ov x^K+\ov{\Aa^{p+1}}+\ov d_p(v)+v'$, with
$v,v'\in(\ov{\mathcal
        A^p_{inv}}/\ov{\mathcal A^{p+1}_{inv}})_{n-1}$.
\end{proof}
\vskip10pt
        Let $\mathcal F_\aa$  be the set of holomorphic functions on
$\h_\aa^*\oplus\C\d_\aa$ viewed as a subset of $\mathcal F$. Consider 
the
subalgebra $\Aa(\aa)$ of $\ov\Aa$ generated by $\{a_r,\ov a_r\mid
a\in\aa\}\cup\mathcal F_\aa$. Set $\Aa^p(\aa)=\Aa^p\cap\Aa(\aa)$ and
$\ov{\Aa^p_{inv}(\aa)}=\ov{\Aa^p(\aa)}\cap\ov{\Aa}^{inv}$. Notice
that $[(G_{\g,\aa})_0,\ov{\Aa(\aa)}]=0$.

        \begin{prop}\label{H(G)graded=H(a)} The embedding  
$\ov{\Aa^p_{inv}(\aa)}\to \ov{\Aa^p_{inv}}$
        induces an isomorphism 
        \begin{equation}\label{e01}
        \ov{\Aa^p_{inv}(\aa)}\big/\ov{\Aa^{p+1}_{inv}(\aa)}\simeq H(\ov d_p).
\end{equation}
\end{prop}
        \begin{proof}We will  show by induction on $n$ that, if $x\in Ker\, 
\ov d_p\cap (\ov{\mathcal
        A^p_{inv}}/\ov{\mathcal A^{p+1}_{inv}})_n$, then there
 is $a\in \ov{\Aa^p_{inv}(\aa)}\big/\ov{\Aa^{p+1}_{inv}(\aa)}$ and 
$b\in
\ov{\Aa^p_{inv}}\big/\ov{\Aa^{p+1}_{inv}}$ such that  $x=a+\ov 
d_p(b)$.

 We set $(\ov{\mathcal
        A^p}/\ov{\mathcal A^{p+1}})_{-1}=\{0\}$ so the result is obvious for 
$n=-1$. If $n\ge0$ we have that 
        $$
x=\sum_i \mathcal S(u_i)+x',
        $$
with $u_i\in K^{p,0}_n\otimes\mathcal F_{|C_{diag}}\otimes\w\h_0$ and 
$x'\in(\ov{\mathcal
        A^p_{inv}}/\ov{\mathcal A^{p+1}_{inv}})_{n-1}$. 
        By Lemma~\ref{deDp}, $\ov d_p(x)=\sum_i\mathcal 
S(\partial_{p,0}(u_i))+x''$ with $x''\in (\ov{\mathcal
        A^p_{inv}}/\ov{\mathcal A^{p+1}_{inv}})_{n}$ so, since $\ov 
d_p(x)=0$, $\sum_i\mathcal
S(\partial_{p,0}(u_i))\in(\ov{\mathcal
        A^p_{inv}}/\ov{\mathcal A^{p+1}_{inv}})_{n}$. Since
 $\partial_{p,0}(u_i)\in K^{p,0}_{n+1}\otimes\mathcal 
F_{|C_{diag}}\otimes\h_0$, we have that
$\partial_{p,0}(u_i)=0$. By \eqref{coomologiazero}, we have that
$u_i=a_i+\partial_{p,0}(u'_i)$ with $a_i\in 
K^{p,0}_\aa\otimes\mathcal F_{|C_{diag}}\otimes\w\h_0$ and  
$u'_i\in \oplus_{m<n}K^{p,0}_{m}\otimes\mathcal 
F_{|C_{diag}}\otimes\w\h_0$. Set $b'=\sum \mathcal S(u'_i)$.
It follows that 
        $$\ov d_p(b')=\sum \mathcal S(\partial_{p,0}(u'_i))+b''=\sum 
\mathcal S(u_i)+b''-\sum \mathcal S(a_i)$$
        with $b''\in(\ov{\mathcal
        A^p_{inv}}/\ov{\mathcal A^{p+1}_{inv}})_{n-1}$. Hence $x=\sum 
\mathcal S(a_i)+\ov d_p(b')+u$ with
$u\in(\ov{\mathcal
        A^p_{inv}}/\ov{\mathcal A^{p+1}_{inv}})_{n-1}$.

        It follows that, if we write explicitly
        $a_i=\sum_Tp_i\otimes f^i_{T}\otimes\wedge^T h$
with $p_i\in (K^{p,0}_n)_\aa$, $f^i_{T}\in\mathcal 
F_{|C_{diag}}\otimes\w\h_\p$, $\wedge^T h\in\wedge \h_\aa$, then, by 
Corollary~\ref{restrictionKoszul}, 
\begin{align*}
        \ov d_p(\sum \mathcal S(a_i))=\sum \mathcal
S(p_i\otimes\partial_{-(\rho_\s)_{|\h_\p}}(f^i_{T})\otimes\wedge^T
h)\!+\ov d_p(v)+v'
\end{align*}
with $v,v'\in(\ov{\mathcal
        A^p_{inv}}/\ov{\mathcal A^{p+1}_{inv}})_{n-1}$. 
        
        Since $\ov d_p(x)=0$ we see that $\ov d_p(u)=-\ov d_p(\sum \mathcal 
S(a_i))$, so 
        \begin{align*}
        \sum \mathcal
S(p_i\otimes\partial_{-(\rho_\s)_{|\h_\p}}(f^i_{T})\otimes\wedge^T
h)=\ov d_p(-v-u)-v'.
\end{align*}
 Since $-v-u\in(\ov{\mathcal
        A^p_{inv}}/\ov{\mathcal A^{p+1}_{inv}})_{n-1}$, by Lemma~\ref{deDp}, 
$\ov d_p(-v-u)=\sum \mathcal
S(x_i)+u'$ with $x_i\in \sum_{\substack{r+s=p\\ m>0}}K^r_\aa\otimes 
(K^s_\p\cap K^s_m)\otimes\mathcal
F_{|C_{diag}}\otimes\h_0$ and $u'\in(\ov{\mathcal
        A^p}/\ov{\mathcal A^{p+1}})_{n-1}$, thus, comparing terms, we find
        $$
\partial_{-(\rho_\s)_{|\h_\p}}(f^i_{T})=0.
        $$
        Using Lemma~\ref{koszulholomorphic} we can find $g^i_T\in\mathcal 
F_{|C_{diag}}\otimes \wedge \h_\p$ such that 
$(f^i_{T})=F^i_{T}+\partial_{-(\rho_\s)_{|\h_\p}}(g^i_T)$. Here, if 
$f^i_T=\sum_L f^i_{T,L}\otimes\w^L h\in\mathcal 
F_{|C_{diag}}\otimes\w \h_\p$,  $F^i_{T}$ is the holomorphic function 
on $C_{diag}$ defined, for $\l\in \h_\p^*$ and 
$\mu=\ov\mu+\ell\d_\aa\in\h_\aa^*\oplus\C\d_\aa$, by 
$$F^i_T(\l+\ov\mu+\ell\d+\rho_{\s\aa}-(\rho_\s)_{|\h_\aa},\mu)=f^i_{T,0}(\ov\mu+\ell\d+\rho_{\s\aa}-\rho_\s,\mu).       
$$
        
Hence, by Corollary~\ref{restrictionKoszul}, 
$$\sum \mathcal S(a_i)=\sum \mathcal S(p_i\otimes F^i_T\otimes 
\wedge^T h)+\ov d_p(\sum \mathcal
S(p_i\otimes g^i_T \otimes\wedge^T h)+w)+w'
$$ with $w,w'\in (\ov{\mathcal
        A^p_{inv}}/\ov{\mathcal A^{p+1}_{inv}})_{n-1}$. Let $\ov 
F^i_T\in\mathcal F_\aa$  be the
function defined by $\ov 
F^i_T(\mu)=f^i_{T,0}(\ov\mu+\ell\d+\rho_{\s\aa}-\rho_\s,\mu)$. 
Since $(\ov F^i_T)_{|C_{diag}}=F^i_T$ we obtain that $\sum \mathcal 
S(a_i)=\sum \mathcal
S(p_i\otimes \ov F^i_T\otimes \wedge^T h)+\ov d_p(\sum \mathcal 
S(p_i\otimes
g^i_T\otimes\wedge^T h)+w)+w'$. Writing explicitly
$\mathcal S(p_i\otimes \ov F^i_T\otimes \wedge^T h)$  as  a sum of 
terms of type $\ov y^I\tilde y^J\ov F^i_T\ov h^T\tilde x^H\ov
x^K+\ov{\Aa^{p+1}}$ we can apply Lemma~\ref{restriction} and write 
$\mathcal S(p_i\otimes \ov F^i_T\otimes
\wedge^T h)$ as a sum of terms $\ov y^I\tilde y^J_\aa \ov F^i_T\ov 
h^T\tilde x^H_\aa\ov x^K+\ov{\Aa^{p+1}}+y$ with
$y\in(\ov{\mathcal
        A^p_{inv}}/\ov{\mathcal A^{p+1}_{inv}})_{n-1}$. The final outcome is 
that
        $
        \sum \mathcal S(a_i)=a'+\ov d_p(z)+z'
        $
        with $a'\in\ov{\Aa^p_{inv}(\aa)}/\ov{\Aa^{p+1}_{inv}(\aa)}$, 
$z'\in(\ov{\mathcal
        A^p_{inv}}/\ov{\mathcal A^{p+1}_{inv}})_{n-1}$, hence 
        $$
        x=a'+\ov d_p(b'+z)+u+z'.
        $$
        Since $\ov d_p(x)=0$ we see that $\ov d_p(u+z')=0$, so, by the 
induction hypothesis, there is 
$a''\in\ov{\Aa^p_{inv}(\aa)}/\ov{\Aa^{p+1}_{inv}(\aa)}$ and 
$b'''\in(\ov{\mathcal
        A^p_{inv}}/\ov{\mathcal A^{p+1}_{inv}})$ such that $u+z'=a''+\ov 
d_p(b''')$, therefore, setting $a=a'+a''$ and $b=b'+z+b'''$ we have 
that $x=a+\ov d_p(b)$.
        
        We now prove by induction on $n$, that, 
if $a\in(\ov{\Aa^p_{inv}(\aa)}/\ov{\Aa^{p+1}_{inv}(\aa)})\cap Im\,\ov 
d_p$ and
$a\in(\ov{\Aa^p_{inv}}/\ov{\Aa^{p+1}_{inv}})_n$ then $a=0$. First of 
all observe that $\ov{\Aa(\aa)}$ is the closure of the
algebra generated by $(\tilde x)_{\aa r}, \ov x_r$, $f$ with 
$x\in\aa^{\ov r}$ and $f\in\mathcal F_\aa$. By
applying the PBW theorem to $L'(\aa,\s)\otimes Cl(L(\ov\aa,\ov\tau)$ 
we can write $a$ as a sum of terms of
type $\ov y^J\tilde y_\aa^Jf\ov h^T\tilde x_\aa^H\ov 
x^K+\ov{\Aa^{p+1}}$. By Lemma~\ref{restriction} we can
assume that $\deg(K+H)=p$ and $|J+H|\le n$.     Suppose now that $a=\ov 
d_p(u)$.  We will show that
$u\in(\ov{\Aa^p_{inv}}/\ov{\Aa^{p+1}_{inv}})_{n-2}$ thus, by the 
induction hypothesis, $a=0$.
        So assume that $u\in\sum \mathcal S(x_i)+u'$ with $x_i\in 
K^{p,0}_m\otimes\mathcal F_{|C_{diag}}\otimes\w\h_0$ and 
$u'\in(\ov{\Aa^p_{inv}}/\ov{\Aa^{p+1}_{inv}})_{m-1}$. Then, 
$a=\sum\mathcal S(\partial_{p,0}(x_i))+v+\ov d_p(u')$ with 
$v\in(\ov{\Aa^p_{inv}}/\ov{\Aa^{p+1}_{inv}})_{m}$. 
        
        Since, by Lemma~\ref{deDp}, $\partial_{p,0}(x_i)
        \in\sum_{\substack{r+s=p\\m>0}}K^r_\aa\otimes (K^s_\p\cap 
K^p_m)\otimes\mathcal F_{|C_{diag}}\otimes\w\h_0$, and, 
by Lemma ~\ref{restriction}, $a=\sum \mathcal S(a_i)+a'$ with $a_i\in 
(K^{p,0}_n)_\aa\otimes\mathcal
F_{|C_{diag}}\otimes\w\h_0$ and 
$a'\in(\ov{\Aa^p_{inv}}/\ov{\Aa^{p+1}_{inv}})_{n-1}$, by comparing 
terms,
we deduce that, if $m\ge n-1$, then  $\partial_{p,0}(x_i)=0$. 
        
        Since $\partial_{p,0}$ is exact except in degree $0$, we see that 
$x_i=a'_i+\partial_{p,0}(y_i)$ with $a'_i\in 
K^{p,0}_\aa\otimes\mathcal F_{|C_{diag}}\otimes\w\h_0$ and $y_i\in 
K^{p,0}_{m-1}\otimes\mathcal F_{|C_{diag}}\otimes\w\h_0$.
 Setting $y=\sum \mathcal S(y_i)$ we have that $\ov d_p(y)=\sum 
\mathcal S(x_i)-\sum \mathcal S(a'_i)+y'$
with $y'\in (\ov{\Aa^p_{inv}}/\ov{\Aa^{p+1}_{inv}})_{m-1}$, hence 
$u=\ov d_p(y)-y'+\sum \mathcal S(a'_i)+u'$.
Substituting $u$ with $u-\ov d_p(y)$ and $u'$ with $u'-y'$, we can 
assume that 
\begin{equation}\label{substituteu}u=\sum \mathcal S(a'_i)+u',
\end{equation}
 with $a'_i\in
(K^{p,0}_m)_\aa\otimes\mathcal F_{|C_{diag}}\otimes\w\h_0$.
        
Setting $a'_i=\sum_Tp_i\otimes f^i_T\otimes \w^T h^T$ and applying 
Corollary~\ref{restrictionKoszul}, we obtain that
$$\ov d_p(u)=\sum \mathcal S(p_i\otimes 
\partial_{-(\rho_\s)_{|\h_\p}}(f^i_T)\otimes \w^T h)+\ov d_p(v')+v''
$$
with
$v',v''\in(\ov{\Aa^p_{inv}}/\ov{\Aa^{p+1}_{inv}})_{m-1}$. Since $\bar 
d_p(u)=a$, comparing terms we deduce that 
$\partial_{-(\rho_\s)_{|\h_\p}}(f^i_T)\in
(\mathcal F_\aa)_{|C_{diag}}$, this means that it is a function on 
$C_{diag}$ that  does not depend from $\l\in\h_\p$. On the other 
hand, since $\partial_{-(\rho_\s)_{|\h_\p}}(f^i_T)$ is a boundary, it 
is zero when computed in $-(\rho_\s)_{|\h_\p}$ hence 
$\partial_{-(\rho_\s)_{|\h_\p}}(f^i_T)=0$. Arguing as in the first 
part of
the proof, we obtain that $\sum \mathcal S(a'_i)=a'+\ov d_p(z)+z'$ 
with
$a'\in(\ov{\Aa^p_{inv}(\aa)}/\ov{\Aa_{inv}^{p+1}(\aa)})$ and $z'\in
(\ov{\Aa^p_{inv}}/\ov{\Aa^{p+1}_{inv}})_{m-1}$. Since $\ov d_p(a')=0$ 
we find that  $\ov d_p(u-\ov
d_p(z)-a')=a$, hence, in light of \eqref{substituteu}, we can 
substitute $u$ with $z'+u'$ and  assume
$u\in(\ov{\Aa^p_{inv}}/\ov{\Aa^{p+1}_{inv}})_{m-1}$. Repeating the 
argument we can assume $m<n-1$ and
conclude the proof.
        \end{proof} 
We now come to the main reduction.
\begin{cor}\label{fr} The embedding  $\ov{\mathcal 
A_{inv}(\aa)}\to\ov{\mathcal A}^{inv}$  induces an isomorphism
$$
\ov{\mathcal A_{inv}(\aa)}\simeq H({\bf d}) .
$$
\end{cor}
\begin{proof}
Fix $x\in \ov{\mathcal A}^{inv}$ such that $\mathbf d(x)=0$. Then 
$\ov d_0(x+\ov{\Aa_{inv}^1})=0$,
 so, by \eqref{e01}, $x+\ov{\Aa_{inv}^1}=a^0+\ov{\Aa_{inv}^1}+\ov 
d_0(v^0+\ov{\Aa_{inv}^1})$. That is $x=a^0+\mathbf
d(v^0)+u^1$ with $a^0\in\ov{\Aa_{inv}(\aa)}$, $v^0\in\ov{\Aa}^{inv}$, 
and $u^1\in\ov{\Aa^1_{inv}}$.
Since $\mathbf d(a^0)=0$ we see that $\mathbf d(u^1)=0$, hence, 
arguing as above,
 we can write $u^1=a^1+\mathbf d(v^1)+u^2$,  with 
$a^1\in\ov{\Aa^1_{inv}(\aa)}$,
$v^1\in\ov{\Aa^1_{inv}}$, and $u^2\in\ov{\Aa^2_{inv}}$. Substituting 
we find
that
$x=a^0+a^1+\mathbf d(v^0+v^1)+u^2$.  An obvious induction shows then 
that for any $n$ we can find
$a^n,v^n,u^{n+1}$ such that  $a^n\in\ov{\Aa_{inv}^n(\aa)}$, 
$v^n\in\ov{\Aa^n_{inv}}$, and
$u^{n+1}\in\ov{\Aa^{n+1}_{inv}}$ and 
$$
x=\sum_{i=0}^na^i+\mathbf d(\sum_{i=0}^nv_i)+u^{n+1}.
$$
Since $\lim_ia^i=\lim_iv^i=\lim_iu^i=0$, setting $a=\sum_{i=0}^\infty 
a^i$ and $v=\sum_{i=0}^\infty v^i$, we have
$$
x=a+\mathbf d(v).
$$

Suppose now that $a\in\ov{\Aa_{inv}(\aa)}$ is such that $a\in 
Im\,\mathbf d$. Then $a+\ov{\Aa_{inv}^1}\in
Im\,\ov d_0$, hence, by \eqref{e01}, $a\in \ov{\Aa_{inv}^1}$.  
Repeating this argument we find that
$a\in\ov{\Aa^p_{inv}}$ for any $p$, thus $a=0$. 
\end{proof}
\subsection{The final step}\label{quarta}
We need to recall the main results of \cite{Kacpnas}.
Let
$\mathcal L\subset \widehat\h^*_0$ be the degeneracy locus of the 
Shapovalov form.
Let $\mathfrak F$  be the set of complex-valued functions
on $\widehat \h^*_0\setminus \mathcal L$. In \cite{Kacpnas}Ê it is 
proven that there exists a natural algebra
structure on  
$U_{\mathfrak F}=U(\widehat L(\g,\si))\otimes_{U(\widehat 
\h_0)}\mathfrak
F$ and on a  completion 
 $\widehat U_{\mathfrak F}$. Roughly speaking, $\widehat U_{\mathfrak 
F}$ consists of suitable series
$\sum_{I,J} y^I\varphi_{I,J}x^J$, with $y^I, x^J$ monomials in 
negative  and positive root vectors
with respect to a triangular decomposition of $\widehat L(\g,\sigma)$ 
and $\varphi_{I,J}\in\mathfrak F$. Recall that 
$C_\g$ denotes the Tits cone of $\widehat L(\g,\sigma)$ (see \eqref{tits}).
\begin{prop}\label{Kaccentro}Let $Z_{\mathfrak F}$ denote the center 
of $\widehat U_{\mathfrak F}$.
\begin{enumerate} 
\item Given $\varphi\in\mathfrak F$ there exists a unique element 
$z_\varphi=\sum_{I,J}y^I\varphi_{I,J}x^J\in
Z_{\mathfrak F}$ such that $\varphi_{0,0}=\varphi$.
\item The ``Harish-Chandra" map $H:Z_{\mathfrak F}\to \mathfrak F, 
H(z_\varphi)=\varphi$ 
is an algebra isomorphism.
\item\label{HC3} Let $\varphi\in\mathfrak F$ a function which can be 
extended to an holomorphic 
function on $-\rhat_\si+C_\g$.
Let $\varphi_{-\rhat_\si}$ be the holomorphic function on $C_\g$ 
defined by
$\varphi_{-\rhat_\si}(\l)=\varphi(\l-\rhat_\si)$. If 
$\varphi_{-\rhat_\si}$ is $\Wa$-invariant, then $z_\varphi$ can
be extended to an holomorphic element (i.e.
$z_\varphi=\sum_{I,J}y^I\varphi_{I,J}x^J$ with all
 the $\varphi_{I,J}$ holomorphic on $-\rhat_\si+C_\g$).
\end{enumerate}
\end{prop}
\vskip10pt

Let $\mathfrak
F_{hol}$ be the subalgebra of $\mathfrak F$ of the functions that can 
be extended to holomorphic functions
on $-\rhat_\s+C_\g$. Set
$U_{\mathfrak F_{hol}}=U(\widehat 
L(\g,\s))\otimes_{S(\ha_0)}\mathfrak F_{hol}\subset U_{\mathfrak F}$. 
Note that
the embedding 
$$
U(L'(\g,\s))\otimes\mathfrak F_{hol}\subset U(\widehat 
L(\g,\s))\otimes\mathfrak
F_{hol}
$$
induces an algebra isomorphism
\begin{equation}\label{prime=hat}
U( L'(\g,\s))\otimes_{S(\h'_0)}\mathfrak F_{hol}\simeq U_{\mathfrak
F_{hol}}.
\end{equation}

As in Section \ref{mainfiltration}, if we choose a basis $\{x^i\}$ of 
$\n'$ and a basis $\{y^i\}$ of $\n'_-$, 
and 
 $I$ is a multi-index, we use the notation $x^I$, $y^I$ to denote the 
corresponding monomials. Assume that $x^i$ and
$y^i$ are  root vectors of $\widehat L(\g,\s)$ and let $\gamma_i$ be 
the root corresponding to $x^i$.
Clearly we can assume that $-\gamma_i$ is the root corresponding to 
$y^i$. If $I=(i_1,i_2,\dots)$, let
$ht(I)=\sum_ji_jht(\gamma_j)$. Consider the subalgebra $\widehat 
U^0_{\mathfrak F_{hol}}$ of $\widehat U_{\mathfrak
F}$ whose elements are series
$\sum y^I\phi_{I,J}x^J$ with $\phi_{I,J}\in\mathfrak F_{hol}$ and 
$\sum i_h \gamma_h=\sum j_h\gamma_h $.

If $\phi \in\mathfrak F_{hol}$, let $f_\phi\in\mathcal F$ be defined 
by
\begin{equation}\label{ffi}
f_\phi(\l,\mu)=\phi(k\L_0+\l).\end{equation}
 Then the mapping $x\otimes \phi\mapsto\tilde xf_\phi$ defines an 
algebra map from
$U(L'(\g,\s))\otimes \mathfrak F_{hol}$ to $\ov{\Aa}$. Note that the 
map pushes down to define, thanks to the
isomorphism \eqref{prime=hat}, a map
\begin{equation}\label{fromhttodeg}
\Omega: U_{\mathfrak F_{hol}}\to \ov\Aa.
\end{equation}

We claim that $\Omega$ can be extended to  a map 
$$
\Omega:\widehat U_{\mathfrak F_{hol}}^0\to \ov\Aa.
$$
Indeed, recall from the beginning of  \S{} \ref{basicsetup} that $\a_i(\hat f_\aa)>0$ for 
all simple roots $\a_i$. Set $C=\inf
(\a_i(\hat f_\aa))$. Then, if $\l=\sum_in_i\a_i,\,n_i\in\nat$,
we have  $\sum_in_i\a_i(\hat f_\aa)\geq C\sum_i n_i$. Thus 
$\deg(\tilde x^J)\geq C\,ht(J)$, hence the series
$\sum \tilde y^I\phi_{I,J} \tilde x^J$ converges in 
$\overline{\mathcal A}$.\par
Recall that given a level $k$ highest weight module  $M$ for $\widehat 
L(\g,\sigma)$, then an action of 
$\widehat U_{\mathfrak F}$ on $M$ is defined in \cite{Kacpnas}. 
Moreover
$M\otimes F^{\ov\tau}(\ov\g)$ decomposes as a direct sum of weight 
spaces  both for $\ha_0$ and $\ha_\aa$. Since the
actions of
$\ha_0$ and $\ha_\aa$ commute, we can write
$$ 
M\otimes 
F^{\ov\tau}(\ov\g)=\bigoplus_{(\l,\mu)\in\ha^*_0\times\ha_\aa^*}(M\otimes 
F^{\ov\tau}(\ov\g))_{(\l,\mu)}.
$$
Then, if $(M\otimes
F^{\ov\tau}(\ov\g))_{(\l,\mu)}\ne\{0\}$, we have that
$\l=k\L_0+\ov \l+x\d$ and $\mu=\sum_S(k+g-g_S)\L^S_0+\ov \mu+y\d_\aa$ 
with $\ov\mu\in\h_\aa^*$. We can therefore define an action of 
$\mathcal F$ on $M\otimes
F^{\ov\tau}(\ov\g)$ by setting $f\cdot 
v=f(\ov\l+x\d,\ov\mu+y\d_\aa)v$ for $v\in (M\otimes
F^{\ov\tau}(\ov\g))_{(\l,\mu)}$. We can then extend this action to 
$\ov\Aa$. Observe that for any $v\in M$,
 $w\in F^{\ov\tau}(\ov\p)$ and $z\in \widehat U_{\mathfrak F_{hol}}^0$
\begin{equation}\label{omega}\Omega(z)(v\otimes w)=(zv)\otimes 
w.\end{equation}
\vskip10pt
We finally come to the main result of this section.

\begin{proof}[Proof of Theorem~\ref{Vogan1}] Let $f$ be as in the 
statement and let  $\phi$
 be the holomorphic function on $-\rhat_\s+C_\g$ defined by 
$\phi(\l)=f(\l+\rhat_\s)$.
Extend $\phi$ to an element of $\mathfrak F_{hol}$. By part \ref{HC3} 
of
Proposition~\ref{Kaccentro} we get the existence of  a central 
element $z_\phi$ of $\widehat
U_{\mathfrak F}$ such that $z_\phi\in\widehat U_{\mathfrak 
F_{hol}}^0$ and $z_\phi\cdot
v=f(\L+\rhat_\s)v$ for any $v\in M$.  Moreover,
$z_\phi=\phi+\sum_{I\ne0,J\ne0}y^I\phi_{I,J}x^J$. It follows that 
$\Omega(z_\phi)=f_\phi+x$
(cf. \eqref{ffi}) with $x\in\ov{\mathcal A^1_{inv}}$. Recall that we are viewing $\mathcal F_\aa$ as a subset of $\mathcal F$.  Let $\Phi\in\mathcal F_\aa$ be 
the function defined in
 $\mu=\ov\mu+x\d_\aa$, by 
 $$
 \Phi(\mu)=\phi(k\L_0+\ov\mu+x\d+\rho_{\s\aa}-\rho_\s).
 $$

By Lemma~\ref{koszulholomorphic}, there is $f'\in \mathcal 
F_{|C_{diag}}\otimes \w \h_\p$ such that 
$(f_\phi)_{|C_{diag}}=f_0+\partial_{-(\rho_\s)_{|\h_\p}}(f')$ where 
$f_0$ is the function on $C_{diag}$ defined, for $\l\in\h_\p^*$, 
$\mu\in\h_\aa^*$, by  
$$
f_0(\l+\mu+\rho_{\s\aa}-(\rho_\s)_{|\h_\aa}+x\d,\mu+x\d_\aa)=f_\phi(\mu+\rho_{\s\aa}-\rho_\s+x\d,\mu+x\d_\aa)=\Phi(\mu+x\d_\aa).
$$ 
Hence $f_0=\Phi_{|C_{diag}}$.

According to Lemma~\ref{dizero}, in $\ov{\Aa^0}/\ov{\Aa^1}$,  
$\mathcal S(\partial_{-(\rho_\s)_{|\h_\p}}(f'))=\bar d_0(\mathcal 
S(f'))$ hence we can write $f_\phi+\ov{\Aa^1_{inv}}=f_0+\ov{\Aa^1}+ 
\bar d_0(\mathcal S(f'))$. 
Thus 
$f_\phi+\ov{\Aa^1_{inv}}=\Phi+\ov{\Aa^1_{inv}}+\ov d_0(w)$ 
with $w\in\ov{\Aa^0_{inv}}/\ov{\Aa^1_{inv}}$.
 Since $\Omega(z_\phi)+\ov{\Aa^1_{inv}}=f_\phi+\ov{\Aa^1_{inv}}$, we 
have that  $\Omega(z_\phi)=\Phi+\mathbf d(z)+z_1$ with 
$z_1\in\ov{\Aa^1_{inv}}$. Since $z_\phi$ is central, 
$[(G_{\g,\aa})_0,\Omega(z_\phi)]=0$, hence $\mathbf 
d(\Omega(z_\phi)-\Phi-\mathbf d(z))=0$.
It follows that
$\mathbf d(z_1)=0$. According to Corollary~\ref{fr}, there is 
$a\in\ov{\Aa^1_{inv}(\aa)}$ such that $z_1=a+\mathbf d(y)$,
 hence,
setting $u=y+z$
$$
\Omega(z_\phi)=\Phi+a+\mathbf d(u).
$$
Suppose now that a highest weight $\widehat L(\sa,\si)$-module $Y$ of
highest weight
$\mu=\sum_S(k+g-g_S)\L^S_0+\ov \mu+y\d_\aa$
 occurs in the Dirac cohomology $H((G_{\g,\sa})_0^{N'})$ of
$N'=M\otimes  F^{\ov\tau}(\ov\p)$ and that $v$ is an highest vector 
for $Y$. Since
$a\in\ov{\Aa^1_{inv}(\aa)}$, by Lemma~\ref{restriction}, $a\cdot 
v=0$, hence, by \eqref{omega} 
$$
\Omega(z_\phi) v=f(\L+\rhat_\s)v=\Phi(\mu)v.
$$
hence 
$f(\L+\rhat_\si)=\Phi(\mu)=\phi(k\L_0+\ov\mu+y\d+\rho_{\s\aa}-\rho_\s)$. 
Finally observe that 
 $\phi(k\L_0+\ov\mu+y\d+\rho_{\s\aa}-\rho_\s)=f((k+g)\L_0+\ov
\mu+y\d+\rho_{\s\aa})=f((\varphi_\aa^*)^{-1}(\mu+\rhat_{\aa\s}))=f_{|\ha_\aa^*}(\mu+\rhat_{\aa\s})$.
\end{proof}

\section{Proofs of technical results}
\subsection{Evaluation of 
$[{G_\g}_\l G_\g]$} \label{Gquadro}Let $G_\g\in V^{k+g,1}(R^{super})$ 
be the affine Dirac operator
defined in \eqref{G}; write for shortness $G$ instead of $G_\g$. We 
want to
calculate $[G_\l G]$  by using expression 
\eqref{diractilde}.   We assume that $\g$ is reductive. We let 
$\g=\sum_S \g_S$ be the eigenspace
decomposition of $\g$ with respect to the Casimir operator $Cas$ of 
$\g$ and let $2g_S$ be the
eigenvalue relative to $\g_S$.

We proceed in steps. First of all we fix  an orthonormal  basis 
$\{x_i^S\}$ of $\g_S$ and choose
$\{x_i\}=\cup_S \{x_i^S\}$ as  orthonormal basis of $\g$.   If $a\in 
\g_S$,\begin{equation}\label{f1}[\tilde 
a_\la G]=\sum_i :\widetilde{[a,x_i]}\overline x_i:+\la:(k+g- 
g_S)\overline 
a:,\end{equation}and\begin{equation}\label{f2}[G_\la\tilde 
a]=-\sum_i  :\widetilde{[a,x_i]}\overline
x_i:+\la:(k+g-g_S)\overline a:+:(k+g- g_S)T(\overline 
a):,\end{equation}\begin{align}\notag[G_\la :\tilde 
a\overline a:]&=-\sum_i :\widetilde{[a,x_i]}\overline x_i\overline 
a:+:\tilde a 
a:\\\label{f3}&+:(k+g-g_S)T(\overline a)\overline  
a:+\frac{\la^2}{2}(k+g-g_S)(a,a).\end{align} Formula
\eqref{f1} is  checked directly whereas \eqref{f2} follows from 
\eqref{f1} by the  sesquilinearity of the
$\l$-bracket. Finally, \eqref{f3}   follows from 
\eqref{f2}  using the Wick formula 
\eqref{Wick}.
\begin{equation}\label{f4}[\overline a_\la:\tilde b\overline  
b:]=(a,b)\tilde
b\end{equation}\begin{equation}\label{f5}[:\tilde b\overline  
b:_\la\overline a]=(a,b)\tilde 
b\end{equation}\begin{equation}\label{f6}[:\tilde a\overline 
a:_\la:\overline  b\overline c:]=(a,b):\tilde
a\overline c:-(a,c):\tilde a\overline  
b:\end{equation}\begin{align}\label{f7}[:\tilde a\overline
a:_\la:\overline  b\overline c\overline d:]&=(a,b):\tilde a\overline 
c\overline d:- (a,c):\overline b\tilde
a\overline d:+(a,d):\overline b\overline c\tilde  
a:\\&\label{f7'}=(a,b):\tilde a\overline c\overline
d:-(a,c):\tilde a\overline  b\overline d:+(a,d):\tilde a\overline 
b\overline c:\end{align} In checking  all the
formulas one uses the Wick formula \eqref{Wick}. Moreover, 
\eqref{f4}    follows from\eqref{tre},  
\eqref{f5}   follows from  \eqref{f4}   by  sesquilinearity. Formula 
\eqref{f6}   is checked similarly. Finally
one proves 
\eqref{f7}  combining   \eqref{f5}     and\eqref{f6}.   The equality 
between   
\eqref{f7}   and    \eqref{f7'} follows from \ref{qa} and 
relations $:\tilde  a\overline b:=:\overline b\tilde a:$
and $:\overline a\tilde b\overline  c:=:\tilde b\overline a\overline 
c:$. From \eqref{f7'} it follows also that 
\begin{equation}\label{f8}
\sum_{h,k}[:\tilde a\overline a:_\la:\overline{[  x_h,x_k]}\overline 
x_h\overline x_k:]=-3\sum_h:\tilde 
a\overline{[a,x_h]}\overline x_h:
\end{equation} In turn \eqref{f3}  and   
\eqref{f8}  imply
\begin{align}\notag
\sum_i[G_\la :\tilde x_i\overline  x_i:]&=\sum_{i,j} 
:\widetilde{[x_i,x_j]}\overline x_i\overline 
x_j:+\sum_{i,S}:(k+g-g_S)T(\overline x^S_i)\overline  
x^S_i:\\&+\sum_{i}:\tilde x_i\overline
x_i:+\sum_S\frac{\la^2}{2}(k+g- g_S)\dim\g_S.
\end{align}
\begin{equation}
\sum_{i,h,k}[:\tilde  x_i\overline x_i:_\la :\overline{ 
[x_h,x_k]}\overline x_h\overline x_k:]=- 3\sum_{i,k}
:\tilde x_i\overline{[x_i,x_k]}\overline x_k:.
\end{equation} We  start now computing 
$\sum_{i,j,h,k}[:\overline{[x_i,x_j]}\overline  x_i\overline
x_j:_\la:\overline{[x_h,x_k]}\overline x_h\overline x_k:]$.  We  have
\begin{equation}
\sum_{i,j}[\overline a_\la :\overline{[x_i,x_j]}\overline  
x_i\overline
x_j:]=-3\sum_i:\overline{[a,x_i]}\overline  x_i:.
\end{equation}

\begin{equation}
\sum_{i,j}[:\overline{[x_i,x_j]}\overline  x_i\overline 
x_j:_\la\overline a
]=-3\sum_i:\overline{[a,x_i]}\overline  x_i:.
\end{equation}

\begin{equation}
\sum_i[\overline  a_\la:\overline{[b,x_i]}\overline  
x_i:]=2\overline{[a,b]}.
\end{equation}

\begin{equation}\label{thetalambdabar}
\sum_i[:\overline{[a,x_ i]}\overline x_i:_\la\overline  
b]=2\overline{[a,b]}.
\end{equation} Hence
\begin{align}\notag &\sum_{i,j,h,k}[:
\overline{[x_i,x_j]}\overline x_i\overline  
x_j:_\la:\overline{[x_h,x_k]}\overline x_h\overline x_k:]\\
\notag &=- 3\sum_{i,h,k}[::\overline{[[x_h,x_k],x_i]}\overline 
x_i:\overline x_h\overline 
x_k:-\sum_{i,j,h,k}:\overline{[x_h,x_k]}[:\overline{[x_i,x_j]}\overline  
x_i\overline x_j:_\la :\overline
x_h\overline x_k:]\\
\notag&- 3\sum_{i,h,k}\int_0^\la[:\overline{[[x_h,x_k],x_i]}\overline 
x_i:_\mu: 
\overline x_h\overline x_k:]d\mu\\
\notag&=- 3\sum_{i,h,k}::\overline{[[x_h,x_k],x_i]}\overline 
x_i:\overline x_h\overline 
x_k:+3\sum_{i,h,k}:\overline{[x_h,x_k]}:\overline{[x_h,x_i]}\overline  
x_i:\overline x_k:\\
\notag&- 3\sum_{i,h,k}:\overline{[x_h,x_k]}\overline 
x_h\overline{[x_k,x_i]}\overline  x_i:\\ &-
\sum_{i,j,h,k}\int_0^\la:\overline{[x_h,x_k]}[[:\overline{[x_i,x_j]}\overline  
x_i\overline x_j:_\la\overline
x_h]_\mu\overline x_k]:d\mu\\
\notag &- 6\sum_{h,k}\int_0^\la[:\overline{[[x_h,x_k],x_h]}\overline 
x_k:d\mu-
6\sum_{h,k}\int_0^\la[:\overline x_h\overline{[[x_h,x_k],x_k]}:d\mu\\
\label{?}  
&-3\sum_{i,h,k}\int_0^\la\int_0^\mu[[:\overline{[[x_h,x_k],x_i]}\overline  
x_i:_\mu \overline
x_h]_\nu\overline x_k:]d\nu d\mu
\end{align} Using the  relation 
$\sum_{h}[x_h,x_k]\cdot[x_h,x_j]=\sum_{h}[[x_j,x_h],x_k]\cdot x_h$ 
for 
any bilinear product $\cdot$, we can rewrite \eqref{?} as 
\begin{align*} &- 3\sum_{i,h,k}::\overline{[[x_h,x_k],x_i]}\overline 
x_i:\overline x_h\overline 
x_k:+3\sum_{i,h,k}:\overline{[[x_i,x_h],x_k]}:\overline x_h\overline  
x_i:\overline
x_k:\\&-3\sum_{i,h,k}:\overline{[x_h,[x_i,x_k]]}\overline  
x_h\overline x_k\overline 
x_i:+6\sum_{i,h,k}\int_0^\la:\overline{[x_h,x_k]}\overline{[x_h,x_k]}:d\mu\\ 
&
-6\sum_{h,k}\int_0^\la:\overline{[x_k,x_h]}\overline{[x_k,x_h]}:d\mu-
6\sum_{h,k}\int_0^\la:\overline{[x_k,x_h]}\overline{[x_h,x_k]}:d\mu\\&+12\frac{\l 
^2}{2}\sum_S g_S\dim\g_S
\end{align*} Since $:\overline a\overline a:=0$ (by 
\eqref{qc}) we get
\begin{align*} &- 3\sum_{i,h,k}::\overline{[[x_h,x_k],x_i]}\overline 
x_i:\overline x_h\overline 
x_k:+3\sum_{i,h,k}:\overline{[[x_i,x_h],x_k]}:\overline x_h\overline  
x_i:\overline
x_k:\\&-3\sum_{i,h,k}:\overline{[x_h,[x_i,x_k]]}\overline  
x_h\overline x_k\overline
x_i:+12\frac{\la^2}{2}\sum_S  g_Sdim\g_S.
\end{align*}

By \eqref{qc} and \eqref{qa} we  
find$$\sum_{i,h,k}::\overline{[[x_h,x_k],x_i]}\overline x_i:\overline 
x_h\overline x_k:=\sum_{i,h,k}:\overline{[[x_h,x_k],x_i]}\overline 
x_i\overline  x_h\overline
x_k:\!\!-8\sum_{S,i}g_S:T(\overline x^S_i)\overline  x^S_i:
$$ hence
$$
\sum_{i,h,k}:\overline{[[x_i,x_h],x_k]}:\overline  x_h\overline 
x_i:\overline x_k:=-
\sum_{i,h,k}:\overline{[[x_i,x_h],x_k]}:\overline x_i\overline 
x_h:\overline  x_k:,
$$ and  therefore
\begin{align*}&\sum_{i,h,k}:\overline{[[x_i,x_h],x_k]}:\overline  
x_h\overline x_i:\overline x_k:=\\&-
\sum_{i,h,k}:\overline{[[x_i,x_h],x_k]}\overline x_i\overline 
x_h\overline 
x_k:+4\sum_{S,i}g_S:T(\overline x^S_i)\overline x^S_i:
\end{align*} Upon  substituting we find
\begin{align*} &- 3\sum_{i,h,k}:\overline{[[x_h,x_k],x_i]}\overline 
x_i\overline x_h\overline 
x_k:-3\sum_{i,h,k}:\overline{[[x_i,x_h],x_k]}\overline x_i\overline  
x_h\overline
x_k:\\&-3\sum_{i,h,k}:\overline{[x_h,[x_i,x_k]]}\overline  
x_h\overline x_k\overline
x_i:+36\sum_{i,S}g_S:T(\overline x^S_i)\overline  
x^S_i:+12\frac{\la^2}{2}\sum_Sg_Sdim\g_S
\end{align*} By the Jacobi identity we  find
$$
\sum_{i,j,h,k}[:\overline{[x_i,x_j]}\overline x_i\overline  
x_j:_\la:\overline{[x_h,x_k]}\overline x_h\overline
x_k:]\!=\!  36\sum_{S,i}g_S:T(\overline x^S_i)\overline  
x^S_i\!:\!+\la^2\sum_S\frac{2g_S}{3}dim\g_S.
$$ Hence
\begin{align*} [G_\la  G]&=\sum_i[G_\la :\tilde x_i\overline 
x_i:]-\frac{1}{6}\sum_{i,j,h}[G_\la 
:\overline{[x_i,x_j]}\overline x_i\overline x_j:]\\&=\sum_i[G_\la 
:\tilde  x_i\overline
x_i:]-\frac{1}{6}\sum_{i,h,k}[:\tilde x_i\overline x_i:_\la  
:\overline{[x_h,x_k]}\overline x_h\overline 
x_k:]\\&+\frac{1}{36}\sum_{i,j,h,k}[:\overline{[x_i,x_j]}\overline 
x_i\overline 
x_j:_\la:\overline{[x_h,x_k]}\overline x_h\overline 
x_k:],\end{align*}
which in turn gives
\begin{align*}
[G_\la  G]&=\sum_{i,j}  :\widetilde{[x_i,x_j]}\overline
x_i\overline x_j:+\sum_{S,i}:(k+g- g_S)T(\overline x^S_i)\overline 
x^S_i:\\&+\sum_i:\tilde 
x_ix_i:+\frac{\la^2}{2}\sum_S(k+g-g_S)\dim\g_S\\&+\frac{1}{2}\sum_{i,k}  
:\tilde
x_i\overline{[x_i,x_k]}\overline  x_k:\\&+\sum_{S,i}g_S:T(\overline 
x^S_i)\overline 
x^S_i:+\frac{\la^2}{2}\sum_S\frac{g_S}{3}\dim\g_S
\end{align*} or
\begin{align*} [G_\la G]&=\sum_{i,j} :\widetilde{[x_i,x_j]}\overline 
x_i\overline  x_j:+(k+g)\sum_i:T(\overline
x_i)\overline x_i:\\&+\sum_i:\tilde x_i\tilde  x_i:+\sum_{i,k} 
:\tilde x_i\overline{[x_i,x_k]}\overline 
x_k:\\&+\frac{\la^2}{2}\sum_S(k+g-
\frac{2g_S}{3})\dim\g_S
\end{align*} Since
$$
\sum_{i,j}  :\widetilde{[x_i,x_j]}\overline x_i\overline 
x_j:=-\sum_{i,j}  :\tilde{x_i}\overline
{[x_i,x_j]}\overline x_j:
$$ we finally  get
\begin{prop}\label{GrelGrel}
\begin{equation}[G_\la G]=\sum_i:\tilde x_i\tilde  
x_i:+(k+g)\sum_i:T(\overline x_i)\overline
x_i:+\frac{\la^2}{2}\sum_S(k+g-
\frac{2g_S}{3})\dim\g_S.
\end{equation}
\end{prop}
\subsection{$(G_{\g,\aa})_0$ is selfadjoint}
Assume that $\g$ is semisimple. If $\{\a_i\}$ is the set of simple 
roots of $\widehat L(\g,\si)$, let $\h_\R$  
be the real span of $\{\frac{h_{\ov\a_i}}{(\ov 
\a_i,\ov\a_i)}\}\subset\h_0$. Set $\mathfrak t_0=i\h_\R$ and
let $\k$ be a compact form of $\g$. By Exercise 8 of Chapter VI of 
\cite{Helgason1}, we can assume that $\k$
is $\sigma$-invariant and that $\mathfrak t_0\subset \k$. Denote by 
$conj$  the conjugation in $\g$ w.r.t.
$\k$: $conj(x+iy)=x-iy$, $x,y\in\k$, and let $\omega_0$ be the 
antilinear antiautomorphism of $\g$ defined by
$\omega_0(x)=-conj(x)$.  
 Extend $\omega_0$ to an antiautomorphism of $\widehat
L(\g,\si)$ by 
$$\omega_0(x_{(r)})=(\omega_0(x))_{(-r)},\quad \omega_0(K)=K,\quad 
\omega_0(d)=d.$$
If $\a_i$ is a simple root of  $\widehat
L(\g,\si)$, with $\a_i=s_i\d+\ov\a_i$, choose $x_i\in 
\g^{s_i}\cap\g_{\ov\a_i}$ in  such a way that 
$[x_i,\omega_0(x_i)]=\frac{2h_{\ov\a_i}}{(\ov\a_i,\ov\a_i)}$. Set 
$e_i=t^{s_i}\otimes x_i,\,f_i=\omega_0(e_i)$, so that 
$[e_i,f_i]=\a_i^\vee=\frac{2h_{\ov\a_i}}{(\ov 
\a_i,\ov\a_i)}+\frac{2s_i}{(\ov \a_i,\ov\a_i)}K$. With this choice of
Chevalley generators for  $\widehat
L(\g,\si)$, $\omega_0$ is the unique antilinear antiautomorphism of  
$\widehat
L(\g,\si)$ which exchanges the $e_i$ with the $f_i$ and which is the 
identity when restricted to $\ha_\R=span_\R(\a_i^\vee)$. Recall that 
a $\widehat L(\g,\si)$-module $M$ is called unitarizable if there is 
a positive definite hermitian
form $H(\cdot,\cdot)$ on $M$ such that $H(x\cdot 
v,w)=H(v,\omega_0(x)\cdot w)$. Theorem~11.7 of
\cite{Kac} asserts that $L(\L)$ is unitarizable if and only if  $\L$ 
is dominant integral.

Recall that we are assuming that $g\in\R^+$ hence  $(\cdot,\cdot)$ is 
negative definite on $\k$. Thus
we can define  a positive definite hermitian form $H(\cdot,\cdot)$ 
on  $\g$ by the formula 
\begin{equation} H(x,y)=(x,\omega_0(y)).
\label{contra}
\end{equation}   Define  a positive definite hermitian form $\ov 
H(\cdot ,\cdot )$ on  $ L(\ov\g,\tau)$ by 
$$\ov H(x_{(r)},y_{(s)})=\d_{r,s}H(x,y).$$  We can extend this form 
on the Clifford module
$F^{\ov\tau}(\ov \g)$ by setting
$$
H(x_{1}\cdots x_{n}\cdot 1,y_{1}\cdots y_{m}\cdot 1 
)=\d_{n,m}\det(H(x_i,y_j))
$$
where it is again positive definite. A standard calculation
shows that, if $\ov x\in\ov\g$ and $v,w\in F^{\ov\tau}(\ov\g)$
$$\ov H( x_{(r-\half)}v,w)=\ov H( v,(\omega_0(x))_{(-r-\half)}w).$$

Suppose that $M$ is a unitarizable $\widehat L(\g,\si)$-module. On 
$N=M\otimes F^{\ov \tau}(\g)$ we can
consider the form $\{\cdot ,\cdot \}=H(\cdot ,\cdot )\otimes\ov 
H(\cdot,\cdot)$. Straightforward
calculations show that for  $v,w\in M\otimes F^{\ov\tau}(\ov\g)$
\begin{align*}\{\sum_i(:\tilde x_i\ov x_i:)^{\ov\tau}_{(\half)}v,w\}&=
\{v,\sum_i(:\widetilde{\omega_0(x_i)}\ov{\omega_0(x_i)}:)^{\ov\tau}_{(\half)}\cdot 
w\}\\ &=
\{v, \sum_i(:\tilde x_i\ov x_i:)^{\ov\tau}_{(\half)}\cdot w\},
\end{align*} and similarly for the cubic term of $(G_\g)^N_0$. In 
particular, if $\L$ is dominant integral, then 
$(G_\g)^N_0$ is self-adjoint with respect to a Hermitian positive 
definite form on $L(\L)\otimes F^{\ov\tau}(\ov\g)$ . For the general 
case of $G_{\g,\sa}$ it
suffices to apply the argument to $(G_\sa)^N_0$ acting on 
$N=(L(\L)\otimes F^{\ov\tau}(\ov\p))\otimes
F^{\ov\tau}(\ov\sa)$ and to note that 
$L(\L)\otimes F^{\ov\tau}(\ov\p)$ is unitarizable since, as shown in \S~\ref{applications}, 
$F^{\ov\tau}(\ov\p)$ is the restriction to $\widehat L(\sa,\si)$ of a 
unitarizable representation of
$\widehat L(so(\p),Ad(\si))$. Summing up
\begin{prop}\label{unitarity}  The action of $(G_{\g,\sa})^{N'}_0$ on 
${N'}=L(\L)\otimes F^{\ov\tau}(\ov\p)$ is
self-adjoint with respect to the positive definite Hermitian form  
$\{\ ,\ \}$. Consequently,
 $Ker(G_{\g,\aa})^{N'}_0=Ker((G_{\g,\aa})^{N'}_0)^2$.\end{prop}
\vskip5pt 
\subsection{Proof of Lemma \ref{faithful}}\label{lemma85}
 If $\L\in-\rhat_\si+C_\g$ is a weight of level $k$, let $M(\L)$ be 
the Verma module of
highest weight $\L$.
  Let $(\cdot,\cdot)_\L$ be the Shapovalov form on $M(\L)$ and 
$\omega$ the Chevalley involution of
$\widehat L(\g,\si)$.  Choose $\rhat\in\ha_0^*$ such that 
$\rhat(\a^\vee_i)=1$ for all $i$.  We assume
furthermore that $2(\L+\rhat+\gamma,\a)\ne n(\a,\a)$ for all $n>0$, 
$\gamma$ in the root lattice of
$\widehat L(\g,\si)$, and $\a\in\Dap$. With this assumption we have 
that $(\cdot,\cdot)_\L$ is
nondegenerate  (see e.g. \cite[Lemma 1]{Kacpnas}).  It follows that, 
for each
$\mu=\sum n_i\a_i$ with $n_i\in\nat$ there is a basis $\{\bf y^i_\mu\}$ of
$U(\n'_-)_{-\mu}$ such that
$\{{\bf y^i_\mu} v_\L\}$ is an orthonormal  basis of 
$M(\L)_{\L-\mu}$. 
Here 
and in the following we use
boldface letters to make clear that we are dealing with polynomials 
and multi-indexes.
Set $\bf x^i_\mu=\omega(\bf
y^i_\mu)$. Note that ${\bf x^i_\mu \bf y^j_\nu}\cdot v_\L\in 
\bigoplus\limits_{ht(\eta)= ht(\nu)-ht(\mu)}M(\L)_{\L-\eta}$. In 
particular, if $ht(\mu)=ht(\nu)$, we have
that
${\bf x^i_\mu \bf y^j_\nu}\cdot v_\L\in \C v_\L$. Since $({\bf 
x^i_\mu \bf y^j_\nu}\cdot v_\L,v_\L)_\L=({\bf
y^j_\nu}\cdot v_\L,{\bf y^i_\mu}\cdot v_\L)_\L=\d_{\nu,\mu}\d_{i,j}$, 
we see that 
\begin{equation}\label{commM} {\bf x^i_\mu \bf y^j_\nu}\cdot
v_\L=\d_{\nu,\mu}\d_{i,j}v_\L.\end{equation}

Fix a basis $\{h_i\}$ of $\h_0$ such that $(h_i,h_{n-j+1})=\d_{ij}$ 
and set
 $$\h_0^+=span(h_i\mid i\le \lfloor 
\frac{n}{2}\rfloor),\quad\h_0^-=span(h_i\mid i\ge \lceil
\frac{n}{2}\rceil+1).$$ Note that $\h_0=\h_0^+\oplus\h_0^-\oplus E$ 
with $dim E\le 1$.
 Set $\ov\g^+=\ov\h_0^+\oplus\ov\n'$ and 
$\ov\g^-=\ov\h_0^-\oplus\ov\n'_-$.  If $s\in\nat$, set
$Cl(\ov\g^\pm)_s=(\ov\g^\pm)^s\subset Cl(\ov\g)$. Note that 
$Cl(\ov\g^+)_sCl(\ov\g^-)_r\cdot 1\in
\oplus_{t\le r-s}Cl(\ov\g^-)_t\cdot 1$.  Moreover it is easy to 
construct a basis $\{{\u}^{\i s}\}$ of
$Cl(\ov\g^+)_s$ and a basis 
$\{\bv^{{\bf j}s}\}$ of $Cl(\ov\g^-)_s$ such that
\begin{equation}\label{commF}
\u^{\i s}\bv^{\j s}\cdot1=\d_{ij}.\end{equation}

Suppose that $u\in \ov{\mathcal W^k\otimes 1}$ is such that  
$$
 u\cdot v=0\quad\text{for any $v\in M(\L)\otimes F^{\ov 
\tau}(\ov\p)$}.
$$ Write $u=\lim \pi(u_n)$. Fix $p>0$.  Since $\pi(u_n)$ is a Cauchy 
sequence we can find
$M$ such that, for any $n\ge M$, $\pi(u_M-u_n)\in\mathcal W^kF^p$. 
Suppose that $\deg({\bf
x^j}_\nu\otimes {\bf u^{s}}^t)<p$. Choose $n\ge M$ big enough so that 
$u_n\cdot({\bf y^j_\nu} \cdot v_\L\otimes {\bf v^{s}}^t\cdot 1)=0$. 
Then $u_M\cdot ({\bf y^j_\nu}\cdot v_\L\otimes
{\bf v^{s}}^t\cdot 1)=(u_n-(u_n-u_M))\cdot ({\bf y^j_\nu }\cdot 
v_\L\otimes {\bf v^{s}}^t\cdot 1)=0$. 
Hence, for any $\j,{\bf s},t,\nu$ such that $\deg({\bf
x^j}_\nu\otimes {\bf u^{s}}^t)<p$,
\begin{equation}\label{perp}
u_M\cdot ({\bf y^j_\nu} \cdot v_\L\otimes
{\bf v^{s}}^t\cdot 1)=0.
\end{equation}

Using the PBW theorem we can write
\begin{equation*}\label{sum}u_M=\sum_{j,\nu,t,s,\epsilon}c_{\j,\nu,{\bf 
t},s,\epsilon} w^\epsilon ({\bf
x^j_\nu}\otimes
\u^{{\bf t}s}),\end{equation*} with
$c_{\j,\nu,{\bf t},s,\epsilon}\in U(\n'_-\oplus\h'_0)\otimes 
Cl(\bar\g^-)$. Here $ w^\epsilon$ occurs only
if $\dim E=1$ and, in such a case, $E=\C w$ and $\epsilon\in\{0,1\}$.

 We now show by induction on $q$ that  $ c_{\j,\nu,{\bf 
t},s,\epsilon}w^\epsilon\in \mathcal W(K-k)$ for
all
${\bf j},\nu,{\bf t},s,\epsilon$ such that $ht(\nu)+s=q$ and 
$\deg({\bf x^j_\nu}\otimes
\u^{{\bf t}s})<p$. This concludes the proof since it implies that 
$\pi(u_M)\in\mathcal W^kF^p$, hence
$\pi(u_n)\in\mathcal W^kF^p$ for $n\ge M$. 

If $ht(\nu)=s=0$,  then  $\pi(u_M)\cdot v_\L\otimes 1=u_M\cdot 
v_\L\otimes 1=0$. Hence
$\sum_{\epsilon}c_{1,0,1,0,\epsilon} w^\epsilon\cdot v_\L\otimes 
1=0$. Write explicitly
$c_{1,0,1,0,\epsilon}=\sum_{\i,\mu,{\bf t},n}({\bf y^\i_\mu}\otimes
\bv^{{\bf t} n})p_{\i,\mu,{\bf t},n}$, with $p_{\i,\mu,{\bf t},n}\in 
U(\h'_0)$.  Since $U(\n'_-)\otimes 
Cl(\ov\g^-\oplus N)$ acts freely on
$v_\L\otimes 1$ (and $U(\h'_0)$ commutes with $w^\epsilon$), we see 
that 
$p_{\i,\mu,{\bf t},n}(k\L_0+\ov\L)=0$ for any $\i,\mu,{\bf t},n$. 
Since $\ov\L$ can be chosen in  a
dense subset of $\h_0$, we see that $p_{\i,\mu,{\bf 
t},n}(k\L_0+\ov\L)=0$ for any $\ov\L$, thus
$K-k$ divides $p_{\i,\mu,{\bf t},n}$.

Fix now $\nu_0$ and $s_0$ with $ht(\nu_0)+s_0=q>0$ and such that 
$\deg({\bf x^l_{\nu_0}}\otimes
\u^{{\bf m} s_0})<p$. Applying \eqref{perp} we get
that 
$$
\sum_{\substack{\j,\nu,s,{\bf t},\epsilon\\ ht(\nu)+s\ge 
q}}c_{j,\nu,s,t,\epsilon}w^\epsilon
  ({\bf x^j_\nu}\otimes  \u^{{\bf t}s})\cdot ({\bf y^l_{\nu_0}}\cdot 
v_\L\otimes \bv^{{\bf m} s_0}\cdot1)=0.
$$ Indeed,  by the induction hypothesis, in the above sum appear only 
coefficients with $ht(\nu)+s\ge q$.
Since ${\bf x^j_\nu y^l_{\nu_0}}\cdot v_\L=0$ if $ht(\nu)>ht(\nu_0)$ 
and $\u^{{\bf t}s}\bv^{{\bf m} s_0}\cdot1=0$
if
$s>s_0$, we can write
$$
\sum_{\substack{\j,\nu,s,{\bf t},\epsilon\\ ht(\nu)=ht(\nu_0),s=s_0}} 
c_{\j,\nu,s,{\bf t},\epsilon}w^\epsilon ({\bf
x^j_\nu}\otimes \u^{{\bf t}s})\cdot ({\bf y^l_{\nu_0}}\cdot 
v_\L\otimes \bv^{{\bf m} s_0}\cdot 1)=0.
$$ Using \eqref{commM}, \eqref{commF}, we get
$$
\sum_{\epsilon}c_{{\bf l},\nu_0,{\bf m},s_0,\epsilon}w^\epsilon \cdot 
(v_\L\otimes1)=0.
$$ Arguing as above we deduce that $c_{{\bf l},\nu_0,{\bf 
m},s_0,\epsilon}\in\mathcal W(K-k)$.

\subsection{Proof of Lemma~\ref{fieldsinW}}\label{lemma86}
We may assume that $a=:T^{i_1}(x^1)\cdots T^{i_h}(x^h):$ with 
$x^i\in\{\tilde
x,\ov x\mid x\in\g_\a^{\ov r},r\in \R,\a\in\h_0^*\}$.  Set 
$N(a,r)=cr+\deg(a)$ and
$U(a,r)=\ov{\mathcal W^kF^{N(a,r)}}$. We will prove by induction on 
$h$ that
  for each $r$ there is $a_r\in U(a,r)$ such that $a_r\cdot 
v=a_r^N\cdot v$.
  
  If $h=1$, we set 
\begin{equation}\label{baseind} T^{i}(\tilde 
x)_r=(-1)^ii!\binom{r+i}{i}t^r\otimes x,\quad T^{i}(\ov
x)_r=(-1)^ii!\binom{r+i-\half}{i}t^{r-\half}\otimes \ov x.
\end{equation}
  
   If $h>1$ we can assume that $a=:T^{i_1}(x^1)b:$ with 
$b=:T^{i_2}(x^2)\cdots T^{i_h}(x^h):$. \par\noindent
  By Wick formula, we have that 
$T^{i_1}(x^1)_{(j)}b=\sum:T^{j_1}(y^1)\cdots T^{j_v}(y^v):$ with 
$v<h$,
so, by the induction hypothesis, we can define
\begin{equation}\label{equa} (T^{i_1}(x^1)_{(j)}b)_r=\sum 
:T^{j_1}(y^1)\cdots T^{j_v}(y^v):_r.
\end{equation}

Choose  $m\in\ov r_{x^1}-\D_{x^1}$ and set
\begin{align*}  a_r&=-\sum_{j\ge
1}\binom{m+\D_{x^1}+i_1-1}{j}(T^{i_1}(x^1)_{(j-1)}b)_r\\&+\sum_{j=0}^\infty
T^{i_1}(x^1)_{m-j-1}b_{r-m+j+1}+p(x^1,b)b_{r-m-j}T^{i_1}(x^1)_{m+j}.
\end{align*}

By \eqref{baseind}
$$ b_{r-m-j}T^{i_1}(x^1)_{m+j}\in U(x^1,m+j)
$$ and, by the induction hypothesis,
$$ T^{i_1}(x^1)_{m-j-1}b_{r-m+j+1}\in U(b,r-m+j+1),
$$ hence the series 
$$
\sum_{j=0}^\infty 
T^{i_1}(x^1)_{m-j-1}b_{r-m+j+1}+p(x^1,b)b_{r-m-j}T^{i_1}(x^1)_{m+j}
$$ converges. Therefore the definition of $a_r$ makes sense and,  by 
\eqref{tensore}, $a_r\cdot
v=a^N_r\cdot v$ for any $N$, $v\in N$.

To conclude the induction step we need to check that $a_r\in U(a,r)$. 
Recall that we wrote 
$T^{i_1}(x^1)_{(j)}b=\sum_{v<h}:T^{j_1}(y^1)\cdots T^{j_v}(y^v):$, 
hence we can assume that
$\deg(:T^{j_1}(y^1)\cdots T^{j_v}(y^v):)=\deg(a)$. By the induction 
hypothesis it follows that 
$(T^{i_1}(x^1)_{(j)}b)_r\in U(a,r)$.

It follows from \eqref{rep1} and Lemma~\ref{faithful}, that 
\begin{equation}\label{2:36} [T^{i_1}(x^1)_s,b_{r-s}]=\sum_{j\ge
0}\binom{s+\D_{x^1}+i_1-1}{j}(T^{i_1}(x^1)_{(j)}b)_r.
\end{equation} The induction hypothesis and \eqref{2:36} now imply 
that
$$T^{i_1}(x^1)_sb_{r-s}\in U(a,r),\quad b_{r-s}T^{i_1}(x^1)_s\in 
U(a,r)$$ for, if
$\deg(b)+c(r-s)<\deg(a)+cr$, then $\deg(x^1)+cs>0$ and
$$ 
T^{i_1}(x^1)_sb_{r-s}=b_{r-s}T^{i_1}(x^1)_s+p(x^1,b)[T^{i_1}(x^1)_s,b_{r-s}]
$$ and, if $\deg(x^1)+cs<\deg(a)+cr$, then
$$ 
b_{r-s}T^{i_1}(x^1)_s=T^{i_1}(x^1)_sb_{r-s}+p(x^1,b)[b_{r-s},T^{i_1}(x^1)_s].
$$ This concludes the induction step.

Finally \eqref{degrepres} follows easily by induction on $h$ and the 
explicit formula for $a_r$.
\subsection{Proof of Lemma~\ref{koszulholomorphic}}\label{newKoszul} 
We may clearly assume that $h_0=0$. Set for shortness
$\partial=\partial_{0}$.
 Note that $\partial =\sum_{i=1}^nz_i\frac{\partial}{\partial \xi_i}$.
Define $h=\sum_{i=1}^n\xi_i\frac{\partial}{\partial z_i}$, and note 
that both $\partial$ and $h$ are 
odd derivations. Hence we have that 
$$h \partial + \partial h = 
\sum_{i,j=1}^n[z_i\frac{\partial}{\partial \xi_i}, 
\xi_j\frac{\partial}{\partial 
z_j}]=\sum_{i=1}^nz_i\frac{\partial}{\partial z_i}+ 
\sum_{i=1}^n\xi_i\frac{\partial}{\partial \xi_i}.$$
Denote by $E=\sum_{i=1}^nz_i\frac{\partial}{\partial z_i}$ the Euler 
operator and take a 
$p$-cycle $f\otimes u\in R_p$ with $p>0$; we  have 
$(h \partial + \partial h)(f\otimes u)=(E(f)+pf)\otimes u$.
Define $\varphi=\sum 
\frac{f_{i_1,\ldots,i_n}}{i_1+\cdots+i_n+p}z_1^{i_1}\cdots z_n^{i_n}$.
 Then $\varphi$ is holomorphic in any polydisk $|z_i|\leq r_i$, since
 $$\sum  \frac{|f_{i_1,\ldots,i_n}|}{i_1+\cdots+i_n+p} 
r_1^{i_1}\cdots r_n^{i_n} \leq
 \sum |f_{i_1,\ldots,i_n} |r_1^{i_1}\cdots r_n^{i_n}.$$
 Set $M=\sum_{i=1}^n
z_i\mathcal E(z_1,\ldots,z_n)+\sum_{p>0}\mathcal E(z_1,\ldots,z_n)\otimes 
\w^p(\xi_1,\ldots,\xi_n)$.  Then 
$E+p\,I:M\to M$
is bijective, and $M$ is preserved by the Koszul operator.
 Moreover $\partial$ commutes with $\partial h+h\partial$. Therefore  
$\varphi\otimes u$ is a cycle if
$f\otimes u$ is, 
 hence 
 $$f=(E+p\,I)(\varphi\otimes u)=(h \partial + \partial h) 
(\varphi\otimes u)=
 \partial (h (\varphi\otimes u))$$
 as required.

\subsection{Proof of Lemma~\ref{restriction}}\label{lemmarestrizione}
We first show that  if $\deg(I)=j_p$ and $|I|=n$ then 
        \begin{equation}\label{aaaa}\tilde x^I_\aa=\tilde x^I+u\end{equation}
        with $u\in \ov{\Aa^p}$ and 
$u+\ov{\Aa^{p+1}}\in\left(\ov{\Aa^p}/\ov{\Aa^{p+1}}\right)_{n-1}$.\par
We will show by induction on $|I|$ that $\tilde x^I_\aa=\tilde x^I+u$ 
with $u$ a series $\sum_{i=p}^\infty u_i$
where $u_i=\sum c^i_{J,T,H,K}\ov y^J\ov h^T\tilde x^H\ov x^K$ and 
$c^i_{J,T,H,K}\in\C$, $\deg(H+K)=j_i$, $|H|<n$. 

If $|I|=0$ the result is obvious. Suppose $|I|=n>0$. Then 
$$
\tilde x^I_\aa=\tilde x^{i_1}_\aa\tilde x^{I_0}_\aa=\tilde 
x^{i_1}\tilde x^{I_0}_\aa+\theta( x^{i_1})\tilde x^{I_0}_\aa,
$$
with $|I_0|=n-1$. If $\deg(I_0)=j_{p'}$, applying the induction 
hypothesis, we have
$$
\tilde x^I_\aa=\tilde x^{I}+\tilde x^{i_1}u'+\theta( x^{i_1})\tilde 
x^{I_0}+\theta( x^{i_1})u',
$$
where $u'$ is a series $\sum_{i=p'}^\infty u'_i$ with $u'_i= \sum 
c'{}^i_{J,T,H,K}\ov y^J\ov h^T\tilde x^H\ov x^K$, $deg(H+K)=j_i$, and 
$|H|<n-1$.

Let $u=\tilde x^{h_1}u'+\theta(x^{i_1})\tilde 
x^{I_0}+\theta(x^{i_1})u'$. Clearly $\tilde x^{i_1}u'_i=\sum 
d{}^i_{J,T,H,K}\ov y^J\ov h^T\tilde x^H\ov x^K$ with 
$\deg(H+K)=j_i+\deg(x^{i_1})\ge p'+\deg(x^{i_1})=j_p$ and, by PBW 
theorem, $|H|<n$. Since $\deg(\theta(x^{i_1}))=\deg(\tilde x^{i_1})$, 
by the explicit expression \eqref{theta(x)} for $\theta(x^{i_1})$, we 
see that $\theta(x^{i_1})=\sum_{i=i_0}^\infty \theta_i$ where 
$j_{i_0}=\deg(x^{i_1})$ and
$\theta_i=\sum k^i_{J,T,K}\ov y^J\ov h^T \ov x^K$ with 
$k^i_{J,T,K}\in\C$, $\deg(K)=j_{i}$. 

Hence 
$\theta( x^{i_1})\tilde x^{I_0}=\sum_{i=i_0}^\infty \theta_i\tilde 
x^{I_0}$ and 
$\theta_i\tilde x^{I_0}=\sum k^i_{J,T,K}\ov y^J\ov h^T\tilde x^{I_0} 
\ov x^K$. Moreover $\deg(I_0+K)=\deg(I_0+i)\ge 
\deg(I_0)+\deg(x^{i_1})=j_p$ and $|I_0|<n$. Finally 
$\theta(x^{i_1})u'=\sum_{i=p'}^\infty\theta(x^{i_1})u'_i$. Note that 
$\theta(x^{i_1})u'_i=\sum_{t=i}^\infty u''_{it}$ with $u''_{it}=\sum 
c^{it}_{J,T,H,K}\ov y^I\ov h^T\tilde x^H\ov x^K$ with $\deg(H+K)=j_t$ 
and $|H|<n$. Since $\deg((\tilde x)_{\aa r})=\deg(\tilde x_r)$ we 
have that $\deg(\tilde x_\aa^{I_0})=\deg(\tilde x^{I_0})$, hence 
$\deg(u')=\deg(I_0)$. It follows that 
$\deg(\theta(x^{i_1})u'_i)=j_p$. 
This implies, by Remark \ref{forrestricting}, that 
$\sum_{i=p'}^tu''_{it}=0$ for $t<p$. This concludes the induction
step and proves \eqref{aaaa}.\par
We prove the statement of the Lemma by induction on $|J|$. If $|J|=0$ 
the statement is easily obtained from
the previous step. If
$|J|>0$ then 
$$
\ov y^I\tilde y^J_\aa f\ov h^T\tilde x^H_\aa\ov x^K=\ov y^I\tilde 
y^{j_1}\tilde y^{J_0}_\aa f\ov h^T\tilde x^H_\aa\ov x^K+\ov y^I 
\theta( y^{j_1})y^{J_0}_\aa f\ov h^T\tilde x^H_\aa\ov x^K.
$$
Applying the induction hypothesis we can write 
$$
\ov y^I\tilde y^J_\aa f\ov h^T\tilde x^H_\aa\ov x^K=\ov y^I\tilde 
y^{J} f\ov h^T\tilde x^H\ov x^K+\ov y^I\tilde y^{j_1}u'+\ov y^I 
\theta( y^{j_1})\tilde y^{J_0}f\ov h^T\tilde x^H\ov x^K+\ov y^I 
\theta( y^{j_1})u'.
$$
with $u'\in\ov{\Aa^p}$ and 
$u'+\ov{\Aa^{p+1}}\in\left(\ov{\Aa^p}/\ov{\Aa^{p+1}}\right)_{n-2}$.

 Set $u=\ov y^I\tilde y^{j_1}u'+\ov y^I \theta( y^{j_1})\tilde 
y^{J_0}f\ov h^T\tilde x^H\ov x^K+\ov y^I \theta( y^{j_1})u'$. Clearly 
$u\in\ov{\Aa^p}$.
 Write $u'=\sum \ov y^{I'}\tilde y^{J'} f_{I',J',T',H',K'}\ov 
h^{T'}\tilde x^{H'}\ov x^{K'}+u''$ with $\deg(H'+K')=j_p$, 
$|J'+H'|<n-1$, and $u''\in\ov{\Aa^{p+1}}$.
 
 By PBW theorem 
\begin{align*}
         \ov y^I\tilde y^{j_1}u'+\ov{\Aa^{p+1}}&=\sum \ov y^Iy^{I'}\tilde 
y^{j_1}\tilde y^{J'} f_{I',J',T',H',K'}\ov h^{T'}\tilde x^{H'}\ov 
x^{K'}+\ov{\Aa^{p+1}}\\
         &=\sum \ov y^{I''}\tilde y^{J''} f_{I'',J'',T',H',K'}\ov 
h^{T'}\tilde x^{H'}\ov x^{K'}+\ov{\Aa^{p+1}}
\end{align*}
with $|J''+H'|<n$. 

Writing, as above, $\theta(y^{j_1})=\sum_{i=0}^\infty \theta_i$ with 
$\theta_i=\sum k^i_{J,T,K}\ov y^J\ov h^T \ov x^K$ with 
$k^i_{J,T,K}\in\C$, $\deg(K)=j_i$ and using the commutation relations 
\eqref{prodotto} and \eqref{prodottoa}, we see that
$$
\ov y^I \theta( y^{j_1})\tilde y^{J_0}f\ov h^T\tilde x^H\ov 
x^K+\ov{\Aa^{p+1}}=\ov y^I \theta_0\tilde y^{J_0}f\ov h^T\tilde 
x^H\ov x^K+\ov{\Aa^{p+1}}
$$
hence
$$
\ov y^I \theta( y^{j_1})\tilde y^{J_0}f\ov h^T\tilde x^H\ov 
x^K+\ov{\Aa^{p+1}}=\sum \ov y^{I'}\tilde y^{J_0} f_{I',T'}\ov 
h^{T'}\tilde x^{H}\ov x^{K}+\ov{\Aa^{p+1}}
$$
and, clearly, $|J_0+H|<n$.

The same argument shows that 
$$
\ov y^I \theta( y^{j_1})u'+\ov{\Aa^{p+1}}=\sum \ov y^{I'}\tilde 
y^{J'}f_{I',J',T',H',K'}\ov h^{T'}\tilde x^{H'}\ov 
x^{K'}+\ov{\Aa^{p+1}}$$
with $|J'+H'|<n$ as desired.

\providecommand{\bysame}{\leavevmode\hbox to3em{\hrulefill}\thinspace}
\providecommand{\MR}{\relax\ifhmode\unskip\space\fi MR }
\providecommand{\MRhref}[2]{%
  \href{http://www.ams.org/mathscinet-getitem?mr=#1}{#2}
}
\providecommand{\href}[2]{#2}

\footnotesize{

\noindent{\bf V.K.}: Department of Mathematics, Rm 2-178, MIT, 77 
Mass. Ave, Cambridge, MA 02139;\\
{\tt kac@math.mit.edu}

\noindent{\bf P.MF.}: Politecnico di Milano, Polo regionale di Como, 
Via Valleggio 11, 22100 Como,
ITALY;\\ {\tt pierluigi.moseneder@polimi.it}

\noindent{\bf P.P.}: Dipartimento di Matematica, Universit\`a di Roma 
``La Sapienza", P.le A. Moro 2,
00185, Roma , ITALY;\\ {\tt papi@mat.uniroma1.it} }


\begin{thebibliography}{10}

\bibitem{KacD}
A.~De~Sole and V.~G. Kac, \emph{Finite vs affine {$W-$}algebras}, Japanese
  Journal of Mathematics \textbf{1} (2006), 137--261.

\bibitem{GNO}
P.~Goddard, W.~Nahm, and D.~Olive, \emph{Symmetric spaces, {S}ugawara's energy
  momentum tensor in two dimensions and free fermions}, Phys. Lett. B
  \textbf{160} (1985), 111--116.

\bibitem{GKRS}
B.~Gross, B.~Kostant, P.~Ramond, and S.~Sternberg, \emph{The {W}eyl character
  formula, the half-spin representations, and equal rank subgroups}, Proc.
  Natl. Acad. Sci. USA \textbf{95} (1998), no.~15, 8441--8442 (electronic).
  \MR{MR1639139 (99f:17007)}

\bibitem{Helgason1}
S.~Helgason, \emph{Differential geometry, {Lie} groups, and symmetric spaces},
  Academic Press, 1978.

\bibitem{HP}
J.-S. Huang and P.~Pand{\v{z}}i{\'c}, \emph{Dirac cohomology, unitary
  representations and a proof of a conjecture of {V}ogan}, J. Amer. Math. Soc.
  \textbf{15} (2002), no.~1, 185--202 (electronic).

\bibitem{KacW}
V.~G. Kac and M.~Wakimoto, \emph{Modular and conformal invariance constraints
  in representation theory of affine algebras}, Adv. in Math. \textbf{70}
  (1988), 156--236.

\bibitem{Kacpnas}
V.~G. Kac, \emph{Laplace operators of infinite dimensional {L}ie algebras and
  theta functions}, Proc. Nat. Acad. Sci. U.S.A. \textbf{81} (1984), no.~2,
  Phys. Sci., 645--647.

\bibitem{Kac}
V.~G. Kac, \emph{Infinite dimensional {L}ie algebras}, third ed., Cambridge
  University Press, Cambridge, 1990.

\bibitem{KacV}
\bysame, \emph{Vertex algebras for beginners}, second ed., University Lecture
  Series, vol.~10, American Mathematical Society, Providence, RI, 1998.

\bibitem{KacT}
V.~G. Kac and I.~T. Todorov, \emph{Superconformal current algebras and their
  unitary representations}, Comm. Math. Phys. \textbf{102} (1985), no.~2,
  337--347.

\bibitem{KW}
V.~G. Kac and M.~Wakimoto, \emph{Quantum reduction in the twisted case},
  Infinite dimensional algebras and quantum integrable systems, Progr. Math.,
  vol. 237, Birkh\"auser, Basel, 2005, pp.~89--131.

\bibitem{Kold}
B.~Kostant, \emph{A cubic {D}irac operator and the emergence of {E}uler number
  multiplets of representations for equal rank subgroups}, Duke Math. J.
  \textbf{100} (1999), no.~3, 447--501.

\bibitem{K3}
\bysame, \emph{Dirac cohomology for the cubic {D}irac operator}, Studies in
  memory of Issai Schur (Chevaleret/Rehovot, 2000), Progr. Math., vol. 210,
  Birkh\"auser Boston, 2003, pp.~69--93.

\bibitem{Kumar}
S.~Kumar, \emph{Kac-{M}oody groups, their flag varieties and representation
  theory}, Progress in Mathematics, vol. 204, Birkh\"auser Boston Inc., Boston,
  MA, 2002.

\bibitem{land}
G.~D. Landweber, \emph{Multiplets of representations and {K}ostant's {D}irac
  operator for equal rank loop groups}, Duke Math. J. \textbf{110} (2001),
  no.~1, 121--160.

\bibitem{V}
D.~A. Vogan, Jr., \emph{Dirac operator and unitary representations}, 3 lectures
  at {MIT Lie group} seminar, Fall of 1997.

\end{thebibliography}
\end{document}